\documentclass[a4paper,10pt]{article}
\usepackage[utf8]{inputenc}

\usepackage[T1]{fontenc}
\usepackage[english]{babel}

\usepackage[svgnames]{xcolor}
\usepackage[margin=1.2in]{geometry}

\usepackage{graphicx}
\usepackage{wrapfig}
\usepackage{subcaption}

\usepackage{enumitem}

\usepackage[colorlinks=true,linkcolor=blue,citecolor=FireBrick,urlcolor=red,pagebackref=true]{hyperref}

\usepackage{amsmath}
\usepackage{amssymb}
\usepackage{mathrsfs}
\usepackage{amsthm}
\usepackage{bm}
\usepackage{bbm}    
\usepackage[scr=boondox]{mathalpha}

\usepackage[ruled]{algorithm2e}
\usepackage{tikz}       

\usepackage{macros}

\title{Cutoff for mixtures of permuted Markov chains: \\ general case}
\author{Bastien Dubail \thanks{Correspondence to be sent to: \href{mailto:bastdub@kth.se}{bastdub@kth.se}} \\ KTH, Stockholm, Sweden }
\date{}

\begin{document}

\maketitle

\begin{abstract}
	We investigate the mixing properties of a finite Markov chain in random environment defined as a mixture of a deterministic chain and a chain whose state space has been permuted uniformly at random. This work is the counterpart of a companion paper where we focused on a reversible model, which allowed for a few simplifications in the proof. We consider here the general case. Under mild assumptions on the base Markov chains, we prove that with high probability the resulting chain exhibits the cutoff phenomenon at entropic time $\log n / h$,  $h$ being some constant related to the entropy of the chain, when the chain is started from a typical state. However contrary to the reversible case uniform cutoff at entropic time does not hold, as we provide an example where the worst-case mixing time has at least polylogarithmic order. We also provide a polylogarithmic upper bound on the worst-case mixing time, which in fact plays a crucial role in deriving the main result for typical states. Incidentally, our proof gives a clear picture of when to expect uniform cutoff at entropic time: it appears as the consequence of a uniform transience property for the covering Markov chain used throughout the proof, which lies on an infinite state space and projects back onto the initial chain. 
\end{abstract}

\section{Introduction}

\subsection{Cutoff for mixtures of permuted Markov Chains}

This paper establishes a cutoff phenomenon at entropic time for a model of Markov chain in random environment. Let $n \geq 1$ be an integer and $P_1, P_2$ two $n \times n$ stochastic matrices. This paper is basically concerned with the mixing time of the Markov chain with transition matrix
\begin{equation*}
	p_1 \, P_1 + p_2 \, S P_2 S^{-1}
\end{equation*}
where $p_1 \in (0,1), p_2 = 1 - p_1$ and $S$ is a uniform $n \times n$ permutation matrix. This model is inspired by the work \cite{hermon2020universality} of Hermon, Sousi and Sly, who proved the cutoff phenomenon at entropic time for the simple random walk on a sequence of deterministic graphs to which is added a random uniform matching of the vertices. The model above aims to look at more general Markov chains, possibly non reversible, obtained by superposing two deterministic ones with a layer of randomness in between taking the form of a permutation of the states. Our model will in fact allow the probabilities $p_1, p_2$ to depend on the current state of the chain and on the permutation $S$, which is necessary to recover the case of the simple random walk on the superposition of graphs if these are not regular. For this case however, and more generally for reversible chains, there is a more adapted model tailored to preserve reversibility, corresponding to a specific choice of $p_1, p_2$. This model is studied in the companion paper \cite{cutoff_reversible}. We focus here on the most general case.

Let us first introduce general notations that will be used throughout the paper.
Given two measures $\mu, \nu$ on a countable set $S$, their total variation distance is defined as 
\begin{equation*}
	\TV{\mu - \nu} := \sup_{A \subseteq S} \abs{\mu(A) - \nu (A)} = \frac{1}{2} \sum_{x \in S} \abs{\mu(x) - \nu(x)}.
\end{equation*}
If $P$ is the transition kernel of a positive recurrent, irreducible and aperiodic Markov chain on $S$, it admits a unique invariant measure $\pi$. In that case, given a starting vertex $x \in S$ and $\e \in (0,1)$, the mixing time is defined as 
\begin{equation*}
	\tmix(x, \e) := \inf \{ t \geq 0: \TV{P^{t}(x, \cdot) - \pi} < \e \}. 
\end{equation*}
and the worst-case mixing time is 
\begin{equation*}
	\tmix(\e) := \sup_{x \in S} \tmix(x,\e).
\end{equation*}
If the chain is not irreducible or aperiodic, we consider the mixing time to be infinite. Given two functions $f,g: \bN \rightarrow \bR$, $f(n)= O(g(n))$ if there exists a constant $C > 0$ such that $\abs{f(n)} \leq C g(n)$ for all $n$, $f=o(g)$ if $f(n) / g(n) \xrightarrow[n \rightarrow \infty]{} 0$ and $f = \Theta(g)$ if $f=O(g)$ and $g = O(f)$. An event $A=A(n)$ is said to occur with high probability, if $\bP(A) = 1 - o(1)$. Given an integer $n \geq 1$, we write $[n] := \{ 1, \ldots, n \}$. 

\begin{theorem}\label{thm:cutoff}
	Let $P_1, P_2$ be two $n \times n$ stochastic matrices, $p_1, p_2 \in M_{n}([0,1])$ be $n \times n$ matrices with entries in $[0,1]$ such that $p_1 + p_2 = 1$ entry-wise, $\sigma$ a uniform permutation of $n$ elements and consider the Markov chain with transition probabilities
	\begin{equation}\label{eq:model0}
		\forall x,y \in [n]: \quad \scP_0(x,y) := p_1(x, \sigma(x)) \, P_1(x,y) + p_2(x, \sigma(x)) \, P_2(\sigma(x), \sigma(y)).
	\end{equation}
	Suppose 
	\begin{enumerate}[label=(H\arabic*)]
		\item \label{hyp:bdd_delta} for each $i \in \{ 1, 2 \}$ the non-zero entries of $P_1, P_2$ are $\Theta(1)$ and $\max_{y \in V} \sum_{x \in V} P_i(x,y) = O(1)$.
		\item \label{hyp:bdd_p} the entries of $p_1, p_2$ are $\Theta(1)$. 
		\item \label{hyp:cc3} For all $x \in [n]$ and $i=1,2$:
		\begin{equation*}
			\abs{\{y \in V: \exists k \geq 0: P_{i}^{k}(x,y) > 0 \} } \geq \left\{ \begin{array}{l l} 3 & \text{if $i=1$}  \\ 2 & \text{if $i=2$} \end{array} \right., 
		\end{equation*}
		\item \label{hyp:backtracking} There exists a constant $l \geq 1$ independent of $n$ such that
		\begin{equation*}
			\abs{ \{ x \in V \ | \ \forall y \in V: P(x,y) > 0 \Leftrightarrow (I+P)^{l}(y,x) > 0 \} } = n - o(n) 
		\end{equation*}
	\end{enumerate}
	Then the chain defined by \eqref{eq:model0} is irreducible, aperiodic and there exists $h = \Theta(1)$ for which the following holds. For all $\e \in (0,1)$, there exists $C(\e)$ such that for all $x \in [n]$ with high probability,
	\begin{align*}
		\tmix(x,\e) &\leq \frac{\log n}{h} + C(\e) \sqrt{\log n},\\
		\tmix(x,1-\e) &\geq \frac{\log n}{h} - C(\e) \sqrt{\log n}.
	\end{align*}
\end{theorem}

Note that in Theorem \ref{thm:cutoff}, the starting state of the chain is fixed before generating the random environment. This implicitely makes the state $x$ \emph{typical}: from the result for a fixed $x$ one can for instance replace $x$ with a state chosen uniformly at random. However it cannot be chosen arbitrarily after the permutation $\sigma$ is fixed. Consequently our result does not establish uniform cutoff in the sense that with high probability $\tmix(x,\e) = \log n / h + o(\log n)$ simultaneously for all $x \in [n]$. This makes a strong difference with the reversible case studied in \cite{cutoff_reversible} in which we were able to prove uniform cutoff at entropic time. It turns out that with the level of generality considered here, the lack of a uniform result is not an artifact of the proof but is characteristic of the model. In whole generality, there is no uniform cutoff at entropic time, as proved by the following result.

\begin{theorem}\label{thm:counterex}
	For all $a > 1$, there exist transition matrices $P_1, P_2$ satisfying Assumptions \ref{hyp:bdd_delta} - \ref{hyp:backtracking} such that with high probability, the Markov chain $\scP_0 = 1/2 ( P_1 + S P_2 S^{-1})$ has worst-case mixing time $\max_{x} \tmix(x, 1/2) \geq (\log n)^{a}$.
\end{theorem}

Our model is not the first instance which exhibits a non-uniform cutoff phenomenon. It was already observed for instance in the case of simple random walks on the giant component of Erdös-Renyi graphs: with high probability these contain segments of length of order $\log n$, resulting in the worst-case mixing time being of order $(\log n)^{2}$ \cite{benjamini2014mixing,fountoulakis2008evolution}, while the typical mixing time is $O(\log n)$ \cite{berestycki2018random}.

In our model, the potential slowdown of the chain is a consequence of irreversibility: having too much bias on the transition probabilities, the environment can with positive probability contain \emph{trapping neighbourhoods}, from which the chain escapes in a time that scales exponentially with the size of the neighbourhood. The previous proposition is based on an example where these neighbourhoods are of size $O(\log \log n)$, hence the polylogarithmic order in the result. Our last result establishes a matching upper bound of polylogarithmic order on the worst-case mixing time. 

\begin{theorem}\label{thm:worstcase}
	In the setting of Theorem \ref{thm:cutoff} under the same assumptions there exists $a \geq 1$ such that for all $\e \in (0,1)$, with high probability
	\begin{equation*}
		\tmix(\e) \leq (\log n)^{a}.
	\end{equation*}
\end{theorem}

\begin{remark}
	Let us comment on the assumptions made above. Hypotheses \ref{hyp:bdd_delta} and \ref{hyp:bdd_p} bound uniformly in $n$ the transition probabilities and the number of ingoing and outgoing transitions at every state, while \ref{hyp:cc3} imposes a minimal size on the set of states reachable by a given state. These are not restrictive: it is easy to find examples where these assumptions are not satisifed and the mixing time scales differently as $\Theta(\log n)$, without cutoff. The most restrictive assumption is thus \ref{hyp:backtracking}. It can be interpreted as a constraint on the degree of irreversibility of the chains $P_1, P_2$, which for instance prevents them from having too many transient states. It is not clear whether it is really needed for our result but from the proof it does not seem superfluous either, as having more  irreversibility may amplify the trapping phenomenon mentionned above. Nonetheless, irreversibility is still permitted to a considerable extent and appears as a genuine limitation rather than a technical one. 
\end{remark}

\subsection{Related works}

This paper is the direct extension of \cite{cutoff_reversible} where we analyse a specific case of the Markov chain \eqref{eq:model0} tailored to preserve reversibility, for which cutoff at entropic time can be proved uniformly over the starting state. We tried to propose there a thorough overview of the most recent works about cutoff in random and generic settings, so we will remain brief in this paper and refer the reader to the "Related works" section of \cite{cutoff_reversible} for a more comprehensive list of references. 

As we mentionned already, our model is directly inspired from that of Hermon, Sly and Sousi \cite{hermon2020universality} who proved cutoff for the simple random on a graph with an added uniform matching, as well as related models: \cite{baran2023phase} in which Baran, Hermon, Šarković and Sousi prove a phase transition when weights are added on the random matching, and \cite{hermon2022cutoff} in which Hermon, Šarković and Sousi study the simple random walk on graphs with community structures. The absence of uniform cutoff was already observed for the simple random walk on the giant component of a supercritical Erdös-Renyi graph, studied in the work of Berestycki, Lubetzky, Peres and Sly \cite{berestycki2018random} from which we re-used a few arguments. Generally, our proof stems from the strategy initially used by Ben Hamou and Salez \cite{benhamou2017cutoff}, then by Bordenave, Caputo and Salez \cite{bordenave2018random, bordenave2019cutoff} for non-backtracking random walks.

\subsection{Proof outline}\label{subsec:outline}

We first mention that the results presented in this paper were obtained in the same time as those of \cite{cutoff_reversible} for the reversible model. As noted there, the arguments barely used the reversibility assumption, which was made mainly to avoid the technical difficulties that eventually make cutoff only valid for typical starting states. Thus the key arguments of the proof already appear in \cite{cutoff_reversible} and we will in this paper mostly emphasize on what needs to be adapted. 

\paragraph{Dealing with an unknown invariant measure:} first the absence of a reversibility assumption prevents any knowledge of the invariant measure. To establish mixing, we follow the strategy of \cite{bordenave2018random,bordenave2019cutoff}, already used in the reversible case, of proving convergence towards an approximate invariant measure $\hat{\pi}$. If $\TV{\scP_0^{t}(x, \cdot) - \hat{\pi}} \leq \e$ holds uniformly in $x$ for a given time $t$, then the invariance property implies that a true invariant measure $\pi$ satisfies $\TV{\pi - \hat{\pi}} \leq \e$, establishing in particular the uniqueness of the invariant measure. The additional issue here is that Theorem \ref{thm:cutoff} is valid only for typical starting vertices, so we will have to prove that mixing towards $\hat{\pi}$ occurs both at $t \simeq \log n / h$ for a typical starting state and at larger times for an arbitrary starting state, namely $t = (\log n)^{a}$. Note that since $\hat{\pi}$ is not a priori an invariant measure, there is no monotonicity in the distance to equilibrium that would maintain the chain distributed as $\hat{\pi}$ once it has become close to it. We note that Proposition \ref{prop:nice_approx} below gives an expression for the proxy measure $\hat{\pi}$ we use in our proofs. 

\paragraph{The entropic method: } the proof is based on the entropic method, which consists in two main arguments: first an entropic concentration is shown to occur at the entropic time $\log n / h$, which is shown to imply cutoff in the second part of the proof. These arguments have come so far with two different variations, depending on whether reversibility is supposed or not. Since we do not assume reversibility, we use the strategy that was originally designed for non-backtracking random walks \cite{benhamou2017cutoff,bordenave2018random,bordenave2019cutoff}. First we prove a quenched entropic concentration property for trajectories, namely we prove a statement of the form $\scP_0(X_0 \cdots X_t) \simeq e^{-t h + O(\sqrt{t})}$, where we write $\scP(X_0 \cdots X_t) = \prod_{i=0}^{t-1} \scP_0(X_i, X_{i+1})$. This statement is given here as a rough heuristic, which would be correct in the non-backtracking case. In general, what we prove in practice is rather concentration of the probability to follow the \emph{loop-erased trace} of $X$, so we are back at studying non-backtracking trajectories. This concentration phenomenon is proved by a coupling with another Markov chain lying on an infinite random state space, which we call a quasi-tree in reference to \cite{hermon2020universality}, as this object is similar to theirs. Basically, it is designed to be a stationary approximation of the universal cover of the finite chain, in the same way Erdös-Renyi random graphs can be approximated by infinite Galton-Watson trees for instance. 

\paragraph{Analysis on the quasi-tree:} all the differences with the reversible case are contained in this part of the argument. At the basis of the analysis for reversible chains in \cite{hermon2020universality,hermon2022cutoff, baran2023phase,cutoff_reversible} is the property of uniform transience, which asserts that at any vertex, the probability to escape to infinity along the neighbouring "branch" of the quasi-tree, conditional on the environment, is lower bounded away from $0$ uniformly in $n$. Such a uniform lower bound may not be achieved in the general model of Theorem \ref{thm:cutoff} and is ultimately the reason why cutoff cannot occur for a uniform starting state. For instance our counter example to uniform cutoff (Theorem \ref{thm:counterex}) stems directly from a quasi-tree which is close to being recurrent, resulting in the chain to be potentially stuck in trapping neighbourhoods. The lack of uniform transience is bypassed by proving that the probability to escape at infinity along is in expectation lower bounded uniformly in $n$. This will be proved thanks to a comparison with a reversible chain, which is the main reason why Hypothesis \ref{hyp:backtracking} comes into play. The lower bound on the expected escape probabilities will then imply that such a lower bound is also achieved when restricting to typical points. Very roughly, the difference with the reversible case will thus consist in arguing that the trajectories of the chain remain up to the mixing time on typical parts of the quasi-tree and consequently avoid traps. While uniform transience allows to directly use "quenched" statements, ie conditional on the environment, we will have to replace these by "annealed" statements, which average over the environment. Our analysis in that regards is thus largely inspired by the paper \cite{berestycki2018random} and the references cited therein, where the same issue was present. 

\paragraph{Argument for the lower bound:} the lower bound in the mixing time is based on a simple coverage argument. From the concentration of entropy, at time $t = O( \log n)$ the chain is necessarily confined in a set of size $e^{t h + O(\sqrt{t})}$ at most. This becomes $o(n)$ for  $t \leq \log n / h - C \sqrt{\log n}$ and a large constant $C > 0$, which is sufficient to conclude provided the invariant measure is well spread-out across the state space. Since the latter is unknown, we use the proxy measure $\hat{\pi}$ and its expression which is explicit enough for our purpose.

\paragraph{Concentration of nice trajectories:} for the upper bound, the second part of the argument uses the entropic concentration to show a second concentration phenomenon. Specifically, for all $x,y \in V$, we define a set $\frN^{t}(x,y)$ of "nice" trajectories between $x,y$, of length $t$, which will have the property that: 1. they are typical, that is the trajectory of the chain is likely to be nice, 2. the probability $\scP^{t}_{\frN}(x,y)$ of following a nice trajectory concentrates around its mean, which is shown to be independent of the starting point. From these properties we deduce that $\scP^{t}(x, \cdot)$ mixes towards a proxy stationary distribution $\hat{\pi}$. Concentration around the mean will be proved using a concentration inequality for "low-degree functions" on the symmetric group, developped specifically for this purpose and stated as Theorem 1.2 in our work for the reversible case \cite{cutoff_reversible}. These arguments are for the most part identical to the reversible case.

\paragraph{Getting a bound on the worst-case mixing time}: nice paths are in fact really defined only from typical states, so the previous arguments will ensure $\scP_0^{t}(x, \cdot) \simeq \hat{\pi}$ at time $t = \log n / h  + O(\sqrt{\log n})$ only for a typical $x$. To obtain an upper bound on the worst-case mixing time, which as remarked in the first paragraph above is crucial to justify $\hat{\pi}$ is close to being invariant, we will essentially argue that for some $a \geq 1$, at time $s = (\log n)^{a}$ the chain is likely to reach a typical point and have its $t$ subsequent steps follownig a nice trajectory, regardless of its starting state. Thus mixing at time $t = O(\log n)$ from typical states implies mixing at time $s+t = O((\log n)^{a})$ from any starting state. This argument was already apparent in the reversible case, where $s = \Theta(\log \log n)$ steps were enough, which was again the consequence of the quenched uniform lower bound on escape probabilities in quasi-trees. In the general case, a heuristic is as follows: we will get from the estimates on escape probabilities (Prop. \ref{prop:escape_gumbel}) that these are so unlikely to fall below $(\log n)^{-b}$ for some $b > 0$ that we can basically take $(\log n)^{-b}$ as a quenched lower bound. Thus we can expect it takes $(\log n)^{b} l$ steps at worst for the chain to escape potentially trapping neighbourhoods of size $l$, after which the chain is likely to be back in a typical neighbourhood if $l$ is chosen adequatly.  

\subsection{Organisation of the paper} 
In Section \ref{section:first} we give material that will be used throughout the paper, including the definition of quasi-trees. In Section \ref{section:main_arguments}, we give in detail the main technical arguments of the paper that sum up the arguments of the entropic method, from which we will be able to deduce Theorems \ref{thm:cutoff} and \ref{thm:worstcase}. In Section \ref{section:counterex}, we provide a counter-example to uniform entropic mixing time, proving Theorem \ref{thm:counterex}. The three next sections deal with the analysis of the quasi-tree: Section \ref{section:QT1} establish the lower bound on escape probabilities, Section \ref{section:QT2} proves technical results about the regeneration structure of quasi-trees while Section \ref{section:QT3} basically establishes "nice properties" in the quasi-tree setting, including in particular the concentration of the speed and entropy. These nice properties are then transferred back to the finite setting in Section \ref{section:nice_first}, which also contains the arguments that will lead to the upper bound on the worst-case mixing time. Finally, these properties are used in Section \ref{section:nice} where nice paths are properly defined and shown to have their probability concentrate around the mean.
\section{First moment argument, quasi-trees}\label{section:first}

\paragraph{Basic notations:}
throughout the paper, all quantities involved may have an implicit dependency in $n$ and the term constant will refer to quantities that are independent of $n$ and of the randomness. Given a set $S$, $\abs{S}$ denotes its cardinality, $\II_S$ is the indicator of $S$. Given $x,y \in \bR$, we write $x \wedge y := \min(x,y)$, $x \vee y := \max(x,y)$. We use the standard Landau notations $o, O$ for deterministic sequences: $f(n)= O(g(n))$ if there exists a constant $C > 0$ such that $\abs{f(n)} \leq C g(n)$ for all $n$, $f=o(g)$ if $f(n) / g(n) \xrightarrow[n \rightarrow \infty]{} 0$. We may write $O_{\e}(\cdot)$ to precise a potential dependence of the implicit constant in $\e$. Given two functions, $f = \Theta (g)$ if $f =O(g)$ and $g = O(f)$, and $f \gg g$ if $g = o(f)$. All the random variables considered in this paper are defined on an implicit probability space with measure $\bP$. If $Y_n, Z_n, Z$ are random variables, we write $Z_n \xrightarrow[]{\bP} Z$ for convergence in probability and $Z_n = o_{\bP}(Y_n)$ if $Z_n / Y_n \xrightarrow[]{\bP} 0$. In particular $Z_n = o_{\bP}(1)$ if $Z_n \xrightarrow[]{\bP} 0$. An event $A=A(n)$ is said to occur with high probability, if $\bP(A) = 1 - o(1)$.

\subsection{Quenched vs annealed probability, first moment argument}\label{subsec:first_moment}

Consider $(X_t)_{t \geq 0}$ a Markov chain in random environment, ie with random transition probabilities, such as the one defined by \eqref{eq:model0}. There are two laws naturally associated with the process $(X_t)_{t \geq 0}$. The probability $\bP$, which averages over both the random walk and the environment, gives rise to the \emph{annealed} law of the process $(X_t)_{t \geq 0}$ under which it is not a Markov chain. It is however a Markov chain under the \emph{quenched} law, which conditions on the environment. To emphasize the distinction, it is written using a different font, namely $\bfP$ will denote the quenched distribution. For all state $x$, we write $\bP_{x} := \bP \cond{\cdot}{X_0 = x}, \bfP_x := \bfP \cond{\cdot}{X_0 := x}$. Of course taking the expectation of the quenched law gives back the annealed law, which is the basis of the following first moment argument that will be used used throughout the article. To prove a trajectorial event $A$ has quenched probability vanishing to $0$ in probability as $n \rightarrow \infty$, it suffices to prove $A$ has annealed vanishing probability, as Markov's inequality implies for all $\e > 0$, 
\begin{equation}\label{eq:markov_quenched}
	\bP \sbra{ \bfP \sbra{A} \geq \e} \leq \frac{\bP \sbra{A}}{\e}.
\end{equation}
In most of this paper we will first prove statements valid for fixed starting states, chosen independantly of the environment, such as $\bP_{x} \sbra{A} = o(1)$, to obtain $\bfP_{x} \sbra{A} = o_{\bP}(1)$. These can be interpreted as conditional statements for the case where the starting state is also random, with a law which independant of the environment. Results for fixed starting states thus extend automatically to \emph{typical states}, for instance taken uniformly at random in the case the state space is $[n]$. To obtain stronger results valid for all states simultaneously, union bound shows it suffices to improve the annealed error bounds to $\bP_x \sbra{A} = o(1/n)$. This strategy will be used in Section \ref{section:nice_first} to extend results from typical to all states.

\subsection{The two-lift chain and half-integer time steps}

Let us start by rewriting and actually generalizing our model further, which lacks symmetry in the roles of $P_1$ and $P_2$. To obtain a model that is more symmetric, we construct the chain as a projection of its two-lift: this consists in seeing $P_1, P_2$ as two disjoint components of one Markov chain on a twice larger state space, which can be projected back to obtain the original model. This is essentially a matter of unifying notation and recover a setting which is somewhat similar to that of \cite{hermon2020universality}, but it can also prove useful on its own. All in all, the main model we will work with in this paper is the following. 

Let $V := [2n], V_1:= [n], V_2 := [n+1, 2n]$. Let $\sigma$ be now a uniform bijection from $V_1$ to $V_2$, which can be extended as a matching, or involution, $\eta$ on $V$ by 
\begin{equation*}
\eta_{| V_1} := \sigma, \qquad \eta_{| V_2} := \sigma^{-1}.
\end{equation*}
This matching defines an equivalence relation, namely $x \sim \eta(x)$. 
For all $x \in V$, define $V(x) = V_1 \ \II_{x \in V_1} + V_2 \ \II_{x \in V_2}$. Let $P$ be a stochastic matrix on $V$ with block form $P := \left( \begin{smallmatrix}
	P_1 & 0 \\ 0 & P_2 \end{smallmatrix} \right)$ and $p: V \times V \rightarrow [0,1]$ such that $p(x,y) + p(y,x) = 1$ for all $x,y \in V$ and let $q(x,y) := 1 - p(x,y) = p(y,x)$. Consider the Markov chain on $V$ given by the transition probabilities:
\begin{equation}\label{eq:2-lift}
	\forall x,y \in V: \scP(x,y) := \left\{ \begin{array}{l l}
		p(x, \eta(x)) \, P(x,y) & \text{if $V(x) = V(y)$} \\
		q(x, \eta(x)) \, P(\eta(x),y) & \text{if $V(x) \neq V(y)$}.\\
	\end{array}\right.
\end{equation}
We will work with this Markov chain most of the time, but the main object of interest of this paper is rather its projection to the quotient $V / \sim \ = [n]$. As is easily checked, the condition $q(x,y) = p(y,x)$ implies that $\scP(x,y) + \scP(x,\eta(y)) = \scP(\eta(x),y) + \scP(\eta(y),\eta(y))$. This condition ensures the projection is itself a Markov chain, with transition matrix given by
\begin{equation}\label{eq:main_model}
	\forall x,y \in [n], \bar{\scP}(x,y) := \scP(x,y) + \scP(x, \eta(y))
\end{equation}
This model contains \eqref{eq:model0} as a particular case: taking $p(x,y) := p_{1}(x,y-n)$ for all $x \in V_1, y \in V_2$ suffices to define $p$ completely. Furthermore $\sigma - n$ is a uniform permutation on $[n]$. From these it is easily checked that the projected chain $\bar{\scP}$ coincides with $\scP_0$. 

Given a probability measure $\mu$ on $V$, its projection $\bar{\mu}$ on $[n]$ is defined by $\bar{\mu}(x) = \mu(x) + \mu(\eta(x))$ for all $x \in [n]$. Notice that if $\pi$ is invariant for $\scP$, then $\bar{\pi}$ is invariant for $\bar{\scP}$. Furthermore, total variation distance is non-increasing under projections, so
\begin{equation}\label{eq:upper_bound_2lift}
	\TV{\bar{\scP}^{t}(x,\cdot) - \bar{\pi}} \leq \TV{\scP^{t}(x, \cdot) - \pi}.
\end{equation}
Therefore to obtain the upper bound in Theorem \ref{thm:cutoff}, which is the most difficult part of the argument, it suffices to prove an upper bound for the two-lift. In fact, our results hold for both chains.

Now let us introduce another characteristic of the two-lift. The two-lift is by construction made of two disjoint subspaces $V_1, V_2$: when at $x$ the Markov chain stays in the same subspace with probability $p(x,\eta(x))$, and change with probability $q(x,\eta(x))$, after which it takes a step independent of $\eta$. Because of this, it may be convenient to actually distinguish between these two steps, adding transitions at half integer times, defining transition probabilities 
\begin{equation}\label{eq:half_time_steps}
	\begin{array}{r c l}
	\bfP \sbra{X_{t+1/2} = y \ | \ X_{t} = x} &=& \left\{ \begin{array}{l l} 
		p(x,\eta(x)) & \text{if $y = x$} \\
		q(x,\eta(x)) & \text{if $y = \eta(x)$} \\
		0 & \text{otherwise} \end{array} \right.
	\\
	\bfP \sbra{X_{t+1} = y \ | \ X_{t+1/2} = x} &=& P(x,y)
	\end{array}
\end{equation}
for all $t \in \bN$. As is easily checked, the Markov chain $(X_t)_{t \in \bN}$ evaluated at integer times exactly has the transition matrix $\scP$ \eqref{eq:2-lift}.

\subsection{The graph point of view}
To emphasize the underlying tree structure of the Markov chains considered in this paper, these can be visualized as evolving on directed graphs.

\begin{definition}
	If $K$ is a Markov kernel on a countable state space $S$, the underlying graph $G = (S , E)$ is the oriented graph with vertex set $S$ and edge set
	\begin{equation*}
		E := \{(x,y) \in S \ | \ K(x,y) > 0 \}.
	\end{equation*}
	We consider in fact multi-graphs, so that loops and multi-edges are allowed in $G$. In quasi-trees, we will use graphs of Markov chains as parts of a larger structure. To that end, define a morphism of directed graphs as a map between two graphs that preserves the adjacency relations as well as the orientations of the edges. If weighted graphs are considered, such as underlying graphs of Markov chains, we may also implicitely require that weights are preserved. An isomorphism is as usual a bijective morphism.

	Given an oriented edge $e = (x,y) \in E$, write $e^{-} := x$ and $e^{+} := y$ for its initial and terminal vertex respectively. A path $\frp$ in $G$ can be made either of vertices or oriented edges, which can be identified. Given a path $\frp$, we write $K(\frp)$ for the product of transition probabilities along the edges of this path.
	The set $S$ is equipped with the graph metric provided by $G$: for all $x,y \in $, $d_{K}(x,y)$ is the minimal number of edges of a path joining $x$ to $y$. It is not symmetric in general because of the orientation of the edges. It thus gives rise to several notions of balls, depending on the orientation: given $x \in V, r \in \bN \cup \{ + \infty \}$, the forward and backward (closed) $K$-balls of radius $r$ around $x$ are respectively the sets
	\begin{align*}
		&B^{+}_{P}(x,r) := \{ y \in S \ | \ d_{K}(x,y) \leq r \}
		&B^{-}_{P}(x,r) := \{ y \in S \ | \ d_{K}(y,x) \leq r \}.
	\end{align*}
	Without further precision, the $K$-ball of radius $r$ around $x$ is $B_{K}(x,r) := B^{-}_{K}(x,r) \cup B^{+}_{K}(x,r)$, which forgets about the orientation of edges.
	Note that for any $x \in S$, $B_{K}^{+}(x, \infty)$ is the set that can be reached by the chain from the state $x$. If $x$ is recurrent, this is the communicating class of $x$.
\end{definition}

Apply the previous definition to the case at hand: from now on, let $G = (V , E)$ be the underlying graph of the Markov chain $P$: the graphs $G_1 = (V_1, E_1), G_2 = (V_2, E_2)$ spanned by $V_1$ and $V_2$ respectively are the underlying graphs of $P_1$ and $P_2$ and form two disconnected components of $G$. Adding the edges $(x, \eta(x))$ of the uniform matching yields a random graph $G^{\ast} := (V, E^{\ast})$, which can be seen as the underlying graph of the Markov chain $(X_t)_{t \geq 0}$, provided we take into account half-integer time steps \eqref{eq:half_time_steps}. We introduce a terminology inspired by \cite{hermon2020universality} to distinguish between these edges: an edge $(x,y) \in E^{\ast}$ is called a \emph{small-range} edge if $(x,y) \in E$, that is if $P(x,y) > 0$  or $P(y,x) > 0$. An edge $(x, \eta(x))$ added by the random matching will be called a \emph{long-range} edge. Note that long-range edges appear in $E$ with both orientations. When restricting to integer times only, the Markov chain $X$ moves along a long-range and a small-range edge at once or along a small-range edge only. Thus two vertices at $\scP$-distance $1$ from each other may in fact be separated by two edges in $G^{\ast}$. Similarly, a path in $G^{\ast}$ without further precision will be refering to a possible trajectory of the chain $\scP$. 

\subsection{Quasi-trees}\label{section:QT0}

We now define quasi-trees, which are designed to be an infinite approximation of the graph $G^{\ast}$. The same terminology and notation as in the finite setting will be used to emphasize analogies. This abuse is justified by the fact that later, both settings will be coupled so that quantities with matching terminology or notation will be equal. If necessary, we will introduce distinct notation.

\begin{definition}\label{def:quasitree}
	A rooted quasi-tree is a $4$-tuple $(\cG, O, \eta, \iota)$ where $\cG = (\cV, \cE)$ is an oriented graph, $O \in \cV$ is a distinguished vertex and $\iota$ is a map $\iota: \cV \rightarrow V$ that labels vertices of $\cG$ with states in $V$. We call $\iota(x)$ the type of the vertex $x \in V$. Finally $\eta$ is a map $\eta= \cV \rightarrow \cV$, which satisfies: 
	\begin{enumerate}[label=(\roman*)]
		\item $\eta$ is an involution either of $\cV$ or of $\cV \smallsetminus \{ O \}$ with no fixed point. The quasi-tree is called respectively \emph{two-sided} or \emph{one-sided} in these cases.
		\item For all $x \in \cV$, the edge $(x,\eta(x))$ is present in $\cE$ with both orientations. Such edges are called long-range edges, others are called small-range edges. 
		\item for all $x, y \in \cV$, there is a unique family, possibly empty, of long-range edges $e_1, \ldots, e_k$ such that all paths between $x$ and $y$ of minimal length contain the edges $e_1, \ldots, e_k$.
	\end{enumerate}
	In particular, there must exist a unique sequence of long-range edges joining the root $O$ to $x$. If this sequence is non-empty let $x^{\circ}$ be the the endpoint of these long-range edges which is the furthest from $O$. We call such a vertex a \emph{center}. If this sequence is empty, let $x^{\circ} := O$, however the root is not considered a center. 
	
	The two types of edges lead in turn to two additional types of paths and distances.	
	\begin{enumerate}
		\item A \emph{small-range path} is a path made exclusively of small-range edges. Given $x,y \in V$, the \emph{small-range distance} $d_{\SR}(x,y)$ is the minimal number of edges in small-range path from $x$ to $y$. The corresponding \emph{small-range balls} are written $\BSR(x,r)$.
		\item A \emph{long-range path} is a sequence $e=(e_i)_{i=1}^{k}$ of long-range edges such that for all $i \leq k-1$ $d_{\SR}(e_i,e_{i+1}) < \infty$ (so $e$ could be completed with small-range paths between long-range edges to obtain a genuine path in $\cG$). The length $\abs{e}$ of $e$ is the number of edges it contain. It joins two vertices $x ,y \in \cV$ if $d_{\SR}(x, e_1) < \infty$ and $d_{\SR}(e_{k}, y) < \infty$.  Given $x,y \in V$, the \emph{long-range distance} $d_{\LR}(x,y)$ is the minimal length of a long-range path between $x$ and $y$. We write $\BLR(x,r)$ for the corresponding \emph{long-range balls}. 
	\end{enumerate}
	As for the distance in Markov chain underlying graphs, these distances are not symmetric, so we may add $+$ or $-$ sign in exponent to denote forward and backward balls. From the definition, for any $x \in \cV$, $\BSR^{+}(x, \infty) = \BLR^{+}(x,0)$ is the set of vertices which can be joined from $x$ by a small-range path. We call the subgraph spanned by this set the \emph{small-range component} of $x$. We can now state a fourth property we require for quasi-trees: 
	\begin{enumerate}[label=(\roman*)]
		\addtocounter{enumi}{3}
		\item \label{enum:SR_component} for all $x \in \cV$, $\BLR^{+}(x, 0)$ is isomorphic to $B_{P}^{+}(\iota(x^{\circ}), \infty)$,
	\end{enumerate}
	that is, the small-range component of $x$ is the whole set of vertices in $V$ than can be reached from $\iota(x^{\circ})$ in the graph $G$.
\end{definition}

In the sequel we will refer to quasi-trees by their graph component only while keeping the other parameters implicit. 

Property (iii) above clearly gives quasi-trees a tree structure. For this reason, let us introduce a few extra notations. Given $k \geq 0$, the $k$-th level $\cV_k$ is the set of vertices which are at long-range distance $k$ from the root. Set also $\cV_{\leq k} := \bigcup_{i \leq k} \cV_i$. The requirement that $\eta$ has no fixed point implies that every vertex $x$ is joined to vertex $\eta(x)$ in another level. If $x$ is a center, $\eta(x)$ is in a level below, whereas if $x$ is not a center $\eta(x)$ is in a level above.

\paragraph{One-sided quasi-trees and subquasi-trees}

We introduced one-sided quasi-trees to consider subquasi-trees. If $x$ is a center, the subquasi-tree $\cG_x$ of $x$ is the graph spanned by the vertices $y$ for which all paths between $O$ and $y$ pass through $x$. If $x$ is not a center, $\eta(x)$ is and we set $\cG_x := \cG_{\eta(x)}$, so $\cG_x$ does not in fact contain $x$ in that case. Finally the complement subgraph $\cG \smallsetminus \cG_{x}$ is the graph spanned by vertices which can be reached from $O$ by a long-range path that does not use the long-range edge $(x, \eta(x))$. 

\paragraph{Non-backtracking paths, loop-erased paths, deviation, regeneration}
One main interest of having a genuine tree structure is the consideration of loop-erased trajectories. We thus introduce the following definitions.

\begin{definition}\label{def:backtracking_deviation}
	Let $\cG$ be a quasi-tree. A long-range path $e = (e_1, \ldots, e_k)$ backtracks over a distance $l \geq 1$ if it contains a subpath of distinct edges and of length $l$ immediately followed by the reversed path, that is there exists $i \leq k - 2l + 1$ such that $e_{i+l+j} = \overline{e_{i+l-1-j}}$ for all $j \in [0, l - 1]$, where $\overline{e_j}$ denotes the edge $e_j$ with reversed orientation. The loop-erased path $\xi(e)$ is the path obtained after erasing all backtracking steps. The long-range path $e$ is called \emph{non-backtracking} if $\xi(e) = e$.
	
	To previous definition is extended to general paths by extracting the long-range path: given a general path $\frp$, let $\xi(\frp)$ denote the loop-erased path formed by the long-range edges crossed by $\frp$. This is a non-backtracking path, which we call the loop-erased path or loop-erased trace of $\frp$. The long-range distance crossed by $\frp$ is the length of $\xi(\frp)$, or equivalently the long-range distance between its endpoints.
	
	The last two definitions require integer parameters. Given $R \geq 1$, let $\cG^{(R)}$ be the connected component of $O$ of the subgraph of $\cG$ spanned by the set
	\begin{equation}\label{eq:restricted_QT}
		\cV^{(R)} := \{ x \in \cV \ | \ d_{\SR}(x^{\circ}, x) < R \}.
	\end{equation}
	A path $\frp$ is said to \emph{deviate} from a small-range distance $R$ if it is not included in $\cG^{(R)}$. 

	Finally, the following notion of \emph{regeneration edges} will be used in several places: let $L \geq 1$ be an integer and consider a path $\frp$. A long-range edge $e$ crossed by $\frp$ is said to be a regeneration edge for $\frp$ with horizon $L$ if after the first time going through $e$ the path crosses a long-range distance $L$ or ends before going back to the endpoint of $e$ which was first visited by $\frp$. 
\end{definition}

\paragraph{Markov chains on quasi-trees}

Let $\cG$ be a quasi-tree. It is naturally the underlying graph of the Markov chain $(\cX_t)_{t \geq 0}$ on $\cG$ which has transition probabilities: 
\begin{equation}\label{eq:free_transitions}
	\begin{array}{r c l}
		\bfP \cond{\cX_{t+1/2} = y }{\cX_t = x} &=& \left\{ \begin{array}{l l} 
			p(\iota(x), \iota(\eta(x))) & \text{if $y = x$} \\
			q(\iota(x), \iota(\eta(x))) & \text{if $y = \eta(x)$} \\ 0 & \text{otherwise} \end{array} \right.
		\\
		\mathbf{P} \cond{\cX_{t+1} = y }{\cX_{t+1/2} = x} &=&  \left\{ \begin{array}{l l} 
			P(\iota(x), \iota(y)) & \text{if $y^{\circ} = x^{\circ}$} \\
			0 & \text{otherwise} \end{array} \right.
		\end{array}
\end{equation}
for all $t \in \bN$. The kernel of this Markov chain will be written $\cP$.

\subsection{The covering quasi-tree}

Among all quasi-trees, one is very natural to consider: this is the quasi-tree obtained by using the random matching $\eta: V \rightarrow V$ to define the matching on the quasi-tree. Given $x \in V$, this is the quasi-tree $(\cG^{\ast}(x), O, \iota, \tilde{\eta})$ defined by $\iota(O) = x$ and
\begin{equation*}
	\forall y \in \cV: \iota(\tilde{\eta}(y)) = \eta(\iota(y)).
\end{equation*}
The fact that this defines a unique quasi-tree is the consequence from property (iv) in Definition \ref{def:quasitree}, which imply the quasi-tree can be built iteratively. Starting from the small-range of $O$, which is necessarily $B_{P}^{+}(x, \infty)$, the previous equation uniquely determines the long-range edges at long-range distance $0$ from $O$ and thus the whole ball $\BLR(O,1)$. The process iterates to infinity to yield an infinite quasi-tree $\cG^{\ast}$. 

This quasi-tree has the property that it covers exactly $G^{\ast}$: the map $\iota$ is a surjective morphism of $\cG^{\ast}$ onto $G^{\ast}$ which preserves the transition probabilities, so the Markov chain $\cX$ defined above projects exactly onto the chain $X$: for all $t \geq 0$ $\iota(\cX_t) = X_t$  in distribution, conditional on $X_0 = x, \cX_0 = O$. 

The chain $\cX$ on $\cG^{\ast}$ will not be studied per se. We introduced it mainly to obtain easier definitions in the finite setting of objects and quantities that are naturally considered in the idealized setting of quasi-trees. First notice that the notions of long-range and small-range edges defined in the finite setting coincide with the projections under $\iota$ of the corresponding edges in the covering quasi-tree. Thus we can extend the notions of small-range, long-range paths, backtracking, etc. defined above to $G^{\ast}$ by taking their projections under $\iota$. An exception is the long-range distance, which in the finite setting will make sense only if restricting the quasi-tree: let $R \geq 1$ and recall the definition of the restricted quasi-tree $\cG^{(R)}$ \eqref{eq:restricted_QT}. Given $x \in V$ and $l \geq 0$, we define $\BLR^{(G,R)}(x,r)$ in $G^{\ast}$ as the projection 
\begin{equation*}
	\BLR^{(G,R)}(x,r) := \iota \left(\BLR(O,r) \cap \cG^{(R)}(x) \right).
\end{equation*}
The exponent $(G,R)$ is used to distinguish between the two settings. When not necessary, we may drop it from notation. Note from the definition that in $G^{\ast}$, for any vertices $x,y \in V$, $d_{\LR}(x,y) = 0$ if and only if $d_{\SR}(x,y) < R$. 

Finally, we introduce a last definition which is specific to $G^{\ast}$: a \emph{long-range cycle} is a non-deviating non-backtracking long-range path whose starting and ending point are at small-range distance at most $R$ from each other. A subgraph of $G^{\ast}$ is said to be \emph{quasi-tree-like} if it does not contain any long-range cycle. A quasi-tree-like subgraph can thus be identified with a neighbourhood of the root in the covering quasi-tree. Lemma \ref{lem:coupling_quasiT} below will establish that most vertices have in fact a quasi-tree-like neighbourhood. This relies on a bounded degree property, which is the object of the following paragraph.

\paragraph{Bounded degrees and transition probabilities:} 
Assumption \ref{hyp:bdd_delta} imply that the graph $G$ and consequently $G^{\ast}, \cG$ as well, have degrees bounded uniformly in $n$. Together with Assumption \ref{hyp:bdd_p} it also implies all transition probabilities are lower bounded uniformly in $n$. Consequently, let $\Delta$ denote a uniform bound on all the degrees and $\delta > 0$ a uniform bound on transition probabilities, which will serve throughout the paper. It gives in particular the growth of balls in $G^{\ast}$ and any quasi-tree $\cG$: for all $x \in V$ and $l \geq 0$
\begin{equation}\label{eq:ball_growth}
	\abs{\BSR(x,l)} \leq \abs{B_{\scP}(x,l)} \leq \frac{\Delta^{l+1} - 1}{\Delta - 1}, \qquad \abs{\BLR^{(G,R)}(x,l)} \leq \Delta^{R} \frac{\Delta^{R (l+1)} - 1}{\Delta^{R} - 1}.
\end{equation}
In particular, $\Delta^{R}$ can be thought of as the "long-range" degree. 

\subsection{Random setting}\label{subsec:random_QT}

The above definitions should make pretty clear that the forthcoming proofs are based on a coupling between the finite chain $X$ and the chain $\cX$ on an infinite quasi-tree. The quasi-tree in question is similar to the covering quasi-tree but contains much more independence, allowing basically to resample the matching $\eta$ at each new long-range edge. It can be constructed iteratively as follows: $\iota(O)$ is taken uniformly at random in $V$, which determines the small-range component around $O$ by point \ref{enum:SR_component} of the definition. Then to every vertex $x$ whose type $\iota(x)$ is known, adjoin a long-range edge $(x, \eta(x))$ to $x$, with $\iota(\eta(x))$ being taken uniformly in $V \smallsetminus V(\iota(x))$. 

The following formalisation of random quasi-trees will be useful to deal with subquasi-trees and their independency properties. It is analog to the Ulam-Harris labelling for Galton-Watson trees. Let 
\begin{equation}
	\scU := \bigcup_{k \geq 0} \{ \left( (u_i)_{i = 1}^{k}, v \right) \ | \ (u,v) \in V^{k+1}, \forall i \in [2, k], u_{i} \notin V(u_{i-1}), v \notin V(u_{k}) \}.
\end{equation}
The elements in $\scU$ corresponding to $k=0$ have their first coordinate empty and are identified with $V$. Given two sequences $u ,v$, write $uv$ for their concatenation. If $u=(u_i)_{i=1}^{k}$ is non-empty, write $u^{\flat} := (u_i)_{i=1}^{k-1}$ and $\bar{u} := u_k$. 
A rooted quasi-tree can then be represented as a pair $(\cV, O)$ where $\cV \subset \scU$, $O \in V$ satisfy:
	\begin{enumerate}[label=(\roman*)]
		\item $\cV \cap V = \BSR^{+}(O, \infty)$
		\item  $\forall (u,v) \in \cV$, $(u^{\flat}, \bar{u} ) \in \cV$.
		\item there exists an involution $\eta: \cV \rightarrow \cV$ such that for all $(u,v) \in \cV$, 
		\begin{itemize}
			\item either $\eta(u,v) = (u^{\flat}, \bar{u})$, ($(u,v)$ is then a center),
			\item or $\eta(u,v) = (uv, \rho(u,v))$ for some $\rho(u,v) \in V \smallsetminus V(v)$. 
		\end{itemize}
		In the second case, for all $w \in V$, $(uv,w) \in \cV$ if and only if $w \in \BSR^{+}(\rho(u,v), \infty)$.
	\end{enumerate}
	It is easily checked that with this definition the set $\cV$ identifies with the vertex set of a quasi-tree as defined in \ref{def:quasitree}. Note the type of a vertex $(u,v) \in \scU$ is $\iota(u,v) = v$. 

Consider now a family $(\zeta_{(u,v)})_{(u,v) \in \scU}$ of independent random variables, where for every $(u,v) \in \scU$, $\zeta_{(u,v)}$ is uniform in $V \smallsetminus V(v)$. Letting $U$ be independent and uniform in $V$, let $\cV$ be the random subset of $\scU$ obtained by taking $O = U$ and $\rho(u,v) := \zeta(u,v)$ in the case $(u,v)$ is not a center. This yields the same random quasi-tree as in the iterative definition above.

The previous construction is a bit heavy and can be avoided most of the time, however it has the advantage that it allows to consider subquasi-trees in a rigourous way. For every $x \in \scU$, $\cG_x$ is well-defined if $x \in \cV$, otherwise we set it to be empty. Then recall that in the case $x \in \cV$ $\cG_x$ contains $x$ if and only if it is a center. Thus the events $x \in \cV$ and $x$ being a center are both measurable with respect to $\cG \smallsetminus \cG_x$. If $x \in \cV$ is not a center, the knowledge of $\cG \smallsetminus \cG_x$ does not reveal the value of $\eta(x)$, thus conditional on $\cG \smallsetminus \cG_x$ and $x \in \cV$ not being a center, $\cG_x$ has the law of $\cG'_{O'}$ where $(\cG', O')$ is an independent copy of $(\cG, O)$. If $x$ is a center, the second coordinate of $x=(u,v)$ reveals $\iota(x)$ and thus reveals a part of the subquasi-tree $\cG_x$. Thus conditional and $\cG \smallsetminus \cG_x$ and $x \in \cV$ being a center, $\cG_x$ has the law of $\cG'_{O'}$ conditional on $\iota(\eta(O')) = \iota(x)$. 

\subsection{Coupling and sequential generation}\label{subsec:coupling}

We now explain how one can couple the Markov chain $(X_t)_{t \geq 0}$ on $G^{\ast}$ with the Markov chain $(\cX_t)_{t \geq 0}$ on the random infinite quasi-tree $\cG$ defined in the previous section. 

\medskip

The following procedure generates the neighbourhood of a vertex in either $G^{\ast}$ or $\cG$ up to some given long-range distance.

Let $x \in V$ be the point whose long-range neighbourhood is to be explored up to long-range distance $L \geq 0$. For $t \geq 1$ we write $EQ_t$ for the exploration queue, that is the set of vertices which remain to be explored, and $D_t$ for the set of explored vertices. The initiation is similar for $G^{\ast}$ and quasi-trees: $D_0 := \emptyset$; $EQ_0 := \BSR^{+}(x,R)$ in $G$, $EQ_0 := \BSR^{+}(x, \infty)$ in $\cG$. The procedure repeats then the following steps. If one wants to explore a neighbourhood in $G^{\ast}$ sampling is performed without replacement: 
\paragraph{Sequential generation of $G^{\ast}$, sampling without replacement:} for $t \geq 0$,
\begin{enumerate}
	\item pick $y \in EQ_t$: sample $\eta(y)$ uniformly at random in $(V \smallsetminus V(y)) \smallsetminus D_t$,
	\item add $y, \eta(y)$ to $D_{t+1}$ and remove them from $EQ_{t+1}$,
	\item add all vertices $z \in \BSR^{+}(\eta(y),R) \smallsetminus \{ \eta(y) \}$ such that $z \notin D_{t}$ and $d_{\LR}(x,z) < L$ to $EQ_{t+1}$. 
\end{enumerate}

If performed with replacement, each new value $\eta(y)$ is picked uniformly in $V \smallsetminus V(y)$ independently of previous draws and considered a new vertex in a set $\cV$. Specifically the procedure becomes:  
\paragraph{Sequential generation of the quasi-tree $\cG$, sampling with replacement:} for $t \geq 0$,
\begin{enumerate}[label=\arabic*'.]
	\item pick $y \in EQ_t$: sample $\eta(y)$ uniformly at random in $V \smallsetminus V(y)$,
	\item add $y, \eta(y)$ to $D_{t+1}$ and remove $y$ from $EQ_{t+1}$,
	\item add all vertices $z \in \BSR^{+}(\eta(y),R) \smallsetminus \{ \eta(y) \}$ such that $d_{\LR}(x,z) < L$ to $EQ_{t+1}$. 
\end{enumerate}
Since the constraint $z \notin D_t$ has been removed in step 3', vertices which would in $G$ be already explored are here added to the exploration queue and thus considered new vertices. The environment explored is the $L$ long-range ball of a quasi-tree as in Definition \ref{def:quasitree}. Taking $L = \infty$ would thus yield a realization of the infinite random quasi-tree $\cG$.
Finally, the procedures can be adapted to explore ball for the graph distance, that is balls $B_{\scP}$ in $G^{\ast}$. 

\paragraph{Sequential generation along trajectories}

Under the annealed law, the two processes can be generated together with the environment. Later our goal will be to couple weights, which require the exploration of the whole $L$-long range neighbourhood around the trajectory. The sequential generation of the environment along trajectories thus consists in the following steps. Consider for instance the finite setting $G$: let the chain be started at $X_0 := x \in V$, and $D_0 := \emptyset$. Then for all $t \geq 0$:
\begin{enumerate}[label=\alph*.]
	\item Explore the long-range $L$-neighbourhood of $X_t$ with the first procedure described above.
	\item This determines completely the transition probabilities \eqref{eq:half_time_steps} at the state $X_t$, allowing to sample $X_{t+1}$. 
\end{enumerate}
If one considers the second procedure one obtains instead $(\cX_t)_{t \geq 0}$. One can very naturally couple the two procedures and hence both processes using rejection sampling: for each $y \in EQ_t$ sample $\eta(y)$ uniformly in $V \smallsetminus V(y)$ and use it for step 1' in the generation of $\cG$. If in addition $\BSR(\eta(y), R) \cap D_t = \emptyset$, it can also be used for step $1$ in the generation of $G^{\ast}$. Otherwise, make a second draw with the first procedure. This rejection sampling scheme yields a coupling of $(X_t)_{t \geq 0}$ and $(\cX_{t})_{t \geq 0}$ until the first time a newly revealed small-range balls $\BSR(\eta(y), R)$ contain vertices that have already been explored, that is when a long-range cycle appears around the trajectory:
 \begin{equation}
	\tcoup := \inf \left\{ t \geq 0 \ | \ \bigcup_{k = 0}^{t} \BLR(X_k, L) \text{ contains a long-range cycle} \right\}.
\end{equation}

As a first application of the coupling and the first moment argument, we give the following Lemma.

\begin{lemma}\label{lem:coupling_quasiT}
	There exists $C_R, C_L$ large enough such that for $R = C_{R} \log \log n, L = C_L \log \log n$, for all $x \in V, \e \in (0,3/10)$ and $t = O(n^{3/10 - \e})$,
	\begin{enumerate}[label=(\roman*)]
		\item $\bfP_{x} \sbra{\exists s \leq t: \text{$B_{\scP}(X_{s}, \lfloor \log n / (5 \log \Delta) \rfloor)$ is not quasi-tree-like}} = o_{\bP}(n^{-\e}) = o_{\bP}(1)$.

		\item $\bfP_{x} \sbra{\exists s \leq t: \text{$\BLR^{(G,R)}(X_s, L)$ is not quasi-tree-like}} = o_{\bP}(n^{-\e}) =  o_{\bP}(1)$. As a consequence $\bfP_{x} \sbra{\tcoup \leq t} = o_{\bP}(n^{-\e})= o_{\bP}(1)$.
	\end{enumerate}
\end{lemma}

\begin{proof}
	Let the chain be started at $x \in V$ and $t = o(n^{3/10 - \e})$ for $\e \in (0,3/10)$. Write $k := \lfloor \log n / (5 \log \Delta) \rfloor$. By \eqref{eq:ball_growth}, the number of vertices contained in a $\scP$-ball of radius $k$ is $O(\Delta^{k})$. This implies $\bigcup_{s \leq t} B_{\scP}(X_s, k)$ contains $m = O((t+1) \Delta^{k})$ vertices. Therefore, the exploration procedure along the path $X_0 \cdots X_t$ requires at most $m$ draws of values $\eta(y)$, each having probability at most $m \Delta^{R} / n$ that the small-range ball $\BSR(\eta(y), R)$ contains already explored vertices. Hence the number of long-range cycles in $\bigcup_{s \leq t} B_{\scP}(X_s, k)$ is stochastically upper bounded by a binomial $\mathrm{Bin}(m, m \Delta^{R}/n)$ so that by Markov's inequality
	\begin{align*}
		\bP_{x} \sbra{\exists s \leq t: \text{$B_{\scP}(X_{s},k)$ is not quasi-tree-like}} &\leq \frac{m^2 \Delta^{R}}{n} = O \left( \frac{(t+1)^{2} \Delta^{R} \Delta^{2k}}{n} \right) \\
		&= O \left( t^2 n^{-3 / 5 + o(1)} \right) = o(n^{-2 \e + o(1)}) = o(n^{- \e})
	\end{align*}
	by the choice of $t = O(n^{3/10 - \e})$ and $R, L = O(\log \log n)$. This bound on the annealed probability implies the quenched result (i) by the first moment argument \eqref{eq:markov_quenched}. The result (ii) is proved similarly, noting that $\Delta^{RL}$ is still $n^{o(1)}$ (in this case we can take $t$ up to $n^{1/2 - \e})$. 
\end{proof}

\section{The entropic method: main arguments}\label{section:main_arguments}

\subsection{Nice trajectories}

Our application of the entropic method consists in finding a definition of nice trajectories designed to be typical trajectories and have their probability concentrating around the mean. The latter will come from a constraint about an entropy-like quantity for the chain, which arises from comparing the trajectories of the finite chain $X$ with the loop-erased trace of $\cX$ in the (infinite) quasi-tree setting. Other properties will thus be required from nice paths, whose goal is basically to make them close to being non-backtracking trajectories in a quasi-tree. All in all, the defining properties of a nice trajectory will essentially be that: 

\begin{enumerate}[label=(\roman*)]
    \item it is contained in a quasi-tree-like portion of the graph, the endpoint having a quasi-tree-like neighbourhood up to $\scP$-distance $\lfloor \alpha \log n \rfloor$ for some $\alpha > 0$ (Lemma \ref{lem:QTlike_uniform}),
	\item it does not deviate or backtrack too much, and contains sufficiently many regeneration edges (Lemma \ref{lem:typical_paths}),
    \item the drift, ie the long-range distance travelled, and entropy concentrate on this trajectory (Proposition \ref{prop:concentration_G}).
\end{enumerate}

As mentionned in the proof outline, nice trajectories will be likely to be followed by the chain only if it is started at a typical state. To derive the upper bound on the worst-case mixing time, we will thus establish that for an arbitrary starting state the trajectory becomes nice only after some time, namely after $s = (\log n)^{a}$ steps for some $a > 0$. This is why in the sequel we state different results with a "typical, non-shifted" version valid from time $0$ and a "worst-case, shifted" version valid for arbitrary starting states from time $s$. For the first important property of nice trajectories, the typical version has in fact already been proved in Lemma \ref{lem:coupling_quasiT}. The worst-case version is given in the following.

\begin{lemma}\label{lem:QTlike_uniform}
	There exists $a \geq 1 $ such that for all $s \geq (\log n)^{a}$ and $t = o(n^{1/16})$,
	\begin{enumerate}[label=(\roman*)]
		\item for all $C_R, C_L > 0$, letting $R := C_R \log \log n, L := C_L \log \log n$,
		\begin{equation*}
			\max_{x \in V} \bfP_{x} \sbra{\exists t' \in [s,s+t]: \BLR^{(G^{\ast},R)}(X_{t'}, L) \text{ is not quasi-tree-like}} = o_{\bP}(1),
		\end{equation*}
		\item
		\begin{equation}
			\max_{x \in V} \bfP_x \sbra{B_{\scP} \left(X_s, \left\lfloor \frac{\log n}{10 \log \Delta} \right\rfloor \right) \text{ is not quasi-tree-like}} = o_{\bP}(1).
		\end{equation}
	\end{enumerate}
 \end{lemma}

For the second property we recall the notions of deviation, backtracking and regeneration edges are given in Definition \ref{def:backtracking_deviation}.

\begin{lemma}\label{lem:typical_paths}
    Let $\Gamma(R, L, M)$ denote the set of paths $\frp$ in $G^{\ast}$ such that $\frp$ does not deviate from a small-range distance more than $R$, backtrack over a long-range distance $L$ or contain a subpath of length $M$ without a regeneration edge.
	There exist constants $\kappa \geq 1$, $a \geq 1$, $C_R, C_L > 0$ such that for $L = C_L \log \log n$, $R = C_R \log \log n$ and $M = (\log \log n)^{\kappa}$, for all $t = O(\log n)$, the following holds: 
	\begin{enumerate}[label=(\roman*)]
		\item for all $x \in V$:
			\begin{equation*}
				\bfP_{x} \sbra{(X_0 \cdots X_{t}) \in \Gamma(R,L,M)} = 1 - o_{\bP}(1)
			\end{equation*}
		\item for all $s \geq (\log n)^{a}$:
    \begin{equation*}
        \min_{x \in V} \bfP_{x} \sbra{(X_s \cdots X_{s+t}) \in \Gamma(R,L,M)} = 1 - o_{\bP}(1).
    \end{equation*}
\end{enumerate}
\end{lemma}

\begin{remark}\label{rk:uniform_shifted}
	Notice that for any trajectorial event $A$
    \begin{align*}
        \bfP_{x} \sbra{(X_s, X_{s+1}, \ldots) \in A} &= \sum_{y \in V} \bfP_{x} \sbra{X_s = y} \bfP_{y} \sbra{(X_0, X_1, \ldots) \in A} \\
        &\leq \max_{y \in V} \bfP_{y} \sbra{(X_0, X_1, \ldots) \in A}.
    \end{align*}
    Thus if we prove $A$ holds from time $t$ with probability $1 - o_{\bP}(1)$ uniformly over the starting state this automatically extends to larger times $t \geq s$. In the above lemmas for instance it will be sufficient to prove the case $s = (\log n)^{a}$ to establish the case $s \geq (\log n)^{a}$.
\end{remark}

\subsection{Concentration of drift and entropy}\label{subsec:concentration}

The third and main property of nice trajectories consists in the concentration of an entropy-like quantity for the finite chain. To that end, we define weights as follows.

Throughout the paper, we will use the letter $\tau$ for a variety of stopping times. It should be clear from the context what the notation refers to. If $x$ is an element or a set $\tau_x$ will generally denote the hitting time of $x$. If $l \geq 1$, $\tau_l$ will generally denote the time a certain distance $l$ is reached. 
For all $l \geq 0$, consider in this section
\begin{equation*}
	\tau_l := \inf \{ t \geq 0 \ | \ \abs{\xi(X_0 \cdots X_t)} = l \}
\end{equation*}
and fix $R, L \geq 1$ for the rest of this section. Let
\begin{align}
	&\TSR := \inf \{t \geq 0 \ | \ (X_0 \cdots X_t) \text{ deviates from a small-range distance $R$} \} \\
	&\TNB := \inf \{ t \geq 0 \ | (X_0 \cdots X_t) \text{ backtracks over a long-range distance $L$}\}.
\end{align}
These are stopping times which depend on $R, L$. By construction if $\frp$ is a path in $\Gamma(R,L,M)$ as defined in Lemma \ref{lem:typical_paths}, the stopping time $\TSR \wedge \TNB$ does not occur on the trajectory $\frp$.
Given a long-range edge $e$ and $x \in V$ such that $d_{\SR}(x, e^{-}) < R$, define the weights
\begin{equation}\label{eq:def_weights_G}
	\begin{split}
	w_{x, R, L}(e) &:= \bfP_{x} \sbra{\xi(X_0 \cdots X_{\tau_L})_1 = e,  \tau_{L} < \TSR} \\
	w_{R, L}(e \ | \ x) &:= \bfP_{x} \cond{\xi(X_0 \cdots X_{\tau_L})_1 = e}{\tau_{L} < \tau_{\eta(x)} \wedge \TSR}
	\end{split}
\end{equation}
Here $\xi(X_0 \cdots X_{\tau_L})_1$ denotes the first edge of $\xi(X_0 \cdots X_{\tau_L})$. Then for a non-backtracking long-range path $\xi = \xi_1 \cdots \xi_k$, set
\begin{align}
	w_{x,R,L}(\xi) := w_{x,R,L}(\xi_1) \prod_{i=2}^{k} w_{R,L}(\xi_i \ | \ \xi_{i-1}^{+})
\end{align}
where empty products are by convention equal to $1$. The notation is consistent with the the identification of edges with paths of length $1$. 

\begin{remark}\label{rk:weight_sum_1}
	Note that for fixed $x \in V$, $\sum_{e} w_{x,R,L}(e) \leq 1$ and $\sum_{e} w_{R,L}(e \ | \ x) \leq 1$ where the sum is over all long-range edges. By extension the sum of weights over all non-backtracking paths starting from $x$ is at most $1$.
\end{remark}

Given a sequence $u = (u_i)_{i \leq l}$ of length $l$ and $k \geq 1$, we write $(u)_{\leq k} := (u_i)_{i \leq k}$ for the sequence truncated at length $k$. The following lemma show that weights are good proxies for measuring the probability that the loop-erased trace follows a given non-backtracking path. It is the same as for the reversible case as there is here no difference.

\begin{lemma}[{\cite[Lemma 3.3]{cutoff_reversible}}]\label{lem:probability_weights}
	Let $x \in V$ and $k \geq 1$ be an integer. Suppose that $\BLR^{(G,R)}(x, k)$ is quasi-tree-like. Then for all non-backtracking long-range path $\xi$ of length $k$, started in $\BSR^{+}(x, R)$
	\begin{equation}\label{eq:probability_weights}
		\bfP_{x} \sbra{\begin{gathered}\xi(X_0 \cdots X_{\tau_{k + L - 1}})_{\leq k} = \xi, \\ \tau_{k+L-1} < \TSR \wedge \TNB \end{gathered}} \leq w_{x,R,L}(\xi) \leq \bfP_{x} \sbra{\begin{gathered}\xi(X_0 \cdots X_{\tau_{k + L - 1}})_{\leq k} = \xi, \\ \tau_{k+L-1} < \TSR \wedge \TNB \end{gathered}} + u(\xi)
	\end{equation}
	where $u(\xi) \geq 0$ is such that 
	\begin{equation*}
		\sum_{\xi} u(\xi) \leq \bfP_{x} \sbra{\TSR \wedge \TNB \leq \tau_{k+L-1}},
	\end{equation*}
	the sum being over non-backtracking long-range paths of length $k$ from $x$.
\end{lemma}

The first part of the proof of Theorem \ref{thm:cutoff} consists in proving the following quenched concentration phenomenon.

\begin{proposition}\label{prop:concentration_G}
	There exist $\mathscr{d}, h = \Theta(1)$ and constants $a \geq 1, C_R, C_L > 0$ for which the following holds. Letting $R := C_R \log \log n, L := C_L \log \log n$, for all $\e \in (0,1)$, there exist constants $C_{\LR}(a,\e), C_{h}(a,\e)$ such that for all $t \gg 1$, $t = O(\log n)$: 
	\begin{enumerate}[label=(\roman*)]
		\item for all $x \in V$, with high probability
		\begin{equation*}
			\bfP_{x} \sbra{ \begin{array}{c} \abs{ \abs{\xi(X_{0} \cdots X_{t})} - \mathscr{d} t} \leq C_{\LR} \sqrt{t} \\
			\abs{ - \log w_{X_0, R, L}( \xi(X_0 \cdots X_{t})) - h t} \leq C_{h} \sqrt{t}
			 \end{array}} \geq 1 - \e. 
		\end{equation*}
		\item for all $s \geq (\log n)^{a}$, with high probability,
			\begin{equation*}
				\min_{x \in V} \bfP_{x} \sbra{ \begin{array}{c} \abs{ \abs{\xi(X_{s} \cdots X_{s+t})} - \mathscr{d} t} \leq C_{\LR} \sqrt{t} \\
				\abs{ - \log w_{X_s, R, L}( \xi(X_s \cdots X_{s+t})) - h t} \leq C_{h} \sqrt{t}
				\end{array}} \geq 1 - \e. 
			\end{equation*}
	\end{enumerate}
\end{proposition}

\subsection{Concentration of nice paths}

The second part of the argument consists in finding a set of "nice trajectories", defined in particular so that they satisfy the entropic and drift requirements given by Proposition \ref{prop:concentration_G}. Eventually the drift and entropic concentration are used to show that the probability of following a nice trajectory concentrates around its mean. The latter can be computed, providing an approximate stationary distribution $\hat{\pi}$ for $\scP$ on $V$.

Given $u \in V$, $L \geq 1$ consider again $\tau_L := \inf \{ t \geq 0 \ | \ \abs{\xi(X_0 \cdots X_t)} = L \}$ and
	\begin{equation}\label{eq:bfQu_G}
		\bfQ^{(L)}_{u} := \bfP \cond{\cdot}{X_{1/2} = u, \tau_{\eta(u)} > \tau_L}
	\end{equation}
In words, this measure considers trajectories immediately after a regeneration time. By Lemma \ref{lem:typical_paths}, if one takes $L = \Theta(\log \log n)$, the conditionning by $\tau_{L} < \tau_{X_0}$ essentially forbids the chain to come back at all to $u$ on a time scale $O(\log n)$. If $\nu$ is a probability measure on $V$, write $\bfQ^{(L)}_{\nu} := \sum_{u \in V} \nu(u) \bfQ^{(L)}_{u}$ and $\bfE^{(L)}_{\bfQ_{\nu}}$ for the expectation with respect to this measure. All in all, the second part of the argument is summarized in the following proposition. 

\begin{proposition}\label{prop:nice_approx}
	There exist a deterministic probability measure $\nu$ on $V$, a deterministic $s_0 =\Theta(\log n)$, constants $C_L > 0, a, \kappa \geq 1$ and for all $x, y \in V, t \in \bN$ a set $\frN^{t}(x,y)$ of length $t$ paths between $x$ and $y$ for which the following holds. Let $L := C_L \log \log n, M := (\log \log n)^{\kappa}$ and write $\scP_{\frN}^{t}(x,y) := \sum_{\frp \in \frN^{t}(x,y)} \scP(\frp)$. Consider the random probability measure
		\begin{equation}\label{eq:pihat}
			\hat{\pi}(v) := \frac{1}{\bfE_{\bfQ_{\nu}^{(L)}} \sbra{T_1 \wedge M}} \sum_{r =0}^{M} \bfQ_{\nu}^{(L)} \sbra{X_{r+s_0} = v, r < T_1 \leq M}
		\end{equation}
		where $T_1$ denotes the first regeneration time with horizon $L = C_L \log \log n$. For all $\e \in (0,1)$ there exists $C(\e) >0$ such that for $t = \log n / h + C(\e) \sqrt{\log n}$,
	\begin{enumerate}[label =(\roman*)]
		\item for all $x \in V$, with high probability,
		\begin{equation*}
			\sum_{y \in V} \scP^{t}_{\frN}(x,y) \geq 1 - \e,
		\end{equation*}\label{enum:not_nice_o1}
		\item for all $s \geq (\log n)^{a}$, with high probability
			\begin{equation*}
			\min_{x \in V} \sum_{y \in V} \scP^{s} \scP^{t}_{\frN} \, (x,y) \geq 1 - \e.
			\end{equation*}\label{enum:not_nice_uniform}
		\item there exists $c = (c_v)_{v \in V}$ such that $\sum_{v \in V} c_v = o_{\bP}(1)$ and with high probability, for all $x,y \in V$,
		\begin{equation*}
			\scP_{\frN}^{t}(x,y) \leq (1 + \e) \hat{\pi}(y) + c(y) + \frac{\e}{n},
		\end{equation*}\label{enum:nice_upper_bound}
		\item the measure $\hat{\pi}$ can be decomposed as $\hat{\pi} = \hat{\pi}_1 + \hat{\pi}_2$ with 
		\begin{equation*}
			\norm{\hat{\pi}_1}_{2}^{2} = o_{\bP}( (\log n)^{b} / n), \qquad \hat{\pi}_2(V) = o_{\bP}(1)
		\end{equation*}
		for some $b > 0$. \label{enum:pihat_flat}
	\end{enumerate}
\end{proposition}

\begin{proof}[Proof of Theorem \ref{thm:cutoff} and \ref{thm:worstcase}]
	Recall $\scP$ is the transition matrix of the two-lift chain on $V$, which projects to a transition matrix $\bar{\scP}$ on $[n]$. 

	\medskip

	Start with the upper bound on the mixing time. Let $\e > 0$ and $t := \log n / h + C(\e) \sqrt{\log n}$. Since $\scP \geq \scP_{\frN}$ entry-wise, for all $x \in V$ and $s \geq 0$,
	\begin{align*}
		\TV{\hat{\pi} - \scP^{s+t}(x,\cdot)} &= \sum_{z \in V} \left[ \hat{\pi}(z) - \scP^{s+t}(x,z) \right]_{+} \leq \sum_{y,z \in V} \scP^{s}(x,y) \left[ \hat{\pi}(z) - \scP^{t}(y,z) \right]_{+} \\
		&\leq \sum_{y,z \in V} \scP^{s}(x,y) \left[(1 + \e) \hat{\pi}(z) + c(z) + \frac{\e}{n} - \scP_{\frN}^{t}(y,z) \right]_{+}.
	\end{align*}
	Point \ref{enum:nice_upper_bound} of the above proposition implies that with high probability, the right hand side summands are non-negative for all $x \in V$, so the sum can be computed to obtain that with high probability,
	\begin{align*}
		\TV{\hat{\pi} - \scP^{s+t}(x,\cdot)} \leq 1 - \scP^{s} \scP_{\frN}^{t}(x,y) + 3 \e.
	\end{align*}
	Using point \ref{enum:not_nice_uniform}, if $s \geq (\log n)^{a}$ with high probability
	\begin{equation*}
		\max_{x \in V} \TV{\scP^{s+t}(x, \cdot) - \hat{\pi}} \leq 4 \e
	\end{equation*}
	which projects by \eqref{eq:upper_bound_2lift} to 
	\begin{equation*}
		\max_{x \in V} \TV{\bar{\scP}^{s+t}(x, \cdot) - \bar{\hat{\pi}}} \leq 4 \e.
	\end{equation*}
	with $\bar{\hat{\pi}}$ the projection onto $[n]$ of $\hat{\pi}$.
	Since this estimate is uniform in the starting state, it extends to any starting distribution and particularly to a stationary distribution. Thus for any stationary distribution $\pi$ of $\bar{\scP}$,
	\begin{equation}\label{eq:bound_hatpi}
		\TV{\pi - \bar{\hat{\pi}}} \leq 4 \e.
	\end{equation}
	and from triangular inequality we obtain that with high probability
	\begin{equation*}
		\max_{x \in [n]} \TV{\bar{\scP}^{s+t}(x, \cdot) - \pi} \leq 8 \e.
	\end{equation*}
	Since this is valid for any invariant measure, the latter must be unique. This proves in particular the irreducibility and aperiodicity of the chain and establishes Theorem \ref{thm:worstcase}. 

	On the other hand for a fixed $x \in V$, taking $s = 0$ in the above bounds yields and using point \ref{enum:not_nice_o1} that with high probability,
	\begin{equation*}
		\TV{\bar{\scP}^{t}(x, \cdot) - \pi} \leq 8 \e,
	\end{equation*}
	proving the upper bound in Theorem \ref{thm:cutoff}.

\medskip

Let us prove the lower bound. For all $t \geq 0, \theta > 0$ and $x,y \in [n]$,
\begin{align*}
	\bar{\scP}^{t}(x,y) &= \scP^{t}(x,y) + \scP^{t}(x, \eta(y)) \\
	&\geq \bfP_{x} \sbra{X_t \in \{y, \eta(y) \}, w_{x,R,L}(\xi(X_0 \cdots X_t)) \leq \theta}.
\end{align*}
If equality holds, then 
\begin{equation*}
	\bar{\hat{\pi}}(y) - \bfP_{x} \sbra{X_t \in \{y, \eta(y) \}, w_{x,R,L}(\xi(X_0 \cdots X_t)) \leq \theta} \leq \left[ \bar{\hat{\pi}}(y) - \bar{\scP}^{t}(x,y) \right]_{+}
\end{equation*}
If equality does not hold, there must exist a non-backtracking long-range path $\xi$ between $x$ and $y$ or $x$ and $\eta(y)$ for which $w(\xi) > \theta$, in which case
\begin{multline*}
	\bar{\hat{\pi}}(y) - \bfP_{x} \sbra{X_t \in \{y, \eta(y) \}, w_{x,R,L}(\xi(X_0 \cdots X_t)) \leq \theta} \leq \hat{\pi}(y) \, \II_{\exists \xi: w_{x,R,L}(\xi) > \theta} \\
	+ \hat{\pi}(\eta(y)) \, \II_{\exists \xi: w_{x,R,L}(\xi) > \theta}.
\end{multline*}
We did not precise where $\xi$ lies to ease notation in the indicator functions but it depends on $y$ and $\eta(y)$ respectively.
Combining the two inequalities we obtain that in either case 
\begin{align*}
	\bar{\hat{\pi}}(y) - \bfP_{x} \sbra{X_t \in \{y, \eta(y) \}, w_{x,R,L}(\xi(X_0 \cdots X_t)) \leq \theta} &\leq \left[ \bar{\hat{\pi}}(y) - \bar{\scP}^{t}(x,y) \right]_{+} + \hat{\pi}(y) \, \II_{\exists \xi: w_{x,R,L}(\xi) > \theta} \\
	&+ \hat{\pi}(\eta(y)) \, \II_{\exists \xi: w_{x,R,L}(\xi) > \theta}.
\end{align*}
Decompose now $\hat{\pi} = \hat{\pi}_1 + \hat{\pi}_{2}$ as in Proposition \ref{prop:nice_approx}. From \ref{enum:pihat_flat}, summing over $y \in [n]$ in the previous inequality yields
\begin{equation*}
	\bfP_{x} \sbra{w_{x,R,L}(\xi(X_0 \cdots X_t)) > \theta} \leq \TV{\bar{\scP}^{t}(x, \cdot) - \bar{\hat{\pi}}} + \sum_{y \in V} \hat{\pi}_{1}(y) \II_{\exists \xi: w_{x,R,L}(\xi) > \theta} + o_{\bP}(1).
\end{equation*}
By Cauchy-Schwarz inequality,
\begin{equation*}
	\sum_{y \in V} \hat{\pi}_{1}(y) \II (\exists \xi: w_{x,R,L}(\xi) > \theta) \leq \left( \sum_{y \in V} \hat{\pi}_{1}(y)^{2} \right)^{1/2} \left( \sum_{\xi} \II (w_{x,R,L}(\xi) > \theta) \right)^{1/2},
\end{equation*}
where the second sum is over non-backtracking long-range paths from $x$. By Remark \ref{rk:weight_sum_1} weights sum up to at most $1$ hence so this sum contains at most $\theta^{-1}$ positive terms and
\begin{equation}\label{eq:mixing_lower_bound}
	\bfP_{x} \sbra{X_t = y, w_{x,R,L}(\xi(X_0 \cdots X_t)) > \theta} \leq \TV{\bar{\scP}^{t}(x, \cdot) - \bar{\hat{\pi}}} + \sqrt{\frac{1}{\theta} \sum_{y \in V} \hat{\pi}_{1}(y)^{2}} + o_{\bP}(1).
\end{equation}
To complete the proof, let $\e \in (0, 1)$ and specialize to $t := \log n / h - C_1 \sqrt{\log n}$ and  $\theta := n^{-1} \exp(C_{2} \sqrt{\log n})$ for some for some $C_1(\e)$, $C_2(\e) > 0$. Choosing the constant $C_1$ large enough, $\exp(-t h - C_{h}(\e) \sqrt{t}) = n^{-1} \exp((C_1 - C_{h} / h ) \sqrt{\log n} - o(\sqrt{\log n})) > \theta$ for large enough $n$, hence Proposition \ref{prop:concentration_G} implies that the left hand-side of \eqref{eq:mixing_lower_bound} is at least $1 - \epsilon$ with high probability. On the other hand, \ref{enum:pihat_flat} shows that the square-root in the right hand side is $o_{\bP}(1)$. All in all, this proves that with high probability 
\begin{equation*}
	\TV{\bar{\scP}^{t}(x, \cdot) - \bar{\hat{\pi}}} \geq 1 - 2 \epsilon.
\end{equation*}
Finally by \eqref{eq:bound_hatpi} we deduce
\begin{equation*}
	\TV{\bar{\scP}^{t}(x, \cdot) - \bar{\pi}} \geq 1 - 6 \epsilon
\end{equation*}
with high probability.
\end{proof}

\begin{remark}\label{rk:pi_flat}
	From the previous proof, the decomposition of $\hat{\pi}$ given by \ref{enum:pihat_flat} also extends to the unique invariant measures of $\scP$ and $\bar{\scP}$: $\pi = \pi_1 + \pi_2$ with $\norm{\pi_1}_{2}^{2} = o_{\bP} ((\log n)^{b} / n)$ and $\pi_2([n]) = o_{\bP}(1)$.
\end{remark}

\section{Counter-example to uniform entropic time}\label{section:counterex}

In this section we prove Theorem \ref{thm:counterex}.
Let $n \geq 1$ and $\delta \in (0,1)$. Consider the reversible Markov chains $Q_1, Q_2$ on the segments $\{0, 1, 2 \}$ and $\{0, 1\}$ which go from left to right with probability $\delta$, from right to left with probability $1 -\delta$ and stay at the extremities with remaining probability. Let $P_1$ be the transition matrix on $[6n]$ corresponding to $2n$ disjoint copies of $Q_1$, that is $P_{1}(i,j) = Q_{1}(i \mod 3, \ j \mod 3)$ for all $i,j \in [6n]$. Similarly let $P_2$ be the matrix corresponding to $3n$ copies of $Q_2$. Our counter-example to uniform mixing time is obtained with the matrix $\bar{\scP} := 1/2 (P_1 + S P_2 S^{-1})$ and small $\delta$. On this example, it will be exceptionally be clearer to forget about the two-lift and half-integer time transitions. We thus identify each $x$ with $\eta(x)$, both in the finite and quasi-tree setting, which has the effect of pruning the long-range edges in the graphs, and we use the notations $G^{\ast}, \cG, \cX$ to refer to the objects obtained after this operation.

Consider first the quasi-tree $\cG$ that would be obtained from $P_1$ and $P_2$: because the underlying graphs $G_1, G_2$ are forests, all realizations of quasi-trees are in fact genuine trees. On the other hand, if $\delta$ is small, the Markov chains on segments are heavily biased towards $0$. Say a vertex $x \in \cG$ has type $0$ if the map $\iota$ identifies $x$ with the leftmost vertex of a segment. Then consider the quasi-tree obtained $T$ if all centers in $\cG$ have type $0$. As one can check, this is a periodic weighted tree $T$ where except the root all vertices have their transition probabilities depending only on whether the degree is $2$ or $3$, see Figure \ref{fig:counterex}. Notice that for sufficiently small $\delta$, the chain on $T$ becomes recurrent. 
\begin{figure}
	\centering
	\begin{subfigure}{.5\textwidth}
		\centering
		\includegraphics[width=\textwidth]{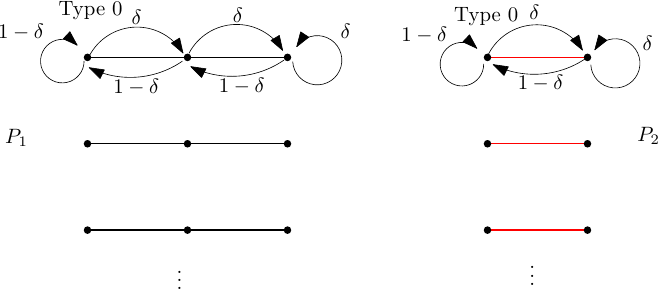}
	\end{subfigure}	
	\hfill
	\begin{subfigure}{.4\textwidth}
		\centering
		\includegraphics[width=\textwidth]{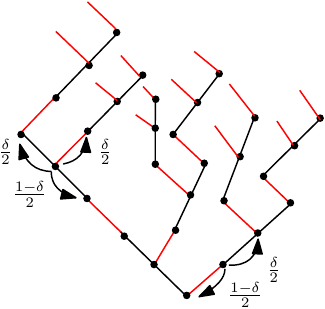}
	\end{subfigure}	
	\caption{The finite chains $P_1$, $P_2$ and the quasi-tree $T$ with all centers of type $0$, after identifying each $x \simeq \eta(x)$. The probability to remain at a vertex is not represented.}
	\label{fig:counterex}
\end{figure}
Of course this configuration occurs with probability $0$, however a finite portion of $T$ can arise with positive probability. Coming back to the finite graph $G^{\ast}$, which resembles locally a quasi-tree, this probability is in fact sufficiently large that $G^{\ast}$ also contains such neighbourhoods, which trap the chain for a long time. 

\begin{lemma}\label{lem:counter_ex}
	There exists $\beta > 0$ such that for all $\delta \in (0,1)$, with high probability there exists $x \in [6n]$ for which
	\begin{equation*}
		\bfP_{x} \sbra{X_t \notin B_{\bar{\scP}}(x, \beta \log \log n)} = o(1)
	\end{equation*}
	for all $t \leq (\log n)^{\beta \log(2 \delta / (1- \delta))}$.
\end{lemma}

\begin{proof}
	Let $l := \beta \log \log n$. In the present setting, the degrees of $G^{\ast}$ and $\cG$ are clearly bounded by $\Delta = 3$. Consequently the exploration of a $\scP$-ball of radius $l$ around any $x \in V$ reveals at most $O(\Delta^{l}) = O((\log n)^{\beta \log 3})$ vertices. Suppose we reveal such neighbourhoods around vertices $x_0, \ldots, x_k$, each time picking a vertex disjoint from the previously explored neighbourhoods. As explained in Section \ref{subsec:coupling}, this exploration can be performed with replacement at a total variation cost which is $o(1)$ as long as $k \Delta^{l} = o(\sqrt{n})$. If this holds we can thus make the neighbourhoods of all $x_k$ coincide with trees with probability $1 - o(1)$. Now in the quasi-tree $\cG$, each center has type $0$ with probability at least $1/3$, so the probability that all centers up to depth $l$ have type $0$ is at least $3^{-\Delta^{l+1}}$. Consequently, the probability that no vertex among $x_0, \ldots, x_k$ has its neighbourhood isomorphic to $T_{\leq l}$, the subtree of $T$ up to depth $l$, is upper bounded by
	\begin{equation*}
		(1 - 3^{-\Delta^{l+1}})^{k} \leq \exp(-k 3^{-\Delta^{l+1}}) = \exp( -k 9^{- (\log n)^{\beta \log 3}})
	\end{equation*}
	 For $\beta > 0$ small enough, we can take for instance $k = n^{1/4}$ so that the previous probability is $o(1)$ and in the same time $k \Delta^{l} = o(\sqrt{n})$ so the approximation with the quasi-tree is justified. This proves that with high probability, there exists $x_0 \in [6n]$ such that $B_{\bar{\scP}}(x_0, l)$ coincides with $T_{\leq l}$. 

	 It remains to show that the Markov chain $(X_t)_{t \geq 0}$ is very unlikely to escape $B_{\bar{\scP}}(x, l)$ by time $e^{c l}$. It suffices to prove the statement for the Markov chain $(\cX_t)_{t \geq 0}$ on $T$. This is done by considering the maximal drift of the chain: the vertices which maximize the probability to go away from the root are vertices of degree $3$. If $\cX_t$ is at such a vertex, it goes away from the root with probability $\delta$, moves towards the root with probability with probability $(1 - \delta) /2$ and stays put with remaining probability. Consequently, the distance between $\cX_t$ and the root is stochastically upper bounded by a lazy random walk on the segment $[0, l]$ which goes left with probability $(1 - \delta )/ 2$ and right with probability $\delta$. By the classical Gambler's ruin problem, this random walk has probability at most $O((2 \delta / (1- \delta))^{l})$ to go from $1$ to $l$ before returning to $0$. In turn this implies that the hitting time $\tau_l$ of level $l$ by $\cX_t$ dominates stochastically a geometric random variable with parameter $r := (2 \delta / (1- \delta))^{l}$. In the end we deduce that if $t = t(n) = o(1/r)$ 
	 \begin{align*}
		\bP \sbra{X_t \notin B_{\bar{\scP}}(x_0, l)} &\leq \bP \sbra{\tau_l \leq t} \\
		&\leq 1 - (1- r)^{t} \\
		&= r t + o(rt) = o(1).
	 \end{align*}
	 Since $1/r=(\log n)^{\beta \log(2 \delta / (1- \delta))}$, this proves the result.
\end{proof}

\begin{proof}[Proof of Theorem \ref{thm:counterex}]
	Let $a > 1$, $t := (\log n)^{a}$ and $\e \in (0, 1)$. By Theorem \ref{thm:cutoff} the chain is irreducible and aperiodic with high probability. Let $\pi$ be the unique stationary measure of $\bar{\scP}$. Thanks to the lemma, we can find $\beta > 0$ and $\delta > 0$ small enough for which there exists with high probability a state $x \in [n]$ satisfying $\bfP_{x} \sbra{X_t \notin B} = o(1)$, where  $B := B_{\bar{\scP}}(x, \beta \log \log n)$. Then recall the decomposition of $\pi = \pi_1 + \pi_2$ given by Remark \ref{rk:pi_flat}. Since the ball $B$ contains at most $O(\Delta^{\beta \log \log n}) = n^{o(1)}$ vertices, Cauchy-Schwarz inequality implies
	\begin{equation*}
		\pi_1(B) \leq \sqrt{\abs{B} \sum_{u \in B} \pi_1(u)^{2}} = n^{-1/2 + o_{\bP}(1)},
	\end{equation*}
	while $\pi_2(B) \leq \pi_2([n]) = o_{\bP}(1)$.
	Consequently
	\begin{equation*}
		\bar{\scP}^{t}(x, B) - \pi(B) = 1 - o_{\bP}(1)
	\end{equation*}
	which implies $\tmix(x, \e) \geq (\log n)^{a}$.

\end{proof}

\begin{remark}\label{rk:trapping}
	As was noticed above for $\delta$ sufficiently small the Markov chain $(\cX_t)_{t \geq 0}$ on the tree $T$ becomes sufficiently biased towards the root to be recurrent. On the other hand, the Borel-Cantelli lemma implies that the quasi-tree $\cG$ contains a.s. infintely many copies of $T_{\leq l}$ for any $l$. The existence of such "trapping neighbourhoods" which could potentially slow down the chain by an important amount is a problem that will arise in the next section where we study asymptotic properties of $(\cX_t)_{t \geq 0}$. In particular the spectral radius of this Markov chain may be equal to $1$, despite the tree-like aspect of the graph.
\end{remark}

\section[Quasi-tree I: Escape probabilities]{Analysis on the quasi-tree on quasi-trees I: escape probabilities}\label{section:QT1}

The objective of the three following sections is to prove the concentration of the drift and entropy along with the other nice properties of Section \ref{section:main_arguments} for the Markov chain $(\cX_t)_{t \geq 0}$ on a random infinite quasi-tree $\cG$. This will allow us later to deduce the corresponding statements for the chain $X$ thanks to the coupling presented in Section \ref{subsec:coupling}. As in \cite{berestycki2018random, hermon2020universality}, the argument is based on the existence of regeneration times, which are times at which $\cX_t$ visits a long-range edge for the first and last time. A first step towards this objective is to lower bound the probability of escaping to infinity in the quasi-tree. Because of the potential trapping phenomenon mentionned in Remark \ref{rk:trapping}, this requires more involved arguments than in the reversible vase.

\subsection{Escape probabilities}

We recall $\cG$ denotes a random rooted quasi-tree which under $\bP$ has the law of the random quasi-tree described in Section \ref{subsec:random_QT}. Its vertex set is $\cV$.

\begin{definition}\label{def:escape}
	Given a non-center vertex $x \in \cV$, let
	\begin{equation*}
		\qesc(x) := \bfP_{x} \sbra{\forall t \geq 1: \cX_{t} \in \cG_x}. 
	\end{equation*}
	be the quenched probability that the chain enters the subquasi-tree of $x$ and never leaves it. If $x$ is a center, it is useful to also consider starting at time $1/2$ and let
	\begin{equation*}
		\qesc(x) := \bfP_{x} \sbra{\forall t \geq 0: \cX_{t} \in \cG_x} \wedge \bfP \cond{\tau_{\eta(x)} = \infty}{\cX_{1/2} = x}.
	\end{equation*}
	We call these quantities the \emph{escape probability} at $x$.
\end{definition}

\begin{remark}\label{rk:escape_center}
	Note that if $x$ is not a center,
	\begin{equation*}
		\qesc(x) \geq q(x, \eta(x)) \bfP \cond{\tau_{x} = \infty}{\cX_{1/2} = \eta(x)}.
	\end{equation*} 
	Similarly, if $x$ is a center, Assumption \ref{hyp:cc3} asserts there exists $y \neq x$ in the same small-range component as $x$, which is thus not a center, so that
	\begin{equation*}
		\qesc(x) \geq p(x, \eta(x)) \, P(x,y) \, \qesc(y). 
	\end{equation*}
	By Assumption \ref{hyp:bdd_p} the entries of $p$ and $q$ are bounded. Thus to lower bound escape probabilities, it matters little to consider a center vertex or not, and to start at integer or half-integer time.

	Furthermore, recall the Ulam labelling of Section \ref{subsec:random_QT}. For all $x \in \scU$, if $x \in \cV$ is not a center then the probability $\qesc(x)$ is measurable with respect to $\cG_x$ and $\iota(x)$ only. Hence
	\begin{equation}\label{eq:subquasiT}
		\bP \cond{\qesc(x) \in \cdot}{x \in \cV, \cG \smallsetminus \cG_x} = \bP \cond{\qesc(O) \in \cdot}{\iota(O) = \iota(x)}.
	\end{equation}
	If $x$ is a center,
	\begin{equation*}
		\bP \cond{\qesc(x) \in \cdot}{x \in \cV, \cG \smallsetminus \cG_x} = \bP \cond{\qesc(\eta(O)) \in \cdot}{\iota(O) = \iota(\eta(x)), \iota(\eta(O)) = \iota(x)}.
	\end{equation*}
	Thus the law of escape probabilities is entirely determined by the case of the root, conditional on $\iota(O)$ and $\iota(\eta(O))$. To ease the notation, we will somes times keep $\iota$ implicit in the sequel and write only $O, \eta(O)$ when conditionning on the types of these vertices.
\end{remark}

In the reversible model, a uniform lower bound remains satisfied as proved by Proposition 4.1 in \cite{cutoff_reversible}. However, in the general model, a uniform lower bound on escape probabilities may not be achievable. The counter example of Section \ref{section:counterex} and Remark \ref{rk:trapping} show that in the corresponding quasi-tree, 
\begin{equation*}
	\inf_{x \in \cV} \qesc(x) = 0.
\end{equation*}
This phenomenon, which was already present in \cite{berestycki2018random}, is ultimately what prevents uniform cutoff. Cutoff from typical vertices remains possible as we will argue that "trapping neighbourhoods", where escape probabilities tend to $0$, are sufficiently rare that they are essentially ignored from the chain, provided it is started at a typical vertex. The starting point is the following lemma, which will be proved with comparison arguments. The result is stated for the root but from Remark \ref{rk:escape_center} it extends to other subquasi-trees.

\begin{lemma}\label{lem:escape_probability}
	For all types of $O, \eta(O)$,
	\begin{equation*}
		\bE \cond{\qesc(O)}{O, \eta(O)} = \Theta(1).
	\end{equation*}
	Equivalently, there exists a constant $q_0 > 0$ such that for all types of $O, \eta(O)$,
	\begin{equation}
		\bP \cond{ \qesc(O) \geq q_0}{O, \eta(O)} \geq q_0.
	\end{equation}
\end{lemma}

Establishing this lemma is the main difficulty so the proof is deferred to the next sections.
Using the tree structure of $\cG$, the previous lemma can be boostrapped to yield a much better control on the behaviour of escape probability at typical vertices.

Using the tree structure of $\cG$, the previous lemma can be boostrapped to yield a much better control on the behaviour of the escape probability at typical vertices.

\begin{proposition}\label{prop:escape_gumbel}
	There exist constants $q_0, \delta \in (0,1)$ such that for all types of $O, \eta(O)$ and $k \geq 0$,  
	\begin{equation}
		\bP \cond{\qesc(x) < q_0 \delta^{4k}}{O, \eta(O)} \leq q_{0}^{2^k}.
	\end{equation} 
\end{proposition}

\begin{proof}[Proof of Proposition \ref{prop:escape_gumbel}]
	Recall that from assumptions \ref{hyp:bdd_delta} and \ref{hyp:bdd_p} we suppose all transition probabilities are lower bounded by a constant $\delta$. This is also the constant $\delta$ of the proposition. By Assumption \ref{hyp:cc3}, every vertex $y \in \cV$ has at least one vertex at small-range distance $1$ and two vertices if $\iota(y) \in V_1$. Consequently, for any realization of $\cG_{x}$ and $k \geq 0$, there are at least $2^{k}$ centers at distance at long-range distance $2k$ from $x$ in $\cG_x$ which can be joined from $x$ in less than $4k$ steps by $\cX_t$. Since transition probabilities are bounded below by $\delta$, it suffices that one of these vertices $y$ has $\qesc(y) \geq q_0$ to obtain that $\qesc(O) \geq q_0 \delta^{4k}$. Now since these vertices are centers and at the same long-range distance from $x$ their respective subtrees are disjoint and thus independent. The probability that all $y$ have $\qesc(y) < q_0$ is thus upper bounded by $q_{0}^{2^k}$ by Lemma \ref{lem:escape_probability}. 
\end{proof}

From the previous proposition, we also deduce the following. Below, a path refers to a possible trajectory of the chain $\cP$ and its $\cP$-length is the corresponding number of transitions. 

\begin{proposition}\label{prop:escape_dense}
	For all $C > 0$ there exist constants $q_0, \alpha \in (0,1]$ such that the following holds: for all $r \geq 0$, with probability at least $1 - \exp(-Cr)$ conditional on the types of $O, \eta(O)$, every path in $\cG$ of $\cP$-length $r$ starting from $O$ contains a proportion at least $\alpha$ of vertices $x$ such that $\qesc(x) \geq q_0$.
\end{proposition}

The proof is concluded with an argument similar in spirit to Lemma 2.2 in \cite{dembo2002large} (initially due to Grimmett and Kesten \cite{grimmett2001random}) and used to prove Lemma 2.3 in \cite{berestycki2018random}. 

\begin{proof}[Proof of Prop \ref{prop:escape_dense}]
	We reason conditional on $\iota(O)$ and $\iota(\eta(O))$ so we can assume these are fixed. 
	We distinguish the case of large and small paths, for which it suffices to find one state with lower bounded escape probability . Namely if $r < 10$, Lemma \ref{lem:escape_probability} shows that $\qesc(O) \geq q_0$ with positive probability, in which case every path of length at most $10$ contains a proportion at least $1/10$ vertices with escape probability lower bounded by $q_0$. 
	
	For larger paths, since $\cG$ is random, we use the Ulam labelling to consider vertices independently of the realization of $\cG$. Given $q_0 > 0$ to be determined and $x \in \scU$, let $A_x = \II_{x \in \cV, \qesc(x) \geq q_0}$.
	Consider now a self-avoiding path $\frp$ in $\scU$ started at $\iota(O)$ which can be the realization of a path of $\cP$-length $r \geq 10$ in $\cG$, starting at $O$. If $\frp$ is indeed a path in $\cG$, let $K_{\frp}$ be the set of vertices which are at small-range distance at most $1$ from $\frp$ but are not endpoints of a long-range edge crossed by $\frp$. Let us give an example for clarity: a path in $\scU$ may for instance contain vertices $u_0, u_1, (u_1, v_1), (u_1,v_2) \in \scU$ such that $u_0 = \iota(O)$, $P(u_0,u_1) > 0$, $V(v_1) \neq V(u_2)$ and $P(v_1, v_2) > 0$. If $\eta(u_1)$ is such that $P(\eta(u_1),v_1) > 0$, this path becomes a genuine path in $\cG$ of $\cP$-length $3$, with $K(\frp) = \{u_0, (u_1,v_1), (u_1,v_2) \}$ unless $\eta(u_1) \in \{ v_1, v_2 \}$ in which case one needs to discard the corresponding vertex. Define $A(\frp) := \sum_{x \in K_{\frp}} A_x$. Then by construction vertices which are in $K_{\frp}$ must have disjoint subquasi-trees, so conditional on $\frp \in \cG$ $A(\frp)$ is a sum of independent Bernoulli variables. Let us lower bound the number of these variables. Since the path is self-avoiding $\frp$ contains $r+1$ vertices. Supposing $k$ is the number of long-range edges crossed by $\frp$, $r+1 - 2k$ vertices do not belong to a long-range edge of $\frp$. $k+1$ is also the number of small-range components crossed by $\frp$. Now every other of these small-range component is isomorphic to a component of $V_1$, which from assumption \ref{hyp:cc3} contains at least three vertices. By the self-avoiding property, there is at least one vertex of such a small-range component which is in $K_{\frp}$. Hence 
	\begin{equation*}
		m := \abs{K_{\frp}} \geq \max (r+1 - 2k, \lfloor (k+1)/2 \rfloor) \geq \frac{r}{5} - 1 \geq 1.
	\end{equation*}
	
	On the other hand, for all $q \in (0,1)$ Proposition \ref{prop:escape_gumbel} and Remark \ref{rk:escape_center} imply the existence of a value of $q_0 > 0$, independent of $n$, such that conditional on $x \in \cV$, the Bernoulli variable $A_{x}$ has parameter lower bounded by $q$. Thus for any path $\frp$ in $\scU$, conditional on it being in $\cG$, $A(\frp)$ dominates a Binomial$(m,q)$. From Chernoff's bound, if $Z$ is such a binomial variable, for any $\alpha, \theta > 0$
	\begin{align*}
		\bP \sbra{ Z < \alpha m} &= \bP \sbra{ e^{-\theta Z} > e^{-\theta \alpha m}} \\
		&\leq e^{\theta \alpha m + \log \bE \sbra{e^{-\theta Z}}} \\
		&= e^{m (\theta \alpha  + \log (1-q + q e^{-\theta}))}.
	\end{align*}
	The base of the exponential can be made arbitrarily small by letting $q \rightarrow 1$, $\theta \rightarrow \infty$ and $\alpha \rightarrow 0$, namely for all $\e \in (0,1)$, for all $q > 1 - \e$,  there exists $\alpha \in (0,1)$ such that for all $m \geq 0$,
	\begin{equation}
		\bP \sbra{ Z < \alpha m} \leq \e^{m}.
	\end{equation}

	Hence, combining the previous observations, for all $\e > 0$ one can find $q_0 > 0$ and $\alpha > 0$, all independent of $n$, so that for all self-avoiding path $\frp$ of length $r$ starting from $\iota(O)$ in $\scU$.
	\begin{equation}\label{eq:paths}
		\bP \cond{A(\frp) < \alpha r}{\iota(O), \iota(\eta(O)), \frp \in \cG} \leq \e^{r}.
	\end{equation}
	Since the number of self-avoiding paths of $\cP$-length $r$ in $\cG$ is bounded by $O(\Delta^{r})$ thanks to \eqref{eq:ball_growth}, we can bound
	\begin{align*}
		&\bP \cond{\exists \frp \subset \cG: \abs{\frp}_\cP = r, A(\frp) < \alpha r}{O, \eta(O)} \\ 
		\qquad &= \bE \cond{\sum_{\frp \subset \scU} \II_{\frp \subset \cG, \abs{\frp}_{\cP} = r} \ \bP \cond{A(\frp) < \alpha r}{\iota(O), \iota(\eta(O)), \frp \in \cG}}{O, \eta(O)} \\
		&\qquad \leq O((\Delta \e)^{r})
	\end{align*}
	where  $\frp \subset \scU, \cG$ denotes a self-avoiding path of $\scU$ or $\cG$ respectively. Thus given $C > 0$, $A(\frp) \geq \alpha r$ holds simultaneously for all paths with probability $1 - (\e \Delta)^{r} = 1 - e^{-C r}$ choosing $\e$ and then $q_0, \alpha$ accordingly. This almost proves the result, except that the vertices considered in the random variables $A(\frp)$ may not be on the path $\frp$. However even if the proportion $\alpha r$ of vertices with lower bounded escape probability may lie partly outside $\frp$, this consists of vertices which can be reached in at most one step from $\frp$. Hence up to modifying the constant $q_0$ we can in the end deduce that every path $\frp$ contains a proportion $\alpha r$ of vertices with lower bounded escape probability.
\end{proof}

\begin{remark}\label{rk:escape_notonpaths}
	The proof actually shows a further property: with probability $1 - e^{-C r}$ every path $\frp$ contains a constant proportion of vertices for which the probability to escape to infinity outside $\frp$ is lower bounded. This can be useful to argue that the chain is unlikely to follow a given path. 
\end{remark}

\subsection{Lower bound on escape probabilities}

The proof relies on a comparison with a reversible chain, for which we have a uniform lower bound on escape probabilities thanks to the analysis done in \cite{cutoff_reversible} (see Prop. 4.1 and Remark 4.3). For this comparison to be valid, we will need Assumption \ref{hyp:backtracking}. 
We start with the simpler case where the property of \ref{hyp:backtracking} is satisfied by all vertices, ie there exists $l \geq 1$ such that 
\begin{equation}\label{eq:backtracking_all}\tag{H4'}
	\forall x,y \in V \quad P(x,y) > 0 \Leftrightarrow (I+P)^{l}(y,x) > 0.
\end{equation}
Note this implies in particular that all states are recurrent.

\subsubsection{Comparison of Green Functions}

The comparison will only be possible on a random subgraph of $\cG$. The starting point is Theorem 2.25 in \cite{woess2000random}, which proves that transience of a non-reversible Markov chain can be deduced from that of a reversible chain, provided their transition probabilities and their invariant measures can be compared up to multiplicative factors. This qualitative result can be made quantitative as it based on a comparison between Green functions, which is the consequence of the following lemma.

\begin{lemma}[{\cite[Lemma 2.24]{woess2000random}}]\label{lem:comparison_green}
	Let $H$ be a Hilbert space with inner product $\innerprod{\cdot}{\cdot}$ and let $T_1, T_2$ be two invertible linear operators on $H$ such that:
	\begin{enumerate}[label=(\roman*)]
		\item $T_1$ is self-adjoint,
		\item $\innerprod{T_2 f}{f} \geq \innerprod{T_1 f}{f} \geq 0$ for all $f \in H$.
	\end{enumerate}
	Then for all $f \in H$,
	\begin{equation*}
		\innerprod{T_{1}^{-1} f}{f} \geq \innerprod{T_{2}^{-1} f}{f}.
	\end{equation*}
\end{lemma}

By Remark \ref{rk:escape_center} it suffices to lower bound $\bfP \cond{\tau_{O} = \infty}{\cX_{1/2} = \eta(O)}$. Consider $\tilde{\cG} = \{ O \} \cup \cG_O$ and let $\cP_{0}$ denote the transition kernel of the chain on $\tilde{\cG}$ which goes from $O$ to $\eta(O)$ with probability $1$ at integer time steps, stays at $O$ at half-integer time steps, and otherwise has the transitions of $\cX$. Let
\begin{equation*}
	G_{\cP_{0}}(x,y) := \sum_{t \in \bN} \cP_{0}^{t}(x,y)
\end{equation*}
be the Green function of $\cP_{0}$. $G_{\cP_{0}}(O, O)$ is the expected number of visits to $O$, which has a geometric distribution of parameter $\qesc(\eta(O))$ since the chain automatically goes from $O$ to $\eta(O)$. Hence 
\begin{equation}\label{eq:green_escape}
	G_{\cP_{0}}(O, O) = \bfP \cond{\tau_{O} = \infty}{\cX_{1/2} = \eta(O)}^{-1}
\end{equation}
so it is equivalent to lower bound escape probabilities and upper bound the Green function. Lemma \ref{lem:comparison_green} will allow to do so by comparison, expressing the Green function as $G_{\cP_{0}} = \lim_{z \rightarrow 1} (I - z \cP_{0})^{-1}$.

We will apply the lemma to compare the Markov chain $(\cX_t)_{t \geq 0}$ with basically the simple random walk on a subgraph of the quasi-tree. To this end, we will need to show that the transition probabilities as well as an invariant measure can be lower bounded pointwise by those of the simple random walk, up to multiplicative factor. Since the graph $\cG$ has bounded degrees by \eqref{eq:ball_growth}, the transition probabilities and invariant measure if its simple random walk are both of constant order. By Assumption \ref{hyp:bdd_delta} transition probabilities of $(\cX_t)_{t \geq 0}$ are also of constant order however these transitions may occur in only one direction. From the hypothesis done in this section \eqref{eq:backtracking_all}, this issue is solved by considering a bounded power of a lazy version of the chain. However when it comes to the invariant measure, the counter example from Section \ref{section:counterex} shows that we can expect the infimum of an invariant measure to be $0$, which is why comparison is only possible on a subgraph. Since the invariant measure can be taken up to multiplicative constant, comparison with the simple random walk requires finding a subgraph where the invariant measure is of the same order or larger than the measure of the root. In particular we need to find explicitly one such invariant measure for $(\cX_t)_{t \geq 0}$. This can be done thanks to the tree structure of $\cG$. 

\subsubsection{Invariant measure of $(\cX_{t})_{t \geq 0}$}

The idea is based on Kolmogorov's criterion for reversibility: for a reversible chain, an invariant measure can be determined explicitely from the detailed balance equation, which allows to propagate the value of the measure from a state $x$ to another state $y$ by considering any path between $x$ and $y$. Since the invariant measure, which does not need to be a probability, is determined up to multiplicative constant anyway, this gives an algorithmic way to compute an invariant measure. For a non-reversible chain this procedure does not work as multiple paths between vertices may give different values of the measure (the absence of this phenomenon is precisely Kolmogorov's condition for reversibility). However if there are no multiple paths, in other words if the underlying graph is a tree, and transitions can be reversed, then we can find such a reversible measure. Now in our setting, quasi-trees have a large-scale tree structure arising from the long-range edges, while cycles can only be present in small-range components. We can thus combine invariant measures of small-range components to obtain invariance with respect to transitions inside a component with the previous idea to obtain invariance along long-range edges as well. 

Let $\pi_P$ be an invariant measure of $P$ which puts positive mass on every communicating class. Assumption \eqref{eq:backtracking_all} implies that $P$ has no transient states so $\pi_P(x) > 0$ for all $x \in V$. Extend $\pi_P$ to $\cG$ by setting $\pi_P(x) := \pi_P(\iota(x))$ for all $x$. 
The invariant measure $\nu$ is described in terms of the following weights: for every center $x$, set
\begin{equation}\label{eq:invariant_ax}
	a_x := \frac{\pi_P(\eta(x)) q(\eta(x), x)}{\pi_P(x) q(x, \eta(x))}
\end{equation}
Then define
\begin{equation}\label{eq:def_Z}
	Z(x) := \prod_{i=1}^{k} a_{x_i}
\end{equation}
where $x_1, \ldots x_k = x$ is the unique sequence of centers that lie on any path in $\cG$ from $O$ to $x$, and $Z(O) := 1$. We extend $Z$ to non-center vertices by $Z(x) := Z(x^{\circ})$. 

\begin{lemma}\label{lem:invariant_measure}
	Under Assumption \eqref{eq:backtracking_all} the measure $\nu$ on $\cV$ defined by
	\begin{equation}
		\nu(x) := Z(x) \pi_P(x)
	\end{equation}
	is an invariant measure of the Markov chain $(\cX_t)_{t \geq 0}$.
\end{lemma}

\begin{proof}
	Recall $\cP$ denotes the transition kernel of the Markov chain $(\cX_t)_{t \geq 0}$ and consider $x,y \in \cV$. Consider $x,y \in \cV$. Having $\cP(x,y) > 0$ requires that either $x$ is in the same small-range component as $y$, which occurs if and only if $x^{\circ} = y^{\circ}$, or $x = \eta(y^{\circ})$ is in a lower level, or $x$ is in a higher level, that is $x = \eta(z)$ for $z \neq y^{\circ}$ in the small-range component of $y$. In the second case, \eqref{eq:invariant_ax} implies
	\begin{align*}
		\nu(\eta(y^{\circ})) \cP(\eta(y^{\circ}),y) &= Z(\eta(y^{\circ})) \pi_P(\eta(y^{\circ})) q(\eta(y^{\circ}), y^{\circ}) P(y^{\circ}, y) \\ 
		&= Z(y^{\circ}) \pi_P(y^{\circ}) q(y^{\circ}, \eta(y^{\circ})) P(y^{\circ}, y)
	\end{align*}
	Similarly if $z \neq y^{\circ}$ is in the small-range component of $y$,
	\begin{align*}
		\nu(\eta(z)) \cP(\eta(z),y) &= Z(y^{\circ}) a_{\eta(z)} \pi_P(\eta(z)) q(\eta(z), z) P(z, y) \\ 
		&= Z(y^{\circ}) \pi_P(z) q(z, \eta(z)) P(z,y).
	\end{align*}
	Thus these two cases can be regrouped together as the sum $\sum_{x^{\circ} = y^{\circ}} Z(y^{\circ}) \pi_P(x) q(x, \eta(x)) P(x, y)$. Then using that $p + q \equiv 1$ and the invariance of $\pi_P$ with respect to $P$, we deduce
	\begin{align*}
		\sum_{x} \nu(x) \cP(x,y) &= \sum_{x: x^{\circ} = y^{\circ}} Z(y^{\circ}) \pi_P(x) p(x, \eta(x)) P(x, y) + Z(y^{\circ}) \pi_P(x) q(x, \eta(x)) P(x, y) \\
		&= Z(y^{\circ}) \sum_{x \in V(y)} \pi_P(x) P(x,y) \\
		&= Z(y^{\circ}) \pi_P(y) \\
		&= \nu(y).
	\end{align*}
\end{proof}

\subsubsection{Lower bounding the invariant measure}

Let us rearrange weights in the product \eqref{eq:def_Z} to obtain ratios that involve vertices in the same small-range component. Let $x \in \cV$ and $(x_i)_{i=1}^{k}$ be the unique sequence of centers joining $O$ to $x$, possibly empty. Then set
\begin{equation}\label{eq:def_Z_alternative}
	Z'(x) := \prod_{i=1}^{k-1} \frac{\pi_P(\eta(x_{i+1})) q(\eta(x_{i+1}), x_{i+1})}{\pi_P(x_i) q(x_i, \eta(x_{i}))}.
\end{equation}
taking by convention empty products to be equal to $1$.
By construction
\begin{equation}\label{eq:comp_Z_alternative}
	Z(x) = \frac{\pi_P(\eta(x_1)) q(\eta(x_1), x_1)}{\pi_P(x_k) q(x_k, \eta(x_k))} Z'(x)
\end{equation}
Thus to lower bound $Z(x)$ is essentially suffices to bound $Z'(x)$. The idea now is that the process $Z'(x)$ can be described as a branching random walk on $\bR_{+}^{\ast}$ with symmetric increments. We do not need to dive to much into the theory of branching random walks, so we only define a branching random walk informally as a Markov field defined on a Galton-Watson tree. 

In our case, the underlying tree $\cT$ can be obtained by collapsing all vertices of a small-range component into one node and keeping only long-range edges. We call this tree the skeleton tree of $\cG$. This tree has naturally a root which identifies with $O$ in $\cG$, while other vertices identify with centers. Thanks to this identification, $Z'$ can be seen as a random field on $\cT$. Write $\cT_O$ for the subtree of $\cT$ spanned by centers of $\cG_O$.

\begin{lemma}\label{lem:branching_subtree}
	Suppose $P$ satisfies \eqref{eq:backtracking_all} with constant $l \geq 1$. Then for all realization of $\cG$ there exists a connected subgraph $\cG' \subset \cG_O$ containing $\eta(O)$ such that
	\begin{enumerate}[label=(\roman*)]
		\item all vertices of $\cG'$ are at $\cP$-distance at most $2 l$ from their center,
		\item $Z'(x) = \Theta(1)$ for every center $x \in \cG'$,
		\item with probability $\Theta(1)$, any reversible chain with underlying graph $\cG'$ has uniformly lower bounded escape probabilities.
	\end{enumerate}
\end{lemma}

\begin{proof}
	The starting point is that the tree $\cT$ and the identification with centers of the quasi-tree can be generated in two steps. Suppose for instance that $P_1$ and $P_2$ are irreducible. Then all small-range components of $\cG$ are isomorphic to the unique communicating class of either $P_1$ or $P_2$, so the tree $\cT$ is just a $n$-regular tree. In general the tree $\cT$ can be described as a two-type Galton-Watson tree. For $i = 1, 2$ and $k \geq 0$, let $M_i(k)$ be the number of communicating classes in $V_i$ with exactly $k$ vertices. The tree $\cT$ has the law of the random tree where every vertex of type $i \in \{1 ,2 \}$ gives rise to $k-1$ children with probability $k M_{{\bar{i}}}(k) / n$, all of types $\bar{i}$, where $\bar{i} = 2$ if $i =1$ and vice versa. By Assumption \ref{hyp:cc3}, $M_1(1) = M_1(2) = M_2(1) = 0$, so $\cT$ always contains the deterministic tree that alternates between degree $2$ and degree $3$ vertices. 

	Suppose now $\cT$ is a tree generated according to the previous law. Suppose the type of the root is fixed, which determines the types of all other vertices. Then assign labels to edges of $\cT$, consisting of pairs $(u,v)$ such that: $v \in B_{P}^{+}(u,\infty), v \neq u$. They need to satisfy the two constraints: 1. the labels on edges incident to a same vertex $x$ share the same first coordinate; 2. if this coordinate is $u$, then for all $v \in B_{P}^{+}(u,\infty), v \neq u$ there is an edge incident to $x$ with label $(u,v)$. Because of the first constraint, we extend labels to vertices, calling $u$ the label of $x$. In addition, call the label $(u,v)$ a backbone label if $d_{\SR}(u,v) \leq 2$ or $d_{\SR}(v,u) \leq 2$, which is symmetric in $(u, v)$. Remark that every vertex $x$ has always one backbone label, nay two if $x$ has type $i = 1$. If $\cT$ is a tree with such a labelling of the edges, we can construct a quasi-tree $\cG$ such that $\cT$ is the skeleton tree of $\cG$: labels of vertices define the type $\iota$ of centers, while labels of edges detail how the corresponding small-range components are linked to each other. Backbone labels correspond to vertices which are close to their center: if $d_{\SR}(u,v) \leq 2$ then \eqref{eq:backtracking_all} implies that $d_{\SR}(v,u) \leq 2 l$ so $u,v$ are at bounded distance from each other in both directions.

	We then define weights on labelled edges. If $e$ has label $(u,v)$, set
	\begin{equation}
		b_{e} = b_{(u,v)} := \frac{\pi_P(v) q(v, \eta(v))}{\pi_P(u) q(u, \eta(u))}. 
	\end{equation}
	Note that $b_{(u,v)} = b_{(v,u)}^{-1}$. Given a vertex $x \in \cT$, we have by construction that $Z'(x) := \prod_{i=1}^{k} b_{e_i}$ where $(e_i)_{i = 1}^{k}$ is the unique path of edges from the root to $x$ and empty product is equal to $1$. From what precedes, if $(u,v)$ is a backbone label then $P(u,v)^{k} > 0$ for some $k \in [1, 2l]$. Recall that from Assumption \ref{hyp:bdd_delta} and \ref{hyp:bdd_p} there is a constant lower bound $\delta$ on all transition probabilities. Thus
	\begin{equation*}\label{eq:bdd_weights_b}
		b_{(u,v)} \geq \delta ^{2l+1} =: b_{\min}
	\end{equation*}
	is bounded below uniformly in $n$, using that $\pi_{P}(v) \geq \pi_{P}(u) P^{k}(u,v)$. Consequently backbone labels imply bounded weights uniformly in $n$.

	Now labels can be assigned randomly: suppose $x \in \cT$ has type $i \in \{ 1, 2 \}$ and $k-1$ vertices descendant vertices. Pick a communicating class $\cC$ of size $k$ in $V_i$ uniformly at random, pick a uniform $u \in \cC$ and distribute labels $((u,v))_{v \in \cC \smallsetminus \{ u \}}$ uniformly over the descendant edges of $x$. With this procedure, the quasi-tree obtained from the random labels has the distribution of $\cG$.
	Furthermore, labels are independent over edges that are not incident to a same vertex. 
	
	In particular $Z'_x$ is a product of independent weights over the path from the root to $x$. Furthermore, since every choice in this construction is uniform, the labels are also uniform over all possible pairs $(u,v)$. In particular an edge is given a label $(u,v)$ or $(v,u)$ with the same probability, which implies it is given a weight $b_{(u,v)}$ or $b_{(v,u)}$ with same probability. Since the latter are inverse of each other we deduce that the distribution of the weight assigned to an edge is symmetric, that is $b_{e}$ has the same law as $b_{e}^{-1}$. Then for all $x \in \cT$, $Z'_x$ needs also be symmetric as it is a product of independent weights. 
	
	We conclude the proof by a comparison with a branching process. We can reason conditional on $\cT$, so only labels and the associated weights are random. Consider the power tree $\cT_{O}^{(4)}$ obtained by keeping only the vertices of $\cT_O$ which are at depth a multiple of $k$ from the root, and where vertices are joined by an edge if one is the ancestor of the other in $\cT$. By definition an edge of $\cT_{O}^{(4)}$ identifies to a $4$-tuple of edges in $\cT_O$. Let $\cT''$ be the tree defined as the connected component of $\eta(O)$ of all edges in $\cT_{O}^{(4)}$ which identify to $4$-tuple of edges with backbone labels, and such that the product of weights is above $1$. What precedes shows that except $\eta(O)$, every vertex in $\cT_{O}^{(4)}$ has always $4$ edges with backbone labels, and the product of weights is above $1$ with probability $1/2$. Thus the tree $\cT''$ dominates stochastically a Galton Watson tree whose offspring law is binomial $\mathrm{Bin}(4,1/2)$ and is thus supercritical. This gives back a subtree $\cT' \subset \cT_O$ of the original tree, which is the one we are looking for. The labels on these edges give rise to a set of vertices which span the desired graph $\cG'$. Since all labels on $\cT'$ are backbone labels, these vertices are all at distance at most $2 l$ from their center, and $Z'(x) \geq b_{\min}^4$ for all $x \in \cT'$. It remains to prove the last statement. This is the consequence of the arguments used in the reversible model: from the above comparison with a supercritical Galton-Watson tree, if the latter is infinite, which occurs with probability $\Theta(1)$, then the skeleton tree $\cT'$ is sufficiently branching to make any reversible chain on $\cG'$ have uniformly lower bounded escape probabilities. We refer the reader to Proposition 4.1 and Remark 4.3 in \cite{cutoff_reversible} for details. 
\end{proof}

\begin{proof}[Proof of Lemma \ref{lem:escape_probability} with assumption \eqref{eq:backtracking_all}]
	By \eqref{eq:green_escape}, it suffices to bound the Green function $\cG_{\cP_0}$ at $O$. The measure $\nu$ is easily tweaked to give an invariant measure for $\cP_0$: the graph $\cG_{O}^{\ast}$ is still a quasi-tree which can be seen as having $\cC_O := \{ O \} \cup \, \SR(\eta(O))$ as the small-range component of the root. Then replace $\pi_P$ locally on that small-range component by the invariant measure $\pi_P^{\ast}$ of $\cC_O$. This modification does not affect the field $Z'(x)$ \eqref{eq:def_Z_alternative} thus from \eqref{eq:comp_Z_alternative} $Z(x)$ gets only modified by one factor. In the end we obtain an invariant measure $\nu_0$ for $\cP_0$ given by $\nu_{O}(x) := \pi_P^{\ast}(x)$ if $x \in \cC_{O}$ and otherwise
	\begin{equation*}
		\nu_0(x) := \frac{\pi_P^{\ast}(\eta(x_1)) q(\eta(x_1), x_1)}{\pi_P(x_k) q(x_k, \eta(x_k))} Z'(x) \pi_P(x)
	\end{equation*}
	if $(x_i)_{i=1}^{k}$ is the sequence of centers from $O$ to $x$. In particular, $\eta(x_1) \in \cC_O$ and $x_k \in \BSR(x, +\infty)$. Consider now the graph $\cG' =( \cV', \cE')$ of Lemma \ref{lem:branching_subtree}. For all $x \in \cV'$, $Z'(x) \geq c$ and all vertices of $\cV'$ are at bounded distance from their centers. As was argued in the proof of Lemma \ref{lem:branching_subtree}, this implies that the ratios between the measures of vertices inside the same small-range component can be bounded uniformly in $n$. Thus 
	\begin{equation*}
		\frac{\pi_P^{\ast}(\eta(x_1))}{\nu_{0}(O)} \frac{\pi_P(x)}{\pi_P(x_k)} = \Theta(1).
	\end{equation*}
	Since the invariant measure is defined up to multiplicative constant, let us suppose $\pi_P^{\ast}$ was taken in order to have $\nu_{0}(O) = 1$. Using also the fact that the probabilities $q$ are bounded by \ref{hyp:bdd_p}, we deduce that there exist a constant $c$ independent of $n$ such that for all $x \in \cV'$
	\begin{equation*}
	 	\nu_{0}(x) \geq c.
	\end{equation*}

	On the other hand, let $K$ denote the transition kernel of the simple random walk on $\cG'$. The measure $\mu$ defined by $\mu(x) := \deg_{\cG'} (x)$ for all $x \in \cV'$ in an invariant measure for $K$. Since the graph $\cG$ has degrees bounded by $\Delta$, the entries of both $K$ and $\mu$ can be lower bounded by a constant depending only on $\Delta$. We can thus bound
	\begin{equation*}
		\sup_{x \in \cV'} \mu(x) / \nu_{0}(x) \leq C
	\end{equation*}
	for some other constant $C > 0$ independent of $n$. Consider now the transition kernels $\bar{\cP} := (I + \cP_0)^{l+1} / 2^{l+1}$ and $\bar{K} := (1 - \frac{1}{2 C} \frac{\mu}{\nu_0})I + \frac{1}{2 C} \frac{\mu}{\nu_0} K$, where $l$ is given by the assumption \eqref{eq:backtracking_all}. The measure $\nu_0$ is still invariant for $\bar{\cP}$ but now one also has that $\bar{K}$ is reversible with respect to $\nu_0$. 
	
	Let $z \in (0,1)$. We apply Lemma \ref{lem:comparison_green} to operators $T_1 := (I - z \bar{K})$ and $T_2 := (I - z \bar{\cP})$ on the Hilbert space 
	\begin{equation*}
		\ell^{2}(\cV') := \{ f: \cV' \rightarrow \bR \ | \ \innerprod{f}{f}_{\nu_0} < \infty \}, \quad \innerprod{f}{g}_{\nu_0} := \sum_{x \in \cV'} \nu_0(x) f(x) g(x).
	\end{equation*}
	By \eqref{eq:backtracking_all} every small-range edge of $\cG'$, regardless of the orientation, has positive transition probability under $\bar{\cP}$, bounded below by $(\delta/2)^{l+1}$. Similarly, if $x,y \in \cV'$ and $P(\eta(x),y) > 0$, then $\bar{\cP}(x,\eta(x)) \geq 2^{-l-1} \cP_{0}(x,y) \cP_{0}(y,\eta(x)) \geq (\delta/2)^{l+1}$ so long-range edges too have a lower bounded probability. Thus for all $x,y \in \cV'$
	\begin{equation*}
		\bar{\cP}(x,y) \geq \epsilon \bar{K}(x,y)
	\end{equation*}
	for some constant $\epsilon$. As a consequence $\frac{1}{1- \epsilon}(\bar{\cP} - \epsilon \bar{K})$ defines a transition kernel with invariant measure $\nu_0$. This defines thus a contracting operator on $\ell^{2}(\cV')$, which implies
	\begin{equation*}
		\innerprod{(\bar{\cP} - \epsilon \bar{K}) f}{f}_{\nu_0} \leq (1 - \epsilon)\innerprod{f}{f}_{\nu_0}
	\end{equation*}
	and
	\begin{equation*}
		\innerprod{T_2 f}{f}_{\nu_0} \geq \epsilon \innerprod{T_1 f}{f}_{\nu_0}
	\end{equation*}
	for all $f \in \ell^{2}(\cV')$. The reversibility of $\bar{K}$ implying also the self-adjointness of $T_2$ all the assumptions of Lemma \ref{lem:comparison_green} are satisfied, hence
	\begin{equation*}
		\innerprod{T_{2}^{-1} f}{f}_{\nu_0} \leq \epsilon^{-1} \innerprod{T_{1}^{-1} f}{f}_{\nu_0}.
	\end{equation*}
	Taking $f$ as the indicator function at $O \in \cV'$ we obtain
	\begin{equation*}
		G_{\bar{\cP}}(O,O \ | \ z) \leq \epsilon^{-1} G_{\bar{K}}(O,O \ | \ z)
	\end{equation*}
	with $G_{\bar{\cP}}(O,O \ | \ z) = \sum_{t \geq 0} \bar{\cP}^{t}(O,O) z^{t}$ and similarly for $\bar{K}$. Take the limit $z \rightarrow 1$ from below to obtain
	\begin{equation*}
		G_{\bar{\cP}}(O,O) \leq \epsilon^{-1} G_{\bar{K}}(O,O)
	\end{equation*}
	Now we claim that Green functions of $\bar{\cP}$ and $\cP_0$ can be related as
	\begin{equation*}
		G_{\cP_O}(O, O) \leq 2^{l-1} (l+1) \ G_{\bar{\cP}}(O,O).
	\end{equation*}
	Indeed, $\bar{\cP}$ is obtained by applying the successive operations of "lazification" and taking the power $l+1$. It can be checked that the "lazification" multiplies the Green function by $2$. Then let $(\cY_t)_{t \geq 0}$ be the lazy chain with matrix $(I + \cP_0) / 2$. Taking the power $l+1$ of this chain comes down to consider the chain only at times which are a multiple of $l+1$. Now for all $m \in [0, l]$, for all $t \geq 1$, laziness implies
	\begin{equation*}
		\bP_{O}\sbra{\cY_{(l+1)t} = O} \geq \bP_{O} \sbra{\cY_{(l+1)t-m} = O} / 2^{m}
	\end{equation*}
	so we can average and get
	\begin{equation*}
		\bP_{O}\sbra{\cY_{(l+1)t} = O} \geq \sum_{m=0}^{l} \frac{1}{2^{m} (l+1)}\bP_{O} \sbra{\cY_{(l+1)t-m} = O}
	\end{equation*}
	Then summing over $t$ yields:
	\begin{align*}
		G_{\bar{\cP}}(O,O) &\geq \sum_{m=0}^{l} \frac{1}{2^{m} (l+1)} \sum_{t = -m \mod l+1} \bP_{O}\sbra{\cY_t = O} \\
		&\geq \frac{1}{2^{l} (l+1)} G_{(I + \cP_O) / 2}(O, O) \\
		&= \frac{1}{2^{l-1} (l+1)} G_{\cP_O}(O, O).
	\end{align*}

	Finally it remains to upper bound the Green function of $\bar{K}$. Since it is a reversible chain, Point (iii) of Lemma \ref{lem:branching_subtree} shows there is an event of probability $\Theta(1)$ on which it has uniformly lower bounded escape probabilities which from \eqref{eq:green_escape} implies its Green function is uniformly upper bounded at the root.
\end{proof}

\subsection{General case: $o(n)$ transient or close to transient states}

We now consider the general case where we suppose only \ref{hyp:backtracking}. In this case there can be $o(n)$ vertices around which it not possible to revert the transitions in $l$ steps. Transient states are such vertices, so we call these vertices \emph{close to transient}. There are two issues that prevent the previous proof to apply directly: first if $P$ contains transient states, $\pi_P$ takes the value $0$ so the measure $\nu$ as defined in Lemma \ref{lem:invariant_measure} cannot be invariant, and the chain $(\cX_t)_{t \geq 0}$ is incidentally not irreducible. This problem is fixed by the fact that we need only a "sub-invariant" measure, satisfying $\nu \cP \leq \nu$, to apply the comparison argument, which allows for a loss of mass at transient states. Secondly, even without genuine transient states, the comparison with a reversible chain fails at close to transient states. We will argue that since these are in number $o(n)$, they are sufficiently rare to be discarded. 

\paragraph{Quasi-stationary measures and transient classes}

Our goal is to define a sub-invariant measure which puts mass on transient states if there are any. We call such measures quasi-stationary, by analogy to classical quasi-stationary measures, which describe the limiting behaviour of absorbing Markov chains before absorption. Transitionning along a non-backtracking edge, that is an edge that can only be crossed in one direction, is indeed very similar to absorption. 

Let $(Y_t)_{t \geq 0}$ be a Markov chain on a countable state space $W$ with transition kernel $Q$. A set $A \subset W$ is called absorbing if for all $x \in A$, $Q(x,A) = 1$. Suppose the chain contains a non-empty absorbing set $A$ and write $B$ for its complement. A measure $\mu$ on $B$ is called quasi-stationary if for all $t \geq 0$, for all $y \in B$,
\begin{equation*}
	\mu(y) = \sum_{x \in B} \mu(x) \bP_{x} \sbra{Y_t = y, \tau_{A} > t}
\end{equation*}
 As for classical stationary measures, quasi-stationary measures can be obtain from the Perron-Frobenius theorem. Let $Q_B$ the restriction of $Q$ to $B$: this is a sub-stochastic matrix, with non-negative entries, hence it admits a maximal positive eigenvalue $\lambda \leq 1$, which in turn admits left and right eigenvectors with all positive entries. Since $Q_{B}^{t}(x,y) = \bP_{x} \sbra{Y_t = y, \tau_A > t}$, we deduce that left eigenvectors $\mu$ with positive entries define quasi-stationary measures on $B$. Note also that $\mu Q_B = \lambda \mu \leq \mu$ is a sub-invariant measure.

 \begin{remark}
	Our notion of quasi-stationary measures differs slightly from the classical one which not only intersects but conditions by the event that $\tau_A > t$. 
 \end{remark}

Now, if we suppose that $W$ contains transient and recurrent states, the latter form an absorbing set. A quasi-stationary measure in this setting consists thus of a quasi-stationary measure on the set of transient states, with respect to the set of recurrent states. 

\paragraph{Escape probability}

Let us move back to lower bounding the escape probability. We can in fact prove a slightly stronger result, namely allowing for a positive fraction $\alpha n$ of close to transient states, with $\alpha$ small enough. 
$P$ contains in particular at most $\alpha n$ genuine transient states. If there are such transient states, there exist edges $(x,y)$ between the transient and recurrent classes such that $P(x,y) > 0$ but $P^{k}(y,x) = 0$ for all $k \geq 0$. Call these non-backtracking edges. 
Non-backtracking edges in $G$ obviously give non-backtracking edges in $\cG$. We apply the argument with a version of the Green function $G_{\cP_0}$ \eqref{eq:green_escape} killed at non-backtracking edges. Les $\cG'$ denote the connected component of the root in $\cG$ after deletion of the non-backtracking edges and consider the transition kernel $\cP'_0$ given by the restriction of $\cP_0$ to $\cG'$, which is a sub-Markov operator. The corresponding Green function $G_{\cP'_0}(O,O)$ counts the average number of visits to $O$ before crossing a non-backtracking edge. Crossing a non-backtracking edge of $\cG_O$ obviously prevents the chain from going back to $O$, hence the number of visits to $O$ remains a geometric variable of parameter $\qesc(\eta(O))$ and thus $\qesc(\eta(O)) = G_{\cP'_0}(O, O)^{-1}$.

We thus aim to upper bound this Green function uniformly in $n$. After truncation, the remaining long-range edges in $\cG'$ connect small-range components which are either both isomorphic to a communicating class of $G$, or the component furthest from the root is isomorphic to a transient class. We can apply the same idea as in Lemma \ref{lem:invariant_measure} to find a sub-invariant measure, considering quasi-stationary measures on transient classes. Let $\pi_P$ be a positive combination, with bounded coefficients, of invariant measures on communicating classes and quasi-stationary measures on transient states. Then the exact same formula as in Lemma \ref{lem:invariant_measure} yields a sub-invariant measure.

Next we truncate $\cG'$ a second time to get rid of close to transient states as well. Let $\cG''$ be the subgraph of $\cG'$ spanned by the set $\cV''$ of vertices which do not identify under $\iota$ to such states. By construction, the operator $(I + \cP')^{l} / 2^{l}$ can be compared edge-wise with the simple random walk on $\cG''$. We can thus apply Lemma \ref{lem:comparison_green} provided we have a comparison between sub-invariant measures. This can be done similarly by studying the random field $Z$, the only difference with the previous section is that we need we only look at vertices in $\cV''$. This can be done by modifying the definition of the "backbone labels" in the proof of Lemma \ref{lem:branching_subtree}, adding the requirement that a label $(u,v)$ is a backbone label only if $u$ and $v$ are not close to transient. Provided $\alpha$ is small enough, the expected number of backbone labels can be made strictly larger than one, so that the comparison with a supercritical Galton-Watson tree still holds. One last think to be aware of is that we want a result conditional on $\iota(\eta(O))$, which can be close to transient and thus $\eta(O) \notin \cG''$. In this case the argument is to be applied to another subquasi-tree of $\cG_O$ at bounded distance from $\eta(O)$.

\section[Quasi-tree II: Markovian regeneration structure]{Analysis on the quasi-tree II: Markovian regeneration structure}\label{section:QT2}

Proposition \ref{prop:escape_gumbel} implies in particular that $\qesc(O)$ is a.s. positive and thus the Markov chain $(\cX_t)_{t \geq 0}$ a.s. transient. As a consequence, the shortest path from $O$ to $\cX_t$ eventually has to go through a unique sequence of long-range edges $(\xi_{i})_{i=1}^{\infty}$, which is called the \emph{loop-erased chain}. Among the edges of the loop-erased chain, some have the property to be crossed only once. Thanks to this property, these so-called \emph{regeneration edges} yield a Markov decomposition of the quasi-tree which we will use in the next section to prove concentration of the drift and entropy. The regeneration edges will also be used later to compute the approximate invariant measure $\hat{\pi}$.  

\subsection{Markov renewal processes}\label{subsec:mrp}

We start with general results about Markov renewal processes that will be necessary in the sequel. With the exception of Lemma \ref{lem:exponential_tail_markov}, the entirety of this section was already present in \cite{cutoff_reversible} so we will not provide the proofs that can already be found there. 

\begin{definition}\label{def:M1_property}
	Let $S$ be a countable state space and $E$ a Polish space. Consider a process $(Y,Z) = (Y_k, Z_k)_{k \geq 0}$ taking values in $S \times E$, satisfying for all $k \geq 1$, $x,y \in S, z \in E$, 
	\begin{equation}\tag{M1}\label{eq:M1_property}
		\bP \cond{Y_{k} = y, Z_k = z}{Y_{k-1} = x, Z_{k-1}, \ldots, Y_0, Z_0} = \bP \cond{Y'_1 = y, Z'_1 = z}{Y_0 = x},
	\end{equation}
where $(Y'_1, Z'_1)$ is an independent copy of $(Y_1, Z_1)$. This process is thus a Markov chain with a stronger Markov property that the usual one, in that the time dependence occurs only through the $Y$-coordinate. In particular $Y=(Y_k)_{k \geq 0}$ is a Markov chain on $S$. 

\bigskip

A \emph{Markov renewal process} is a process $(Y_k, T_k)_{k \geq 0}$ is a process taking values in $S \times \bN$ such that $(Y_k, T_{k} - T_{k-1})_{k \geq 0}$ satisfies \eqref{eq:M1_property} and $T_k - T_{k-1} \geq 1$ a.s. for all $k \geq 1$, taking $T_{-1} := 0$. The delay $T_0$ can have arbitrary distribution. We call transition kernels of $(Y,T)$ the family of kernels $(Q_{t})_{t \geq 1}$, where for each $t \geq 1$,
\begin{equation*}
	Q_{t}(x,y) := \bP \cond{Y_1 = y, T_1 = t}{Y_0 = x, T_0 = 0} = \bP \cond{Y_1 = y, T_1 - T_0 = t}{Y_0 = x}.
\end{equation*}
Notice then that $Y$ has transition kernel $Q := \sum_{t \geq 1} Q_t$. 
\end{definition}

\begin{remark}\label{rk:mrp_notation}
	If $(Y,Z)$ satisfies \eqref{eq:M1_property}, the initial pair $(Y_0, Z_0)$ can an have arbitrarily law. As usual, if $\nu$ is a probability measure on $S \times E$ we write 
	\begin{equation*}
		\bP_{\nu} := \sum_{y \in S} \int \bP \cond{\cdot}{Y_0 = y, Z = z} \nu(y, dz).
	\end{equation*}
	On the other hand, the \eqref{eq:M1_property} implies that $(Y_k, Z_k)_{k \geq 1}$ conditional on $Y_0$ is independent of $Z_0$. Thus if we are interested in a quantity that is measureable only with respect to $(Y_k, Z_k)_{k \geq 1}$, we will slightly abuse notation by writing $\bP_{u} := \bP \cond{\cdot}{Y_0 = u}$ (and write similarly for expectation), and by extension $\bP_{\mu} := \sum_{y \in S} \mu(y) \bP_{y}$ for a measure $\mu$ on $S$. In particular, note that if $\mu$ is an invariant measure for $Y$, the process $(Y_k, Z_k)_{k \geq 1}$ becomes stationary under $\bP_{\mu}$. In the sequel if $Y$ is positive recurrent, we only consider invariant measures which are probability distributions.
\end{remark}

\begin{lemma}\label{lem:exponential_tail_markov}
	Let $S$ be a countable set and consider a process $(Y_k, Z_k)_{k \geq 0}$ taking values in $S \times \bR_{+}$ which satisfies \eqref{eq:M1_property}. Suppose there exists $\alpha \in (0,1], \theta > 0$ such that $\bE \sbra{e^{\theta Z_{0}^{\alpha}}} < \infty$ and $\max_{u \in S} \bE_{u} \sbra{e^{\theta Z_{1}^{\alpha}}} < \infty$. Then for all $k \geq 0$,
	\begin{equation}
		\bE \sbra{e^{\theta \left(\sum_{i=0}^{k} Z_i \right)^{\alpha}}} \leq \left( \max_{u \in S} \bE_{u} \sbra{e^{\theta Z_{1}^{\alpha}}} \right)^{k}.
	\end{equation}
	As a consequence, there exists $C, C' > 0$ depending on $\theta$, $\bE \sbra{e^{\theta Z_{0}^{\alpha}}}$ and $\max_{u \in S} \bE_{u} \sbra{e^{\theta Z_{1}^{\alpha}}}$ such that for all $k \geq 0$
	\begin{equation}
		\bP \sbra{\sum_{i=0}^{k} Z_i > C k^{1/\alpha} } \leq e^{-C' k}.
	\end{equation}
\end{lemma}

\begin{proof}
	Let $\alpha \in (0,1], \theta > 0$ such that $\bE \sbra{e^{\theta Z_{0}^{\alpha}}} < \infty$ and $\max_{u \in S} \bE_{u} \sbra{e^{\theta Z_{1}^{\alpha}}} < \infty$.
	Let $(\cF_k)_{k \geq 0}$ denote the standard filtration of $(Y,Z)$.  Since $\alpha \leq 1$, $(x+y)^{\alpha} \leq x^{\alpha} + y^{\alpha}$ for all $x,y \geq 0$. Consequently Markov's property used with \eqref{eq:M1_property} implies
	\begin{align*}
		\bE \sbra{e^{\theta \left(\sum_{i=0}^{k} Z_{i} \right)^{\alpha}}} &= \bE \sbra{e^{\theta \left(\sum_{i=0}^{k-1} Z_{i} \right)^{\alpha}} \, \bE \cond{e^{\theta Z_{k}^{\alpha}}}{\cF_{k-1}}} \\
		&= \bE \sbra{e^{\theta \left( \sum_{i=0}^{k-1} Z_{i} \right)^{\alpha}} \, \bE_{Y_{k-1}} \sbra{e^{\theta (Z'_{1})^{\alpha}}}},
	\end{align*}
	where $Z'_1$ denotes an independent copy of $Z_1$. Thus
	\begin{equation*}
		\bE \sbra{e^{\theta \left( \sum_{i=0}^{k} Z_i \right)^{\alpha}}} \leq \max_{u \in S} \bE_{u} \sbra{e^{\theta Z_{1}^{\alpha}}} \bE \sbra{e^{\theta \left( \sum_{i=0}^{k-1} Z_i \right)^{\alpha}}}
	\end{equation*}
	from which the first inequality follows by induction.
	The second is an application of Chernoff's bound.
\end{proof}

The next results specify to the setting where the process $Z$ takes integer values. In particular we state analogs of classical renewal theorems in the context of a Markov renewal process $(Y,T)$. The following generalizes the so-called elementary renewal theorem.

\begin{proposition}\label{prop:lln_mrp}
	Let $(Y,T)$ be a Markov renewal process with state space $S$ such that $Y$ is positive recurrent with stationary distribution $\mu$ and $\max_{u \in S} \bE_u \sbra{T_1} < \infty$. Given $t \geq 0$ let
	\begin{equation*}
		N_t := \sup \{k \geq 0, T_k \leq t \}.
	\end{equation*}
	Then a.s.
	\begin{equation}
		\lim_{k \rightarrow \infty} \frac{T_k}{k} = \bE_{\mu} \sbra{T_1}
	\end{equation}
	and
	\begin{equation}\label{eq:lln_mrp}
		\lim_{t \rightarrow \infty} \frac{N_t}{t} = \frac{1}{\bE_{\mu}\sbra{T_1}}.
	\end{equation}
\end{proposition}

The following Proposition is an analog of the classical renewal theorem, with a quantitative bound on the speed of convergence. It is not necessary for the asymptotic analysis on the quasi-tree but will be used to essentially compute the annealed laws of $\cX$ and $X$ and obtain the value of the limiting measure $\hat{\pi}$ in Proposition \ref{prop:nice_approx}. 

\begin{proposition}\label{prop:markov_renewal_thm}
	Let $(Y,T)$ be a Markov renewal process with state space $S$, such that $\max_{u \in S} \bE_{u} \sbra{T_1} < \infty$ and $Y$ is positive recurrent with stationary distribution $\mu$, irreducible and aperiodic. Let $Q_{t}(x,y) := \bP \cond{Y_1 = y, T_1 - T_0 = t}{Y_0 = x}$ for all $t \in \bZ$ and suppose
	\begin{equation*}
		\alpha := \min_{x \in S} \sum_{\substack{t \in \bZ \\ y \in S}} \left( Q_{t}(x,y) \wedge Q_{t+1}(x,y) \right) > 0.
	\end{equation*}
	Then for all probability distribution $\nu$ on $S \times \bN$, 
	\begin{equation}
		\sum_{y \in S} \abs{\bP_{\nu}\sbra{\exists k \geq 0: Y_k = y, T_k = t} - \frac{\mu(y)}{\bE_{\mu}\sbra{T_1}}} \xrightarrow[t \rightarrow \infty]{} 0.
	\end{equation}
	More precisely, given $\e \in (0,1)$, let $K_{\nu}(\e)$ be the minimal integer such that
	\begin{equation*}
		\bP_{\nu} \sbra{T_0 > K_{\nu}(\e)} \leq \e.
	\end{equation*}
	There exists $C(\e) > 0$ such that for all probability distribution $\nu$ on $S \times \bN$,
	\begin{equation*}
		\sum_{y \in S} \abs{\bP_{\nu}\sbra{\exists k \geq 0: Y_k = y, T_k = t} - \frac{\mu(y)}{\bE_{\mu}\sbra{T_1}}} \leq \e
	\end{equation*}
	for all
	\begin{equation}\label{eq:tmix_renewal}
		t \geq \frac{C(\e)}{\alpha} \left( \tmix^{(Y)}(\e) \max_{u \in S} \bE_{u} \sbra{T_1} + K_{\nu}(\e) + \frac{\bE_{\mu} \sbra{T_{1}^{2}}}{\bE_{\mu} \sbra{T_1}}  \right)^{2} \max_{v} \bE_{v} \sbra{T_1}
	\end{equation}
	where $\tmix^{(Y)}(\e)$ denotes the $\e$-mixing time of $Y$.
\end{proposition}

The last result we need is a variance bound for Markov chains. Since we will only deal here with Markov chains with constant order mixing time, what we require is the following. It is obtained by combining Propositions 5.4 and 5.5 in \cite{cutoff_reversible}.

\begin{proposition}\label{prop:variance_markov}
	Let $S=S(n)$ be countable and $E$ a Polish space. Let $(Y_k, Z_k)_{k \geq 0}$ be a process on $S \times E$ satisfying \eqref{eq:M1_property}. Suppose that $Y$ is irreducible, positive recurrent, with invariant measure $\mu$ and has mixing time $\tmix(1/4) = O(1)$ as $n \rightarrow \infty$.
	Let $(f_i)_{i \geq 1}$ be a family of functions such that for all $i \geq 1$, $f_i: S \times E \rightarrow \bR$ and $\max_{u \in S} \bE_{u} \sbra{f_i (Y_1, Z_1)^{2}} < \infty$. Then there exists a constant $C > 0$ such that for all $k \geq 1$
	\begin{equation*}
		\Var_{\mu} \sbra{ \sum_{i=1}^{k} f_{i}(Y_i, Z_i) } \leq C \sum_{i=1}^{k} \Var_{\mu} \sbra{f_{i}(Y_1, Z_1)}.
	\end{equation*}
\end{proposition}

The previous result can be applied to the case of a Markov renewal process $(Y_k, T_k)_{k \geq 0}$, to get that 
\begin{equation}\label{eq:variance_mrp}
	\Var_{\mu} \sbra{T_k} = O( k \Var_{\mu} \sbra{T_1}).
\end{equation}

\subsection{Regeneration structure}

Let us come back to the setting of quasi-trees.

\begin{definition}
	A time $t \in \bN / 2$ is called a regeneration time if $(\cX_{t-1/2}, \cX_{t})$ is a long-range edge crossed for the first and last time at time $t$.
\end{definition}

\begin{lemma}
	A.s. there exists an infinity of regeneration times.
\end{lemma}
The proof is similar to that of Lemma 3.3 in \cite{lyons1996biased}. It can also be deduced from the tail estimates proved later on.

Frow now on, let $T_0 := 0, Y_0 := \iota(O), L_0 := 0$ and $(T_k)_{k \geq 1}$ be the sequence of successive regeneration times of $\cX$. For all $k \geq 1$, define 
\begin{equation*}
	Y_k = \iota(\cX_{T_{k} - 1/2}) \qquad L_k := d(O, \cX_{T_k}).
\end{equation*}
$(Y_{k})_{k \geq 1}$ is the Markov chain that dictates the law of the environments between successive regeneration times. To describe this Markov chain, introduce the measure 
\begin{equation}\label{eq:Qu}
	\bQ_{u} = \bP \cond{\cdot}{\cX_{1/2} = \eta(O), \iota(O) = u, \tau_{O} = \infty }.
\end{equation}

\begin{remark}\label{rk:comparison_Qu}
	By Lemma \ref{lem:escape_probability}, there exists a constant $q_0 > 0$ such that for all $u \in V$,
	\begin{equation*}
		\bP \cond{\tau_{O} = \infty}{\cX_{1/2} = \eta(O), \iota(O) = u} \geq \bE \cond{\qesc(\eta(O))}{\iota(O) = u}  \geq q_0. 
	\end{equation*}
	which implies 
	\begin{equation*}
		\bP \cond{\cdot, \tau_O = \infty}{\cX_{1/2} = \eta(O), \iota(O) = u} \leq \bQ_{u} \sbra{\cdot} \leq q_0^{-1} \bP \cond{\cdot}{\cX_{1/2} = \eta(O), \iota(O) = u}.
	\end{equation*}
	Consequently the laws of the quasi-tree $\cG$ under $\bP \cond{\cdot}{\iota(O) = u}$ and $\bQ_{u}$ are within constant factors of each other for events that do not contradict $\tau_O = \infty$. In particular, up to a change of constant the uniform lower bounds of Proposition \ref{prop:escape_gumbel} and \ref{prop:escape_dense} for the random variables $\qesc(O), \qesc(\eta(O))$ also hold under the probability $\bQ_{u}$, for any $u \in V$. Furthermore, the law of the centers added to the quasi-trees remains essentially uniform in the sense that for all  $v \in V$ such that $V(v) \neq V(u)$
	\begin{equation}\label{eq:etaO_unif}
		\bQ_u \sbra{\iota(\eta(O)) = v} = \Theta(1/n).
	\end{equation}
	Indeed by what precedes
	\begin{equation*}
		\bQ_{u} \sbra{\iota(\eta(O)) = v} \leq q_0^{-1} \bP \cond{\eta(O) = v}{\iota(O) = u} = q_0^{-1} / n
	\end{equation*}
	and 
	\begin{align*}
		\bQ_{u} \sbra{\iota(\eta(O)) = v} &\geq \bE \cond{\II_{\iota(\eta(O)) = v}\II_{\tau_O = \infty}}{\iota(O) = u, \cX_{1/2} = \eta(O)} \\
		&\geq \bE \cond{\II_{\iota(\eta(O)) = v} \bE \cond{\qesc(\eta(O))}{\iota(O), \iota(\eta(O))}}{\iota(O) = u} \\
		&\geq q_0 / n.
	\end{align*}
\end{remark}

\begin{remark}
	Section \ref{subsec:mrp} only considered integer valued processed for the time component of Markov renewal processes. To apply the results of this section we will thus implicitely identify regeneration times with an integer-valued process. 
\end{remark}

The following lemma is the analog of Lemma 3.6 in \cite{hermon2020universality} and is proved in a similar way.

\begin{lemma}\label{lem:markov_decomposition}
	\begin{itemize}
		\item The sequences $(Y_k, T_k)_{k \geq 1}$ and $(Y_k, L_k)_{k \geq 1}$ are Markov renewal processes. 
		\item The sequence $\left(Y_{k+1}, \cG_{\cX_{T_k}} \smallsetminus \cG_{\cX_{T_{k+1}}}, (\cX_{t})_{T_{k} \leq t < T_{k+1}} \right)_{k \geq 0}$ is a Markov chain whose transition probabilities only depend on the first coordinate $Y_k$, ie the law of this triplet at time $k+1$ conditional on time $k$ is only measurable with respect to $Y_k$.
		\item For all $k \geq 1$, conditional on $Y_k$, the pair $(\cG_{\cX_{T_{k}}}, (\cX_{t})_{t \geq T_k})$ has the law of $(\cG_O, \cX)$ under the probability $\bQ_{Y_k}$. 
	\end{itemize}
\end{lemma}

\begin{remark}
	Note that although $(Y,T)$ starts at time $0$, it is a Markov renewal process from time $1$, and the delay is thus $T_1$. However by the lemma, under the measure $\bQ_{u,v}$, for any $u,v \in V$ such that $V(u) \neq V(v)$, $T_1$ has distribution given by a transition probability and thus $(Y,T)$ becomes a Markov renewal process already from time $0$, with $0$ delay. We will use this observation often in the sequel. 
\end{remark}

A non-negative random variable $Z$ has \emph{stretched exponential tail} if there exists $\alpha \in (0,1)$ such that $\bE \sbra{\exp{ Z^{\alpha}}} < \infty$. A stretched exponential tail implies the existence of polynomial moments of all orders. 

\begin{lemma}\label{lem:regeneration_tails}
	For all $u \in V$, under both measures $\bQ_{u}$ and $\bP_O \cond{\cdot}{\iota(O) = u}$: 
	\begin{itemize}
		\item $L_1$ has exponential tail,
		\item $T_1$ has stretched exponential tail.
	\end{itemize}
	Furthermore, given $L,t \geq 0$, let $\tT_{1}^{(L)}$ denote the first time a regeneration occurs outside the ball $\BLR(O,L)$ and $U_{t}(L) = \sum_{s \leq t} \II_{\cX_s \in \BLR(O,L)}$ the total time spent by the chain in $\BLR(O,L)$ before $t$. There exists $\alpha \in (0,1]$ such that for all $L, m, k \in \bN, m \geq 1$, for all $x \in \BLR(O,L)$ such that $d_{\LR}(O, x) = L$, conditional on either $\cX_0 = x$ or $\cX_{1/2} = x$
	\begin{equation*}
		\bP \cond{\tT_{1}^{(L)} > k + m}{\BLR(O,L), U_{k+m}(L) \leq m} \leq e^{-k^{\alpha}}.
	\end{equation*}
\end{lemma}

\begin{remark}
	Proposition 2.2 in \cite{aidekon2010large} shows that in the case of a randomly biased random walk on a Galton-Watson tree, stretched exponential can be the right order of magnitude for the tail of the first regeneration time. From the proof, it can be seen that this slowdown is precisely caused by the same trapping phenomenon as mentionned in Remark \ref{rk:trapping}, when the walk can move from a part of the tree where it goes very fast to infinity to a part where it tends to go back to the root. This suggests that for some choices $P_1, P_2$ and $p$, stretched exponential tails give the right tail for the regeneration times in our model too.
\end{remark}

\begin{proof}
	For the tails, Remark \ref{rk:comparison_Qu} shows it suffices to prove the result for the law $\bP_O \cond{\cdot}{\iota(O)}$. Suppose $\iota(O)$ is fixed. The proof of the exponential tail for regeneration levels is taken from that of Lemma 4.2 in \cite{dembo2002large}. It actually proves exponential tails for a stronger notion of regeneration, which looks at times when a long-range distance is reached for the first and last time. Namely let 
	\begin{equation*}
		\tT_1 := \inf \{ t \in \bN/2 \ |  \forall s < t: d_{\LR}(O, \cX_{s}) < d_{\LR}(O, \cX_{t}), \forall s \geq t: d_{\LR}(O, \cX_{s}) \geq d_{\LR}(O, \cX_{t}) \}.
	\end{equation*}
	and $\tL_1 := d_{\LR}(O, \cX_{\tT_1})$. Obviously $T_1 \leq \tT_1$ and $L_1 \leq \tL_1$, so it suffices to prove an exponential tail for $\tL_1$. 
	
	Suppose the chain is started at $O$. Consider the "ladder times" $(S_i)_{i = 0}^{K}$, where $K \in \bN \cup \{ \infty \}$ is a random integer, possibly infinite, defined recursively as follows. Let $S_0 := 0$, and define $\tau_{0}$ as the first return time to $O$. If $\tau_0 = \infty$ set $K := 0$. Otherwise, let $M_0 := \max \{d_{\LR}(O, \cX_{t}), 0 \leq t \leq \tau_0 \}$ be the maximal distance reached until $\tau_0$ and $S_1 := \min \{ t \in \bN/2: d_{\LR}(0, \cX_t) > M_0$ \}. Define the rest of the sequence recursively: let $\tau_i$ the first return time to level $d_{\LR}(0, \cX_{S_{i}}) - 1$. If $\tau_i = \infty$, $K := i$, and $S_{j} : =\infty$ for all $j \geq i+1$, otherwise set $M_i := \max \{ d_{\LR}(O, \cX_t), 0 \leq t \leq \tau_{i} \}$ and $S_{i+1} := \min \{t \in \bN/2: d_{\LR}(O, \cX_t) > M_i \}$. Notice that $S_K = \tT_1$ if $K \in (0, \infty)$ and $\tT_1 = 1/2$ if $K = 0$. Consequently for any $C > 0$, if $l \geq 1$ is large enough,
	\begin{align*}
		&\bP_O \cond{\tL_1 > C l}{\iota(O)} \leq \bP_O \cond{K < \infty, d(O, \cX_{S_K}) > C l}{\iota(O)} \\
		&\qquad = \bP_O \cond{K < \infty, \sum_{i=1}^{K} \left( d_{\LR}(O, \cX_{S_i}) - d_{\LR}(0, \cX_{S_{i-1}}) \right) > C l}{\iota(O)} \\
		&\qquad = \sum_{k=1}^{\infty} \bP_O \cond{\sum_{i=1}^{k} \left( d_{\LR}(O, \cX_{S_i}) - d_{\LR}(0, \cX_{S_{i-1}}) \right) > C l, \{\tau_{i} < \infty\}_{i=1}^{k-1}, \tau_k = \infty}{\iota(O)} \\
		&\qquad \leq \bP_O \cond{K > l}{\iota(O)} + \bP_O \cond{ \sum_{i=0}^{l-1} Z_i > C l}{\iota(O)}.
	\end{align*}
	where we write $Z_0 := d_{\LR}(O, \cX_{S_1}) \II_{\tau_0 < \infty}$, $Z_i := \left( d_{\LR}(O, \cX_{S_{i+1}}) - d_{\LR}(0, \cX_{S_{i}}) \right) \II_{\tau_i < \infty}$ for all $i \geq 1$. To prove this is exponentially small in $l$, it thus suffices to establish that $K$ has exponential tail and that for fixed $l \geq 1$, the second probability in the right-hand side is also exponentially small in $l$. 
	
	It was already observed in Remark \ref{rk:comparison_Qu} that $\bP \cond{\tau_{O} = \infty}{\cX_{1/2} = \eta(O), \iota(O)} = \Theta(1)$. We deduce that $K$ is stochastically dominated by a geometric variable of constant parameter and thus has exponential tail, even conditional on $\iota(O)$. 
	
	On the other hand, by construction, for all $i \geq 1$ if $S_i < \infty$ the trajectory $(\cX_t)_{S_i \leq t < \tau_{i}}$ is entirely contained in the subquasi-tree $\cG_{\cX_{S_i}}$. Furthermore, $d_{\LR}(O, \cX_{S_{i+1}}) = M_i + 1$ is determined by this trajectory. Thus for $i \geq 1$, $Z_i = (M_i - M_{i-1}) \II_{\tau_i < \infty}$ depends only on $\eta(\cX_{S_i - 1/2})$ and $\cG_{\cX_{S_i}}$. Letting $\tY_i := \iota(\eta(\cX_{S_i - 1/2}))$ for $i \geq 0$, we can deduce that if $i \geq 1, l \geq 1$ and $S_i < \infty$,
	\begin{align*}
		&\bP_O \cond{Z_i = l}{\cG \smallsetminus \cG_{\cX_{S_i}}, (\cX_s)_{s \leq S_i}} \\
		&= \bP \cond{\tau_{O'} < \infty, M_{0}' = l}{\cX_{1/2} = \eta(O'), \iota(O') = \tY_i} 
	\end{align*}
	with $(\cG', O', M_0')$ an independent copy of $(\cG, O, M_0)$. 
	We do not prove this in detail, which can be established as Lemma \ref{lem:markov_decomposition}. See also the proof of Lemma 4.4 of \cite{dembo2002large}. If we extend the process $(\tY, Z)$ by considering that $\tY_i$ takes an arbitrary value $\dagger$ if $S_i = \infty$ and consider that $Z_{i+1} = 0$ a.s. given $\tY_i = \dagger$, the previous equation implies in particular that $ \left(\tY_i,  Z_i \right)_{i \geq 0}$ is a Markov chain that satisfies the \eqref{eq:M1_property} property. 
	Furthermore for all $l \geq 1$,
	\begin{align*}
		&\bP \cond{\tau_{O} < \infty, M_0 \geq l}{\cX_{1/2} = \eta(O), \iota(O)} \\
		&\leq \bP \cond{\exists t > s \geq l: d_{\LR}(\eta(O),\cX_{s}) = l-1, \cX_t = O}{\cX_{1/2} = \eta(O), \iota(O)}.
	\end{align*}
	This is an annealed probability of backtracking over a long-range distance $L$. Let us here anticipate on the sequel and use Lemma \ref{lem:typical_paths_QT} which will prove this probability is exponentially small in $L$. With what precedes, this proves that there exists constants $C', \theta > 0$ such that for any possible state $u$ of the chain $\tY$, $\bE \cond{e^{\theta Z_1}}{\tY_0 = u} < C'$. Exponential moments of $Z_0$ can be bounded similarly. The reader will then have recognized the condition to apply Lemma \ref{lem:exponential_tail_markov}, which proves that for some constant $C >0$, $\bP_O \cond{ \sum_{i=0}^{l-1} Z_i > C l}{\iota(O)}$ is exponentially small in $l$, with all constants being independent of $n$.
	
	\medskip

	Let us move to the proof of stretched exponential tails for $T_1$. It is inspired from the proof of Prop. 2.2 in \cite{aidekon2010large}. For ease of notation we will here keep the conditionning by $\iota(O)$, but all probabilities below should be considered conditionning on it. Given $k,t \geq 0$, let $\tau_k$ denote the hitting time of level $k$ and $K_t$ denote the time spent before $t$ at a center, at half-integer times steps. For any $C> 0$ and $k \geq 0$,
	\begin{align*}
		\bP_O \sbra{T_1 > C k^3} &\leq \bP_O \sbra{T_1 > \tau_{k}} + \bP_{O} \sbra{C k^3 < T_1 \leq \tau_k} \\
		&\leq  \bP_O \sbra{T_1 > \tau_k} + \bP_O \sbra{K_{C k^3} \leq k^3} + \bP_O \sbra{K_{\tau_k} > k^3}.
	\end{align*}  
	The first term can be rewritten as $\bP \sbra{T_1 > \tau_k} = \bP \sbra{L_1 > k }$ which is exponentially small by the first part. 
	For the second term, note that when the chain is at a non-center vertex $x$ at integer times, it has probability $\Theta(1)$ to go to the center $\eta(x)$ by \ref{hyp:bdd_p}. Thus the first hitting time time of a center is stochastically dominated by a geometric variable, even under the quenched law. In turn we deduce
	\begin{equation*}
		\bfP_O \sbra{K_{C k^3} \leq k^3} \leq \bP \sbra{\sum_{i=1}^{k^3} Z_i \geq C k^3},
	\end{equation*}
	where the $Z_i$ are independent geometric variables of positive constant parameter. By Chernoff's inequality, the right-hand side is exponentially small in $k^3$ for a good choice of $C > 0$.
	The third term can be further decomposed as follows. Given a center $x \in \cV$, let $N(x)$ denote the number of visits to $x$ at half-integer times. Given $i \geq 1$, write $z_i$ for the $i$-th distinct center visited by the chain and $M_i$ for the number of distinct centers visited in $\cV_i$ at half-integer times. Then
	\begin{align*}
		\bP_O \sbra{K_{\tau_k} > k^3} &\leq \bP_O \sbra{\begin{array}{c} \text{more than $k^2$ distinct centers are visited} \\ \text{before $\tau_k$ at half-integer times} \end{array} } \\ &\quad + \bP_O \sbra{\exists i \leq k^2: N(z_i) > k}.
	\end{align*}
	By union bound
	\begin{equation*}
		\bP_O \sbra{\begin{array}{c} \text{more than $k^2$ distinct centers are visited} \\ \text{before $\tau_k$ at half-integer times} \end{array} } \leq \sum_{i=1}^{k} \bP_O \sbra{M_i \geq k}.
	\end{equation*}
	Let $\tau_{k}^{(i)}$ denote here the hitting time of the $k$-th distinct center of $\cV_i$. By the strong Markov property
	\begin{align*}
		\bfP_O \sbra{M_i \geq k} &= \bfP_O \sbra{\tau_{k}^{(i)} < \infty} \\
		&\leq \bfE_O \sbra{\II_{\tau_{k-1}^{(i)} < \infty} (1 - \qesc(\cX_{\tau_{k-1}^{(k)}}))}.
	\end{align*}
	Since the centers visited at $\tau_{k-1}^{(i)}$ and $\tau_{k}^{(i)}$ are distinct and at the same level, their subtrees are disjoint and independent. Hence
	\begin{equation*}
		\bP_O \sbra{M_i \geq k} \leq \bE \sbra{ \bfP_O \sbra{\tau_{k-1}^{(i)} < \infty} \left( 1 - \min_{u,v} \bE \cond{\qesc(\eta(O'))}{\iota(O') = u, \iota(\eta(O') = v)} \right)},
	\end{equation*}
	where $(\cG', O')$ is as usual an independent copy of $(\cG, O)$ and $V(u) \neq V(v)$.	By Lemma  \ref{lem:escape_probability}, the expected escape probability can be lower bounded by some constant $q_0 \in (0,1]$ hence by induction
	\begin{equation*}
		\bP \sbra{M_i \geq k} \leq (1-q_0)^{k-1}
	\end{equation*}
	and
	\begin{equation*}
		\bP_O \sbra{\begin{array}{c} \text{more than $k^2$ distinct centers are visited} \\ \text{before $\tau_k$ at half-integer times} \end{array} } \leq k (1-q_0)^{k-1}
	\end{equation*}
	is exponentially small in $k$.
	
	To bound the last term, let
	\begin{equation*}
		\qesc'(x) := \bfP \cond{\tau_{x} = \tau_{\eta(x)} = \infty}{\cX_{1/2} = x}.
	\end{equation*}
	for all center $x \in \cV$. Like $\qesc(x)$, this quantity depends only on $\cG_x$ and $\eta(x)$. Furthermore, if $y$ is a vertex at small-range distance $1$ from $x$, then 
	\begin{align*}
		\qesc'(x) &\geq p(x, \eta(x)) P(x,y) \qesc(y) \\
		&\geq \delta \qesc(y),
	\end{align*}
	thus lower bounding $\qesc'$ is essentially the same as lower bounding $\qesc$.
	For all $i \geq 1$,
	\begin{align*}
		\bfP_O \sbra{N(z_i) > k} &= \sum_{x \in \cV} \bfE \sbra{\II_{z_i = x, N(x) > k}} \\
		&\leq \sum_{x \in \cV} \bfE \sbra{\II_{z_i = x}(1 - \qesc'(x))^{k}} \\
		&= \sum_{x \in \cV} \bfP_O \sbra{z_i = x} (1 - \qesc'(x))^{k}
	\end{align*}
	To average over the environment, use the "Ulam labelling" of Section \ref{subsec:random_QT}:
	\begin{equation*}
		\bP_O \sbra{N(z_i) > k} \leq \sum_{x \in \scU} \bE \sbra{\II_{x \in \cV} \bfP_O  \sbra{z_i = x} } \bE \cond{ (1 - \qesc'(x))^{k}}{x \in \cV, \cG \smallsetminus \cG_x}
	\end{equation*}
	From Proposition \ref{prop:escape_gumbel}, one can easily obtain that
	\begin{equation*}
		\max_{x \in \scU} \bE \cond{ (1 - \qesc'(x))^{k}}{x \in \cV, \cG \smallsetminus \cG_x} \leq e^{-k^{c}}
	\end{equation*}
	for some constant $c > 0$, hence by union bound
	\begin{equation*}
		\bP_O \sbra{\exists i \leq k^2: N(z_i) > k} \leq k^2 e^{-k^{c}},
	\end{equation*}
	which is stretched exponential, as desired.

	Finally, for the last statement note that if the chains is started precisely at the boudnary of the ball $\BLR(O,L)$, then a regeneration can occur within a number of steps which is bounded in $L$, thus from the lower bound on escape probabilities (Lemma  \ref{lem:escape_probability}) it is easy to lower bound $\bP_{x} \cond{U_{k+m}(L) \leq m}{\BLR(O,L)}$ uniformly in $n, L, m$. Thus the conditionning by $U_{k+m}(L) \leq m$ only modifies upper bounds under the usual law of the quasi-tree by a factor of constant order. The upper bound for the conditional probability then follows from the same arguments as above. 
\end{proof}

\subsection{Mixing of the regeneration chain}\label{subsec:mixing_regeneration}

We now investigate the mixing properties of the Markov chain $Y$ underlying the regeneration process, which we call the regeneration chain. 
Lemma \ref{lem:mixing_regen} below will establish that $Y$ has constant mixing time, allowing us later to get moment bounds similar to the iid case. The lemma actually proves a stronger mixing property of both the chain $Y$ and the time process $T$ that will be used to derive an approximation of the invariant measure on $X$. 

\begin{lemma}\label{lem:mixing_regen}
	Let $(Q_t)_{t \geq 1}$ denote the transition kernels of the Markov renewal process $(Y,T)$ and $Q := \sum_{t \geq 1} Q_t$, after identification of $T$ with a process in $\bN$. Recall that from Hypothesis \eqref{hyp:backtracking} there exists a constant $l \geq 1$ such that the set
	\begin{equation*}
		S = S(l) :=\{ x \in V \ | \ \forall y \in V: P(x,y) > 0 \Leftrightarrow (I+P)^{l}(y,x) > 0 \}.
	\end{equation*}
	satisfies $\abs{S} = 2n - o(n)$.
	\begin{enumerate}[label = (\roman*)]
		\item For all $u \in V$ and $v \in S$,
		\begin{equation}\label{eq:doeblin}
			Q(u,v) = \Theta(1/n).
		\end{equation}
		As a consequence for all $\e \in (0,1)$, the chain $Y$ has mixing time $O_{\e}(1)$ as $n \rightarrow \infty$. 
		\item for all $u \in V$, 
		\begin{equation}\label{eq:lower_bound_alpha}
			\sum_{t \geq 1} \sum_{v \in V} Q_{t}(u, v) \wedge Q_{t+1}(u,v) = \Theta(1).
		\end{equation}
	\end{enumerate}
\end{lemma}

\begin{proof}[Proof of Lemma \ref{lem:mixing_regen}]
	Let $\e \in (0,1)$. The statement about the mixing time of $Y$ is easily deduced from \eqref{eq:doeblin} and the fact that $\abs{S} = 2n - o(n)$. Summing on $v$ yields the following Doeblin's condition for $Y$: there exists a constant $c > 0$ such that
	\begin{equation*}
		\TV {Q(u,\cdot) - Q(u', \cdot)} \leq 1 - c - o(1),
	\end{equation*}
	for all $u, u' \in V$. It is then easy to obtain from Doeblin's condition that $Y$ has mixing time $\tmix^{(Y)}(\e) \leq (c^{-1} + o(1)) \log \e^{-1}$.	

	Let us now prove \eqref{eq:doeblin}. Consider $u,v \in V$ with $v \in S$. By definition for all $t \geq 1$
	\begin{equation*}
		Q_{t}(u,v) = \bQ_{u} \sbra{Y_1 = v, T_1 = t} = \bP \cond{T_1 = t, \iota(\cX_{t - 1/2}) = v}{\tau_{O} = \infty, \iota(O) = u, \cX_{1/2} = \eta(O)}.
	\end{equation*}
	Probabilistic statements below will thus be made with respect to $\bQ_u$ unless stated otherwise. 
	
	Observe that if $V(u) \neq V(v)$, it is possible to realize $Y_1 = v$ with regeneration occurring at the first long-range edge crossed by the chain, which can be reached in half a step from $\eta(O)$ by Assumption \ref{hyp:cc3}. Let $w \in V$ such that $P(w,v) > 0$. From Remark \ref{rk:comparison_Qu}, $\bQ_{u} \sbra{\iota(\eta(O)) = w} = \Theta(1/n)$. Using that transition probabilities are bounded uniformly in $n$ by \ref{hyp:bdd_delta}, \ref{hyp:bdd_p}, we deduce that $\bQ_{u} \sbra{Y_1 = v} \geq \Theta(1/n)$. On the other hand, if $V(u) = V(v)$ the alternation between $V_1$ and $V_2$ forbids a regeneration at the first long-range edge. This is why we need to suppose $v \in S$ to make sure the chain can backtrack and prevent regeneration at the first long-range. Since $S$ occurs with probability $1 - o(1)$ under the uniform law on $V$, centers of the quasi-tree $\cG_O$ also have probability $\Theta(1)$ to be in $S$ under $\bQ_u$. Thus using the above arguments we can lower bound by $\Theta(1/n)$ the probability that the chain goes from $\eta(O)$ to a vertex $y$ with $\iota(y) = v$ using one long-range edge, comes back to $\eta(O)$, which ensures no regeneration occurred on the long-range edge, and returns to $y$, all that within a bounded number of steps. The chain can then escape in $\cG_y$ with probability $\Theta(1)$ by Lemma \ref{lem:escape_probability}, which will imply $Y_1 = v$. 

	\medskip

	The proof of \eqref{eq:lower_bound_alpha} is similar but requires taking time into account. Suppose $u \in V_1$ and let $v \in S \cap V_2$. Our argument is illustrated in Figure \ref{fig:mixing_argument}.
	\begin{figure}
		\centering
		\includegraphics[scale=.7]{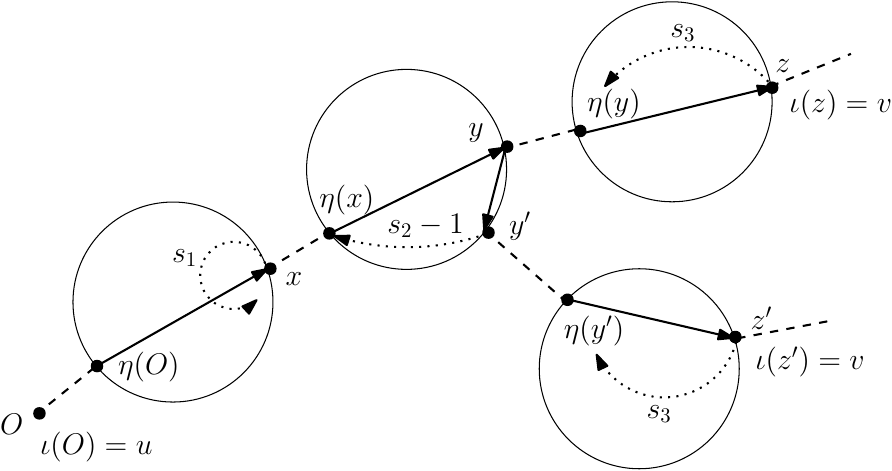}
		\caption{Argument of the proof of Lemma \ref{lem:mixing_regen}}
		\label{fig:mixing_argument}
	\end{figure}
	The idea is to use an intermediary component of $V_1$ to make the regeneration time shift by $1$. Consider the set
	\begin{equation*}
		S_1 := \{ x \in V \cap S \ | \exists y \neq x \in V, s \leq l:  P^{2}(x,y) > 0 \text{ and } P^{s}(y,x) > 0 \}.
	\end{equation*}
	From Assumption \ref{hyp:cc3}, $n - o(n)$ vertices of $S \cap V_1$ are recurrent and contained in a communicating class of size at least $3$. Consequently if $x \in S \cap V_1$, either there exists $y \neq x$ on a loop around $x$ of length less than $l$ such that $P^{2}(x,y) > 0$, in which case $x \in S_1$, or $x$ is in a communicating class $\{x,y,z\}$ of size exactly $3$ with $P^{2}(x,x) = 1$. However in that case $P^{2}(y,z), P^{2}(z,y) > 0$, so $y,z \in S_1$. Thus $S_1$ has size at least $2n/3 - o(n)$ and probability $\Theta(1)$ under the uniform law on $V_1$ or $\bQ_w$ for any $w \in V_2$.

	For any realization of $\cG$, there exists a path $(\eta(O), x, \eta(x), y, \eta(y), z)$ with two long-range edges $(x, \eta(x))$ and $(y, \eta(y))$ whereas $(\eta(O), x$), $(\eta(x), y)$, $(\eta(y), z)$ are small-range edges with distinct endpoints. 
	If $\iota(\eta(O)) \in S$, which occurs with probability $\Theta(1)$, we can suppose that $x$ was chosen so that $P^{s_1}(\iota(x),\iota(x)) > 0$, for some $s_1 \leq l+1$. Similarly, $\iota(\eta(x)) \in S_1$ with probability $\Theta(1)$, in which case we can suppose $y$ was chosen so that there exists a path between $y$ and $\eta(x)$ of length $s_2 \leq l$ and a vertex $y' \neq x$ on this path with $P^{2}(\iota(\eta(x)), \iota(y')) > 0$. 
	Then if $v \in S$ there exists $w \in V_2$ such that $P(w,v) > 0$ and $P(v,w)^{s_3} > 0$ for some $s_3 \leq l$. Conditional on the previous values of $\iota$, $\iota(\eta(y)) = w$ with probability at least $\Theta(1/n)$. Gathering what precedes, make the following observation: from $\eta(O)$ the chain can go to $x$ in half a step and then to $z$ in two more steps, after which it can come back to $x$ in $s_3 + s_2 + s_1$ steps going through $y'$, to eventually return at $z$ in two steps and escape to infinity through $\eta(z)$. We will then have $Y_1 = v, T_1 = t + 1/2$ with $t = 4 + s_1 + s_2 + s_3$. 

	Now instead of $y$, we can apply the same arguments with $y'$, supposing that the chain reaches a vertex $z'$ with $\iota(\eta(y')) = w$ and $\iota(z') =v$. By construction $z'$ is reachable in just one more step from $x$, but the loop around $x$ which passes through $z'$ takes the same number of steps. Thus the previous argument can be applied similarly with just one time step difference to obtain $Y_1 = v, T_1 = t+1 + 1/2$. Summing on the values of $\iota(\eta(O)), \iota(\eta(x))$, we obtain that 
	\begin{equation*}
		\sum_{t \geq 1} Q_t(u,v) \wedge Q_{t+1}(u,v) \geq \Theta(1/n).
	\end{equation*}
	The case where $u \in V_2$ is similar, with one less necessary long-range edge to cross to reach a component of $V_1 \cap S_1$, hence the bound applies to all $u \in V$. Then as $\abs{V_2 \cap S} = n - o(n)$ summing over $v \in V_2 \cap S$ yields \eqref{eq:lower_bound_alpha}. 
\end{proof}

Let $\mu$ denote the stationary distribution of the Markov chain $(Y_k)_{k \geq 0}$ and 
\begin{equation*}
	\bQ_{\mu} := \sum_{u \in V} \mu(u) Q_{u}.
\end{equation*}
In the sequel $\bE_{\bQ_{\mu}}$ denotes the expectation with respect to $\mu$. Later, we use similar notations for variance, covariance, etc.

The fact that $Y$ mixes in $O(1)$ steps will be sufficient for the rest of the analysis of the quasi-tree. Point (ii) in Lemma \ref{lem:mixing_regen} will then be used to obtain the expression of the approximate invariant measure in \eqref{eq:pihat}. The mixing of the regeneration chain will be used conditional on some already revealed parts of the environment, which in turn requires conditionning by a neighbourhood of the root. We did something similar in the reversible case (see Prop. 5.6 in \cite{cutoff_reversible}) however in the more general case, the neighbourhood can trap the chain. This should not be the case however if the neighbourhood is typical. To implement this idea we condition further by the fact that the chain rapidly escapes this neighbourhood. 

\begin{proposition}\label{prop:mixing_annealed_QT}
	Let $\nu$ be the law of $\iota(\eta(O))$ under $\bQ_{\mu}$. For all $v \in V$, 
	\begin{equation}\label{eq:nu_uniform}
		\nu(v) = \Theta(1/n).
	\end{equation}
	Furthermore given $L, t \geq 0$, let $U_{t}(L) = \sum_{s \leq t} \II_{\cX_s \in \BLR(O,L)}$ denote the total time spent by the chain in $\BLR(O,L)$ before $t$. Consider here $\tT_1 = \tT_1(L)$ to be the first regeneration time outside $\BLR(O,L)$, while for $k \geq 2$, let $\tT_{k}$ be the first regeneration time after $\tT_{k-1}$. 
	For all $\e \in (0,1)$ there exists $C(\e)$ such that for all $L, m \geq 0$ and $t \in \bN/2$, for all $x \in \BLR(O,L)$, $d_{\LR}(O,x) = L$, if $t \geq C(\e) m^{2}$ then
	\begin{equation*}
		\sum_{v \in V} \abs{ \bP_{x} \cond{\exists k \geq 0: \tT_k = t, \iota(\cX_t) = v}{\BLR(O,L), U_{t}(L) \leq m} - \frac{\nu(v)}{\bE_{\bQ_{\mu}} \sbra{T_1}}} \leq \e.
	\end{equation*}
	and 
	\begin{equation*}
		\sum_{v \in V} \abs{ \bP \cond{\exists k \geq 0: \tT_k = t, \iota(\cX_t) = v}{\cX_{1/2} = x, \BLR(O,L), U_{t}(L) \leq m} - \frac{\nu(v)}{\bE_{\bQ_{\mu}} \sbra{T_1}}} \leq \e.
	\end{equation*}
\end{proposition}

\begin{proof}
	The first statement is a direct consequence of \eqref{eq:etaO_unif}.

	Let us prove the second statement. Let $\e \in (0,1), m, t \geq 1$. For all $k \geq 1$ let $\tY_k := \iota(\cX_{\tT_{k} - 1/2})$. Apply Proposition \ref{prop:markov_renewal_thm} with the Markov renewal process of regeneration times considered here. Note that by the Markov property of $(Y,T)$ conditionning by $\BLR(O,L), U_{t}(L) \leq m$ yields the same transition kernels as $(Y,T$) and only affects the law of $(\tY_1, \tT_1$). Lemmas \ref{lem:mixing_regen} and \ref{lem:regeneration_tails} imply that the two quantities $\alpha$ and $\max_{u \in S} \bE_{u} \sbra{T_1}$ in this Proposition are bounded uniformly in $n$, as is the mixing time $\tmix^{(Y)}(\e)$ of $Y$ for all $\e \in (0,1)$. Finally, the last statement of Lemma \ref{lem:regeneration_tails} implies that 
	\begin{equation*}
		\bP_x \cond{\tT_1 \geq m + C_1(\e)}{\BLR(O,L), U_{m + C_1(\e)}(L) \leq m} \leq \e 
	\end{equation*}
	for some $C_1(\e) > 0$. This gives the value of the quantity $K_{\nu}(\e)$ considered in Proposition \ref{prop:markov_renewal_thm}, which thus proves that there exists $C(\e)$ such that for $t \geq C(\e) m^2$, 
	\begin{equation*}
		\sum_{u \in V} \abs{ \bP_{x} \cond{\exists k \geq 0: \tT_k = t-1/2, \tY_k = u}{\BLR(O,L), U_{t}(L) \leq m} - \frac{\mu(u)}{\bE_{\bQ_{\mu}} \sbra{T_1}}} \leq \e.
	\end{equation*}
	Then note that conditional on $\tT_k = t - 1/2, \tY_k = u$, $\iota(\cX_{t})$ is distributed as $\iota(\eta(O))$ under $\bQ_u$. Finally the arguments apply in the same way if the chain is started at time $x$ at time $1/2$ instead of $0$. 
\end{proof}

\section[Quasi-tree III: concentration of drift and entropy]{Analysis on the quasi-tree III: concentration of drift and entropy}\label{section:QT3}

In this section we establish "nice properties" for the chain $\cX$, proving in particular concentration of the drift and entropy. 

\subsection{Typical paths in quasi-trees}

We start with the following lemma, which will basically ensure the chain $\cX$ only visits typical vertices by time $t = O(\log n)$, allowing to use the uniform lower bound on escape probabilities.

\begin{lemma}\label{lem:escape_stationary}
	Let $\cA$ be a measurable set of quasi-trees. For all $t \geq 0$,
	\begin{align*}
		\max_{u,v \in V} \bP_O \cond{\bigcup_{s \leq t} \{ \cG_{\cX_s} \in \cA \}}{\iota(O) = u, \iota(\eta(O)) = v} \\ \leq t \max_{u,v \in V} \bP_O \cond{\cG_O \in \cA}{\iota(O) = u, \iota(\eta(O)) = v}. 
	\end{align*}
\end{lemma}

\begin{proof}
	Fix the types of $O, \eta(O)$. Let $z_{0} := O$ and for $i \geq 1$ let $z_i$ be the $i$-th distinct vertex visited by the chain. There are at most $t$ distinct vertices visited up to time $t$, hence
	\begin{align*}
		\bP_O \cond{\bigcup_{s \leq t} \{ \cG_{\cX_s} \in \cA \}}{O, \eta(O)} &\leq \bP \cond{\bigcup_{i \leq t} \{ \cG_{z_i} \in \cA \}}{O, \eta(O)} \\
		&\leq t \max_{i \geq 0} \, \bP \cond{\cG_{z_i} \in \cA }{O, \eta(O)}.
	\end{align*}
	Note that for $x \in \scU$ the probability $\bfP_{O} \sbra{z_i = x}$ is measurable with respect to $\cG \smallsetminus \cG_x$. Thus for all $i \geq 0$, 
	\begin{align*}
		\bP_O \cond{\cG_{z_i} \in \cA }{\iota(O),\iota(\eta(O))} &= \sum_{x \in \scU} \bP_{O} \cond{z_i = x, \cG_x \in \cA}{O, \eta(O)} \\ 
		&= \sum_{x \in \scU} \bE \cond{\bfP_{O} \sbra{z_i = x} \, \bP_O \cond{\cG_x \in \cA}{ \cG \smallsetminus \cG_x} }{O, \eta(O)}.
	\end{align*}
	Now the second factor in this sum can be bounded as 
	\begin{equation*}
		\bP_O \cond{\cG_x \in \cA}{ \cG \smallsetminus \cG_x} \leq \max_{u,v \in V} \, \bP \cond{\cG'_{O'} \in \cA}{\iota(O') = u, \iota(\eta(O')) = v}
	\end{equation*}
	with $(\cG', O')$ an independent copy of $(\cG,O)$. The remaining terms then sum up to $1$, yielding the result.	 
\end{proof}

The previous Lemma will be combined with the following.

\begin{lemma}\label{lem:no_backtracking}
	Given $\alpha, q_0 \in (0,1]$ and $l \geq 0$, let $\cA_l = \cA_l (q_0, \alpha)$ be the set of quasi-trees for which every path of length $l$ starting from $O$ contains a proportion at least $\alpha$ of vertices $x$ satisfying $\qesc(x) \geq q_0$. There exists a constant $C$ such that for all $l \geq 0$, on the event $\{ \cG \in \cA_l \}$, 
	\begin{gather*}
		\bfP_O \sbra{\exists t > s \geq 0: \cX_{s} \in \cG_x, d_{\LR}(O, \cX_s) \geq l, d_{\LR}(O, \cX_t) = 0} \leq e^{- C l} \\
		\bfP_O \sbra{\exists t \geq 0: d_{\SR}(O, \cX_t) \geq l} \leq e^{-C l}.
	\end{gather*}
	Similarly, for all $x \in \cV$, on the event $\{\cG_x \in \cA_{l} \}$,
	\begin{gather*}
		\bfP_x \sbra{\exists t > s \geq 0: \cX_{s} \in \cG_x, d_{\LR}(x, \cX_s) \geq l, d_{\LR}(x, \cX_t) = 0} \leq e^{- C l} \\
		\bfP_x \sbra{\exists t \geq 0: \cX_t \in \cG_x, d_{\SR}(\cX_{t}^{\circ}, \cX_t) \geq l} \leq e^{-C l}.
	\end{gather*}
\end{lemma}

\begin{proof}
	Given $q_0 > 0$, consider the set $\mathrm{Esc} = \mathrm{Esc}(q_0) := \{x \in \cV \ | \ \qesc(x) \geq q_0 \}$. For all $x \in \cG \cap \mathrm{Esc} \smallsetminus \{ O \}$, $\bfP_x \sbra{\tau_{O} < \infty} \leq 1 - \qesc(x) < 1 - q_0$.

	Let $l \geq 0$ and $\tau_l$ be the first time $t \in \bN$ such that $d(O, \cX_t) = l$ and $\tau_{O}$ the hitting time of $O$. By strong Markov's property,
	\begin{align*}
		\bfP_{O} \sbra{\exists t > s \geq 0: d_{\LR}(O, \cX_s) \geq l, d_{\LR}(O, \cX_t) = 0} &= \bfE_{O} \sbra{ \II_{\tau_l < \infty} \bfP_{\cX_{\tau_l}} \sbra{\tau_{O} < \infty} } \\
		&\leq \max_{x_1: d_{\LR}(O, x_1) = l} \bfP_{x_1} \sbra{\tau_{O} < \infty}.
	\end{align*}
	Let $\tau^{(k)}_{\mathrm{Esc}}$ be the $k$-th hitting time of $\mathrm{Esc} \smallsetminus \{ O \}$ and suppose $x_1 \in \cG$ is at long-range distance $l$ from $O$. Then $x_1$ must also be at $\cP$-distance at least $l$ from $O$, thus if $\cG \in \cA_l$ and $\cX_0 = x_1$ it is impossible that $\tau_{O} < \tau_{\mathrm{Esc}}^{(\lfloor \alpha l \rfloor )}$ as this would imply the existence of a path contradicting the definition of $\cA_l$. Thus when $\cG \in \cA_l$, applying Strong Markov's property at the successive stopping times $\tau_{\mathrm{Esc}}^{(k)}$ yields
	\begin{align*}
		\bfP_{x_1} \sbra{\tau_{O} < \infty} &= \bfP_{x_1} \sbra{\tau_{\mathrm{Esc}}^{(\lfloor \alpha l \rfloor )} \leq \tau_{O} < \infty} \\
		&\leq (1 - q_0)^{\lfloor \alpha l \rfloor}.
	\end{align*}

	The argument for the second bound is similar. Let $\tau^{(k)}_{\mathrm{Esc}}$ be now the $k$-th hitting time of $\mathrm{Esc}$, so we do not discard $O$ anymore, and let $\tau_{l}$ denote now the hitting time of $B_{\SR}(O, l)$. As before, observe it is impossible to have $\tau_{l} < \tau_{\mathrm{Esc}}^{(\lfloor \alpha l \rfloor )}$ if $\cG \in \cA_l$. However, to reach $B_{\SR}(O, l)$ there must be no escape to infinity before $\tau_l$. Hence on the event $\{ \cG \in \cA_l \}$,
	\begin{align*}
		\bfP_{O} \sbra{\exists t \geq 0: d_{\SR}(O, \cX_t) \geq l} &= \bfP_{O} \sbra{\tau_l < \infty} \\
		&= \bfP_{O} \sbra{\tau_{\mathrm{Esc}}^{(\lfloor \alpha l \rfloor )} \leq \tau_{l} < \infty} \\
		&\leq (1- q_0)^{\lfloor \alpha l \rfloor},
	\end{align*}
	the last line being obtained by applying the strong Markov property at times $\tau^{(k)}_{\mathrm{Esc}}$.

	Finally the two last statements for a starting state $x \in \cV$ are proved with the exact same reasoning.
\end{proof}

We now establish an analog of Lemma \ref{lem:typical_paths} in the quasi-tree setting.

\begin{lemma}\label{lem:typical_paths_QT}
	Let $\Gamma(R, L, M)$ denote the set of paths $\frp$ in $\cG$ such that $\frp$ does not deviate from a small-range distance more than $R$, backtrack over a long-range distance $L$ or contain a subpath of length $M$ without a regeneration edge. 
	There exists $C > 0$ and $\alpha \in (0,1]$ such that for all $R, L, M \geq 0$, for all $t \geq 0$, for any types of $O, \eta(O)$,
	\begin{equation*}
		\bP_O \cond{\cX_{s} \cdots \cX_{s+t} \notin \Gamma(R,L,M)}{O, \eta(O)} \leq (s+t)e^{-C (R \wedge L \wedge M^{\alpha})}.
	\end{equation*}
\end{lemma}

\begin{proof}
	Fix $R, L, M$, $s,t \geq 0$ and the types of $O, \eta(O)$. Let us start with the deviation property:
	\begin{multline}\label{eq:deviation_bound}
		\bP_O \cond{\exists t' \in [s,s+t]: d_{\SR}(\cX_{t'}^{\circ}, \cX_{t'}) \geq R}{O, \eta(O)} \\ \leq \bP_O \cond{\exists t' \leq s+t: d_{\SR}(O, \cX_{t'}) \geq R}{O, \eta(O)} \\ + \bP \cond{\exists t_1 \leq t_2 \leq s+t: \cX_{t_2} \in \cG_{\cX_{t_1}}, d_{\SR}(\cX_{t_2}^{\circ}, \cX_{t_2}) \geq R}{O, \eta(O)} 
	\end{multline}
	Given $q_0 > 0$, $\alpha \in (0,1)$, let $\cA_R = \cA_{R}(q_0, \alpha)$ be the set of quasi-trees, for which every path of length $R$ starting from $O$ contains a proportion at least $\alpha$ of vertices $x$ such that $\qesc(x) \geq q_0$. Then the first term can be bounded as 
	\begin{multline*}
		\bP_O \cond{\exists t' \leq s+t: d_{\SR}(O, \cX_{t'}) \geq R}{O, \eta(O)} \leq \bP \cond{\cG \notin \cA_R}{O, \eta(O)} \\ + \bE \cond{\II_{\cG \in \cA_R} \bfP_O \sbra{\exists t' \leq s+t: d_{\SR}(O, \cX_{t'}) \geq R}}{O, \eta(O)}.
	\end{multline*}
		By Proposition \ref{prop:escape_dense} there exist constants $q_0, \alpha > 0$ such that the first term can be made exponentially small in $R$, while Lemma \ref{lem:no_backtracking} shows the second term is also exponentially small. 
		The second term of \eqref{eq:deviation_bound} can be bounded similarly: 
		\begin{align*}
			&\bP \cond{\exists t_1 \leq t_2 \leq s+t: \cX_{t_2} \in \cG_{\cX_{t_1}}, d_{\SR}(\cX_{t_2}^{\circ}, \cX_{t_2}) \geq R}{O, \eta(O)} \\
			& \leq \bP \cond{\exists t_1 \leq s+t: \cG_{\cX_{t_1}} \notin \cA_R}{O, \eta(O)} \\ 
			& + \sum_{t_1 \leq s+t} \bE \cond{\II_{\cG_{\cX_{t_1}} \in \cA_{R}} \bfP_{\cX_{t_1}} \sbra{\exists t_2 \geq 0: \cX_{t_2} \in \cG_{\cX_{t_1}}, d_{\SR}(\cX_{t_2}^{\circ}, \cX_{t_2}) \geq R}}{O, \eta(O)} \\
			& \leq (s+t) e^{-C R}
		\end{align*}
		for some constant $C > 0$, using Lemma \ref{lem:escape_stationary} and Proposition \ref{prop:escape_dense} to bound the first term and Lemma \ref{lem:no_backtracking} for the second. 

		The probability that the trajectory backtrack over a long-range distance $L$ between $s$ and $s+t$ is established with the same arguments. It then remains to handle the regeneration requirement. By union bound for all $t \geq 0$ and $u \in V$
	\begin{align*}
		&\bP_O \cond{\exists s' \in [s,s+t]: [s', s' + M] \cap \{T_{k}, k \geq 1\} = \emptyset}{O, \eta(O)} \\
		&\qquad \leq \bP_O \cond{\exists k \leq s+t: T_{k+1} - T_{k} \geq M}{O, \eta(O)} \\ 
		&\qquad \leq( s+t) \max_{k \leq s+t} \bP_O \cond{T_{k+1} - T_{k} \geq M}{O, \eta(O)}.
	\end{align*} 
	The stretched exponential tails of regeneration times (Lemma \ref{lem:regeneration_tails}) conclude the proof. 
\end{proof}

\subsection{Concentration of the drift}

\begin{proposition}\label{prop:drift}
	Let $\mathscr{d} := \frac{\bE_{\bQ_{\mu}} \sbra{L_1}}{\bE_{\bQ_{\mu}} \sbra{T_1}}$. Then for all $s \geq 0$, a.s. 
	\begin{equation}\label{eq:cv_drift}
		\frac{d_{\LR}(\cX_s, \cX_{s+t})}{t} \xrightarrow[t \rightarrow \infty]{} \mathscr{d}.
	\end{equation}
	Furthermore, there exists $\alpha > 0$ for which the following holds. For all $\e > 0$ there exists $C=C(\e)$ such that for all $s,t \geq 0$, for all types of $O, \eta(O)$,
	\begin{equation}\label{eq:clt_drift}
		\bP_O \cond{ \abs{d_{\LR}(\cX_s, \cX_{s+t}) - \mathscr{d} t} > C \sqrt{t} }{O, \eta(O)} \leq \e + C \sqrt{s} e^{- t^{\alpha}}.
	\end{equation}
\end{proposition}

\begin{proof}
		For notational simplicity we omit writing the conditionning by $\iota(O)$ and $\iota(\eta(O))$. As can be checked this conditionning does not affect the proof as the technical results that will be used hold even conditional on the long-range edge at the root. For all $t \geq 0$, let 
		\[
			N_t := \max \{k \geq 0 \ | \ T_k \leq t \}.
		\]
		Then 
		\begin{equation*}
			L_{N_t} \leq d_{\LR}(O, \cX_t) \leq L_{N_{t}} + T_{N_t +1} - T_{N_t}.
		\end{equation*}
		It is easy to prove that $(T_{N_{t} + 1} - T_{N_t}) / t \rightarrow 0$, hence the law of large numbers \eqref{eq:cv_drift} follows from Lemma \ref{lem:markov_decomposition} and Proposition \ref{prop:lln_mrp} which prove $N_t / t \xrightarrow[t \rightarrow \infty]{} 1/ \bE_{\mu}\sbra{T_1}$ and $L_k / k \xrightarrow[k \rightarrow \infty]{} \bE_{\mu} \sbra{L_1}$ a.s..

		Then we establish \eqref{eq:clt_drift} in the case $s=0$. It suffices to prove that for all $\epsilon > 0$ there exists $C > 0$ such that for all $t,k \geq 0$
		\begin{align}
			&\bP_O \sbra{\abs{N_t - \frac{t}{\bE_{\bQ_{\mu}} \sbra{T_1}}} > C \sqrt{t}} \leq \e \label{eq:fluctations_Ns} \\
			&\bP_O \sbra{\abs{L_k -\bE_{\bQ_{\mu}} \sbra{L_1} k} > C \sqrt{k}} \leq \e. \nonumber
		\end{align}
		Indeed if this holds using that regeneration levels are non-decreasing we deduce that $\abs{L_{N_t} - \mathscr{d} t} > C \sqrt{t}$ with probability at most $\e$ for some $C = C(\e)$, whereas union bound and Markov's inequality show that for all $C' > 0$
		\begin{align*}
			&\bP_O \sbra{T_{N_t +1} - T_{N_t} > C' \sqrt{t}, \abs{N_t - \frac{t}{\bE_{\bQ_{\mu}} \sbra{T_1}}} \leq C \sqrt{t}} \\
			&\quad \leq \bP_O \sbra{\exists k: \abs{k - \frac{t}{\bE_{\bQ_{\mu}} \sbra{T_1}}} \leq C \sqrt{t}, T_{k+1} - T_{k} > C' \sqrt{t}} \\
			&\quad \leq 2 C \sqrt{t} \frac{\max_{u \in V} \bE_{\bQ_u} \sbra{T_{1}} \vee \bE_O \sbra{T_1}}{C' \sqrt{t}} \\
			&\quad \leq \e
		\end{align*}
		for large enough $C' = C'(\e)$, using also Lemma \ref{lem:regeneration_tails} to argue the expectations are $O(1)$.

		Let us now prove the claim. We only prove the upper tail of the first inequality as the other bounds are established similarly.
		Let 
		\begin{equation*}
			k = \left\lfloor \frac{t}{\bE_{\bQ_{\mu}} \sbra{T_1}} + C \sqrt{t} \right\rfloor.
		\end{equation*}
		The objective is to apply the Bienaymé Chebychev inequality with the variance bound for Markov chains \eqref{eq:variance_mrp}, which is valid only when started at equilibrium. Let $\e \in (0, 1/2)$ and $t_0 = \tmix^{(Y)}(\e)$ be the mixing time of $(Y_l)_{l \geq 0}$, then decompose $T_k = Z_1 + Z_2$ with $Z_1 = \sum_{i = 1}^{t_0} T_{i} - T_{i-1}$ and $Z_{2} = \sum_{i=t_0 +1}^{k} T_{i} - T_{i-1}$. 
		
		We write further $Z'_2 := Z_2 - \bE_{\bQ_{\mu}} \sbra{Z_2}$. Stationarity implies $\bE_{\bQ_{\mu}} \sbra{Z_2} = (k-t_0) \bE_{\bQ_{\mu}} \sbra{T_1}$ hence
		\begin{align*}
			\bP_O \sbra{N_t > k} &= \bP \sbra{ T_k < t } \\
			&= \bP_O \sbra{Z_1 + Z'_2 < t - (k - t_0) \bE_{\bQ_{\mu}}\sbra{T_1}} \\
			&\leq \bP_O \sbra{Z_1 + Z'_2 < z}
		\end{align*}
		where
		\begin{equation*}
			z := (t_0 + 1) \bE_{\bQ_{\mu}}\sbra{T_1} - C \bE_{\bQ_{\mu}}\sbra{T_1} \sqrt{t}.
		\end{equation*} 
		By Lemmas \ref{lem:regeneration_tails} and \ref{lem:mixing_regen}, $t_0$ and the expectation above can be bounded by constants independent of $n$. Thus for some constant $C' >0$,
		\begin{align*}
			\bP_O \sbra{Z_1 + Z'_2 < z} &\leq \bP_O \sbra{ \abs{Z_1 + Z'_2} >  2 C' \sqrt{t}} \\
			&\leq \bP_O \sbra{Z_1 > C' \sqrt{t}} + \bP_O \sbra{\abs{Z'_2} > C' \sqrt{t}}.
		\end{align*}
		From Lemma \ref{lem:mixing_regen} and the fact that $t_0 = O(1)$, we also deduce that $\bE_O \sbra{Z_1} = O(1)$ so Markov's inequality ensures the first term is smaller than $\e$ for all $t$ by taking the constant $C'$ large enough. For the second term, since $t_0 = \tmix^{(Y)}(\e)$
		\begin{align*}
			\bP_O \sbra{\abs{Z'_2} > C' \sqrt{t}} &= \sum_{y} \bP_O \sbra{Y_{t_0} = y} \bQ_{y} \sbra{\abs{T_{k-t_0} - \bE_{\bQ_{\mu}}\sbra{T_{k-t_0}}} > C' \sqrt{t}} \\
			&\leq \TV{Y_{t_0} - \mu} + \bQ_{\mu} \sbra{\abs{T_{k-t_0} - \bE_{\bQ_{\mu}}\sbra{T_{k-t_0}}} > C' \sqrt{t}} \\
			&\leq \e + \bQ_{\mu} \sbra{\abs{T_{k-t_0} - \bE_{\bQ_{\mu}}\sbra{T_{k-t_0}}} > C' \sqrt{t}}.
		\end{align*}
		Using the Bienaymé-Chebychev inequality with \eqref{eq:variance_mrp}, we deduce 
		\begin{equation*}
			\bQ_{\mu} \sbra{\abs{T_{k-t_0} - \bE_{\bQ_{\mu}}\sbra{T_{k-t_0}}} > C' \sqrt{t}} \leq \frac{\Var_{\bQ_{\mu}}(T_{k-t_0})}{(C')^2 t} \leq \frac{C'' (k-t_0) \Var_{\bQ_{\mu}} \sbra{T_1}}{(C')^{2} t},
		\end{equation*}
		for some constant $C'' > 0$, as the regeneration chain has constant mixing time by Lemma \ref{lem:mixing_regen}. By Lemma \ref{lem:regeneration_tails}, the variance of $T_1$ is of constant order while $k = O(t)$, hence the probability above can be made smaller than $\e$ for all $t$ by taking $C'$ large enough. 

		For the general case $s \geq 0$, we apply the previous arguments with a shifted version of the processes. For $s \geq 0$ fixed, consider
		\begin{equation*}
			T^{(s)}_{k} := T_{N_s + k} - s, \qquad L^{(s)}_{k} := L_{N_s + k} - d_{\LR}(O,\cX_s)
		\end{equation*}
		if $k \geq 1$ and $T^{(s)}_0 := 0, L^{(s)}_0 := 0$. Then let
		\begin{equation*}
			N^{(s)}_t := \max \{ k \geq 0 \ | \ T^{(s)}_k \leq t \}.
		\end{equation*}
		Notice that
		\begin{equation*}
			L^{(s)}_{N^{(s)}_t} \leq d_{\LR}(\cX_s, \cX_{s+t}) \leq L^{(s)}_{N^{(s)}_t} + T^{(s)}_{N^{(s)}_t+1} - T^{(s)}_{N^{(s)}_t}.
		\end{equation*}
		These processes still satisfy the conclusions of Lemma \ref{lem:markov_decomposition} and have the same increments as the usual regeneration times. Thus the only thing to be careful when applying the above arguments is the law of the first regeneration time that now depends on $s$. Using the concentration \eqref{eq:fluctations_Ns} for $N_s$, union bound and Lemma \ref{lem:regeneration_tails}, we can see that for all $m \geq 0$,
		\begin{align*}
			\bP \sbra{T^{(s)}_1 > m } &= \sum_{k \geq 0} \bP \sbra{N_s = k, T_{k+1} - s > m} \\
			&\leq \bP \sbra{\abs{N_s - s / \bE_{\bQ_{\mu}} \sbra{T_1}} > C \sqrt{s}} + \bP \sbra{\exists k: \abs{k - s / \bE_{\bQ_{\mu}} \sbra{T_1}} \leq C \sqrt{s}, T_{k+1} - T_k > m} \\
			&\leq \e + 2 C \sqrt{s} e^{-m^{\alpha}}.
		\end{align*}
		Taking $m = \sqrt{t}$ yields the stretched exponential term of \eqref{eq:clt_drift}. Then for higher regeneration times $(T^{(s)}_{k})_{k \geq 2}$, the above arguments apply. 
\end{proof}

\subsection{Concentration of the entropy}

The concentration of the entropy in Proposition \ref{prop:nice_approx} is based on the convergence of the entropy for the loop-erased chain in the quasi-tree, that is the convergence of $- \log \bfP \cond{\xi'_k = \xi_k}{\xi} / k$ towards a constant, the entropic rate of the chain. Such convergence is well-known in the context of groups or random walks on Galton-Watson trees, see \cite{kaimanovich1983random, lyons1995ergodic}. We will however not prove this result but establish concentration directly for a notion of weights similar to those of \eqref{eq:def_weights_G}. Of course, these are designed to mimick the law of the loop-erased chain, so the argument is similar. In fact the first step is to prove the convergence and concentration of the loop-erased chain when restricted to regeneration steps. 

\begin{lemma}\label{lem:entropy_regen}
	Let $\xi'$ be an independent copy of the loop-erased chain $\xi$. 
	There exists $h' = \Theta(1)$ such that 
	\begin{equation}\label{eq:lln_regen}
		\lim_{k \rightarrow \infty} \frac{- \log \bfP \cond{ \xi'_{L_k} = \xi_{L_k} }{\cX}}{k} = h'.
	\end{equation}
	Furthermore, for all $\e > 0$, there exists $C(\e) >0$ such that for all $k, l \geq 1$, for all types of $O, \eta(O)$, 
	\begin{equation}\label{eq:clt_regen}
		\begin{gathered}
			\bP \cond{\abs{- \log \bfP \cond{ \xi'_{L_k} = \xi_{L_k} }{\cX} - h' k} > C(\e) \sqrt{k}}{O, \eta(O)} \leq \e, \\
			\bP \cond{\abs{- \log \bfP \cond{ \xi'_{L_{k+l}} = \xi_{L_{k+l}} }{\cX, \xi'_{L_k} = \xi_{L_k}} - h' l} > C \sqrt{l}}{O, \eta} \leq \e.
		\end{gathered}
	\end{equation}
\end{lemma}

\begin{proof}
	Again let us omit the conditionning by $O, \eta(O)$ which does not affect the argument. Given $k \geq 1$, we write $\cG(k) := \cG_{\cX_{T_{k}}}$, to simplify notations. Given $u \in V$ and $g$ a possible realization of the subquasi-tree $\cG_O$, let
	\begin{equation*}
		\bfQ_{u,g} := \bP \cond{\cdot}{\cX_{1/2} = \eta(O), \tau_{O} = \infty, \iota(O) = u, \cG_O = g}.
	\end{equation*}
	Finally let $\xi'(k)$ denote a loop-erased chain on $\cG(k-1)$, started at $\cX_{T_{k-1}}$, independent of $\cX$.
	Having $\xi'_{L_k} = \xi_{L_k}$  implies $\xi_{L_{m}} \in \xi'$ for all $m \leq k$. Thus for $k \geq 2$
	\begin{align*}
		\bfP \cond{\xi_{L_k} \in \xi'}{\cX} &= \bfE \cond{ \II (\xi'_{L_{k-1}} = \xi_{L_{k-1}}) \, \bfP \cond{\xi'_{L_k} = \xi_{L_k}}{\cX, (\xi'_i)_{i \leq L_{k-1}}} }{\cX} \\
		&= \bfE \cond{\II (\xi'_{L_{k-1}} = \xi_{L_{k-1}}) \,  \bfQ_{Y_{k-1}, \cG(k-1)} \cond{\xi_{L_k} \in \xi'(k)}{(\cX_t)_{t \geq T_{k-1}}}}{\cX} \\
		&= \bfP \cond{\xi'_{L_{k-1}} = \xi_{L_{k-1}}}{\cX} \bfQ_{Y_{k-1}, \cG(k-1)} \cond{\xi_{L_k} \in \xi'(k)}{(\cX_t)_{t \geq T_{k-1}}}.
	\end{align*}
	Letting $Z_1 := - \log \bfP \cond{\xi'_{L_1} = \xi_{L_1}}{\cX}$ and
	\begin{equation*}
		Z_i := - \log \bfQ_{Y_{i-1}, \cG(i-1)} \cond{\xi_{L_i} \in \xi'(i)}{(\cX_t)_{t \geq T_{i-1}}}
	\end{equation*}
	for $i \geq 2$, we just proved that 
	\begin{equation*}
		- \log \bfP \cond{ \xi'_{L_k} = \xi_{L_k} }{\cX} = \sum_{i=1}^{k} Z_i
	\end{equation*}
	and
	\begin{equation*}
		- \log \bfP \cond{ \xi'_{L_{k+l}} = \xi_{L_{k+l}} }{\cX, \xi'_{L_k}} = \sum_{i=k+1}^{k+l} Z_i.
	\end{equation*}
	Under $\bQ_{\mu}$, the sequence $(Y_k)_{k \geq 1}$ is stationary. Lemma \ref{lem:markov_decomposition} then shows that $(Z_i)_{i \geq 2}$ is stationary as well.
	Provided $Z_1 < \infty$ a.s., $\bE_O \sbra{Z_1} < \infty$ and $\bE_{\bQ_{\mu}} \sbra{\abs{Z_1}} < \infty$, which is proved afterwards, Birkhoff's ergodic theorem implies the law of large numbers \eqref{eq:lln_regen}. 

	Let us now show that for all $l \geq 1$, there exists a constant $C_l \geq 0$ such that
	\begin{equation*}
		\bE_O \sbra{\abs{Z_1}^{l}} < C_l
	\end{equation*}
	and for all $i \geq 2$ and $u \in V$
	\begin{equation}\label{eq:moment_Zi}
		\bE_{\bQ_{u}} \sbra{\abs{Z_i}^l} \leq C_l.
	\end{equation}
	This is sufficient to justify that $Z_1 < \infty$ a.s., $\bE_{\bQ_{\mu}} \abs{Z_i} < \infty$ and will be used to prove \eqref{eq:clt_regen}. We only treat the case $ i \geq 2$ as the case $i = 1$ is similar. Let $l \geq 1$ and $u \in V$. Let us identify here a long-range edge $(\eta(x), x)$ with its center vertex $x$. By Lemma \ref{lem:markov_decomposition}
	\begin{align*}
		\bE_{\bQ_{u}} \sbra{\abs{Z_i}^l} &= \bE_{\bQ_{u}} \sbra{ \sum_{x \in \cV} \bfQ_{u, \cG} \sbra{\cX_{T_i} = x} \left( - \log \bfQ_{Y_{i-1}, \cG(i-1)} \sbra{x \in \xi'(i)} \right)^l }\\
		&= \sum_{v \in V} \bQ_{u} \sbra{Y_{i-1} = v} \bE_{\bQ_{v}} \sbra{ \sum_{x \in \cV} \bfQ_{v, \cG} \sbra{\cX'_{T'_1} = x} \left( - \log \bfQ_{v, \cG} \sbra{x \in \xi'} \right)^l }
	\end{align*}
	where $\cX'$ is an independent copy of $\cX$ under the law $\bfQ_{v, \cG}$, $\xi'$ is its loop-erased trace and $T'_1$ its first regeneration time. Thus it suffices to prove the upper bound
	\begin{equation*}
		\max_{u \in V} \bE_{\bQ_{u}} \sbra{\sum_{x \in \cV} \bfQ_{u, \cG} \sbra{\cX_{T_1} = x} \left( - \log \bfQ_{u, \cG} \sbra{x \in \xi} \right)^{l}} \leq C_l.
	\end{equation*}
	Fix $u \in V$. For all $k \geq 1$, if $d_{\cP}(\eta(O), x) = k$, the chain goes from $\eta(O)$ to $x$ with probability at least $\delta^{k}$, after which it escapes to infinity in $\cG_x$ with probability $\qesc(x)$, hence 
	\begin{align*}
		\sum_{x \in \cV} \bfQ_{u, \cG} \sbra{\cX_{T_1} = x} \left( - \log \bfQ_{u, \cG} \sbra{x \in \xi} \right)^l &\leq \sum_{k \geq 1} \bfQ_{u, \cG} \sbra{d_{\cP}(\eta(O),\cX_{T_1}) = k} k^{l}  (- \log  \delta)^{l} \\
		&+ \sum_{x \in \cV} \bfQ_{u, \cG} \sbra{\cX_{T_1} = x} \left( - \log \qesc(x) \right)^l 
	\end{align*}
	Averaging on the environment yields
	\begin{align*}
		\bE_{\bQ_{u}} \sbra{\sum_{x \in \cV} \bfQ_{u, \cG} \sbra{\cX_{T_1} = x} \left( - \log \bfQ_{u, \cG} \sbra{x \in \xi} \right)^{l}} &\leq (- \log \delta)^{l} \, \bE_{\bQ_{u}} \sbra{ d_{\cP}(\eta(O), \cX_{T_1})^{l}} \\
		&+ \bE_{\bQ_{u}} \sbra{ (- \log \qesc(\cX_{T_1}))^{l}}
	\end{align*}
	Now for any $k \geq 1$, 
	\begin{equation*}
		\bQ_{u} \sbra{d_{\cP}(\eta(O), \cX_{T_1}) \geq k} \leq \bQ_{u} \sbra{T_1 \geq k}
	\end{equation*}
	which is stretched exponential by Lemma \ref{lem:regeneration_tails}, hence $\bE_{\bQ_{u}} \sbra{ d_{\cP}(\eta(O), \cX_{T_1})^{l}} = O(1)$. 
	On the other hand
	\begin{align*}
		\bE_{\bQ_u} \sbra{ (- \log \qesc(\cX_{T_1}))^{l}} = \sum_{v \in V} \bQ_{u}\sbra{Y_1 = v} \bE_{\bQ_{v}} \sbra{(- \log \qesc (\eta(O)))^l}.
	\end{align*}
	By Proposition \ref{prop:escape_gumbel} and Remark \ref{rk:comparison_Qu}, $\bE_{\bQ_{v}} \sbra{(- \log \qesc(\eta(O)))^{l}} = O(1)$ for any $v \in V$, hence $\bE_{u} \sbra{ (- \log \qesc(\cX_{T_1}))^{l}} = O(1)$ as well and we deduce \eqref{eq:moment_Zi} for some constant $C_l$. 
	
	To establish \eqref{eq:clt_regen}, we argue as in the end of the proof of Proposition \ref{prop:drift} for the drift. The bound is eventually proved using the Bienaymé-Chebychev inequality. Since $Y$ mixes in constant time $t_0$ by Lemma \ref{lem:mixing_regen}, separating the sum $\sum_{i=1}^{k} Z_i$ into two sums corresponding respectively to regeneration times before or after $t_0$ shows that it suffices to establish the variance bound $\Var_{\bQ_{\mu}} \sbra{\sum_{i=2}^{k} Z_i} = O(k)$ under the stationary measure. Let us modifiy the definition of $Z_1$, setting $Z_1 := - \log \bfQ_{\iota(O), \cG_O} \cond{\xi_{L_1} \in \xi'}{\cX}$ so that the whole sequence $(Z_i)_{i \geq 1}$ is stationary. It then suffices to prove that for all $k \geq 1$, $\Cov_{\bQ_{\mu}}(Z_1,Z_k)$ is stretched exponential in $k$. 

	Let $k > 1$. Let $\tau'_{k-1}$ be the first time the chain $\cX'$ reaches the $L_{k-1}$-th level, and consider the loop-erased trace $\xi'(1,k)$ obtained from observing $\cX'$ up to time $\tau'_{k-1}$.  Define $W_1 := \exp(-Z_1), W_{1k} := \bfQ_{\iota(O), \cG_O} \cond{\xi_{L_1} \in \xi'(1,k)}{\cX}$, $Z_{1,k} := - \log W_{1k}$.
	By the inequality $\abs{\log x - \log y} \leq \abs{x-y} / (x \wedge y)$,
	\begin{equation*}
		\abs{Z_{1} - Z_{1,k}} \leq \frac{\abs{W_{1} - W_{1k}}}{W_{1} \wedge W_{1k}}
	\end{equation*}
	Observe that if $\xi_{L_1}$ is exclusively in one of the two paths $\xi'(1,k)$, $\xi'$ then the chain $\cX'$ backtracks from the $L_{k-1}$-th level to the $(L_1 - 1)$-th. Letting $A'$ denote this event, we thus have
	\begin{equation*}
		\abs{W_{1} - W_{1k}} \leq \bfQ_{\iota(O), \cG_O} \cond{A'}{\cX}.
	\end{equation*}
	Let $\cA_{k}(q_0,\alpha)$ be the set of quasi-trees for which every length $k$ path from the root contains a proportion at least $\alpha > 0$ of vertices $x$ such that $\qesc(x) \geq q_0 > 0$. For some constants $q_0, \alpha, \e \in (0,1)$ and $C_1 > 0$ to determine, consider the events
	\begin{align*}
		&E_0 := \{ \cG_O \in \cA_k(q_0,\alpha) \} \cap \{\qesc(\xi_{L_1}) \geq \e^{k} \} \\
		&E_1 = \{ d_{\cP}(O,\cX_{T_1}) \leq  k / C_1  \}
	\end{align*}
	If $\cG_O \in \cA_k(q_0, \alpha)$, Lemma \ref{lem:no_backtracking} implies $\bfQ_{\iota(O), \cG_O} \cond{A'}{\cX}$ is exponentially small in $k$. On $E_1$, the chain $\cX'$ goes from $O$ to $\xi_{L_1}$ with probability at least $\delta^{k/C_1}$, after which $\xi_{L_1} \in \xi' = \xi'(1,k)$ if the chain $\cX'$ escapes to infinity. Thus on the event $E_0 \cap E_1$, 
	\begin{equation*}
		W_1 \wedge W_{1k} \geq \delta^{k/C_1} \e^{k},
	\end{equation*}
	hence choosing the constant $C_1$ large enough and $\e$ close enough to $1$ yields eventually
	\begin{equation*}
		\abs{Z_{1} - Z_{1,k}} \II_{E_0 \cap E_1} \leq e^{-c_1 k}
	\end{equation*}
	for some constant $c_1 > 0$. Next we claim that given this choice of $C_1$ and $\e$, there is an appropriate choice of the remaining parameters that ensures
	\begin{equation*}
		\bQ_{\mu} \sbra{E_{0}^{c}} \leq e^{- c_2 k}, \quad \bQ_{\mu} \sbra{E_1^{c}} \leq e^{-c_2 k^{\beta}}
	\end{equation*}
	for some $c_2 > 0$ and $\beta \in (0,1]$. By Remark \ref{rk:comparison_Qu} and Proposition \ref{prop:escape_dense} there exist $q_0, \alpha > 0$ such that $\cG_O \notin \cA_k(q_0,\alpha)$ has exponentially small probability, whereas Proposition \ref{prop:escape_gumbel} implies $\qesc(\xi_{L_1}) < \e^{k}$ with doubly exponentially small probability. For $E_1$ observe $d_{\cP}(O,\cX_{T_1}) > k/C_1$ implies $T_1 > k /C_1$ which has stretched exponentially small probability by Lemma \ref{lem:regeneration_tails}. Then by the moment bounds \eqref{eq:moment_Zi} for some $c_3 > 0$,
	\begin{align*}
		\Cov_{\bQ_{\mu}}(Z_1,Z_{2k}) &= \bE_{\bQ_{\mu}} \sbra{(Z_1 - \bE_{\bQ_{\mu}} \sbra{Z_1})(Z_{2k} - \bE_{\bQ_{\mu}} \sbra{Z_{2k}}) \II_{E_1}} + O \left( e^{- c_3 k^{\beta}} \right) \\
		&= \bE_{\bQ_{\mu}} \sbra{(Z_1 - Z_{1,k}) Z_{2k} \II_{E_1}} - \bE_{\bQ_{\mu}} \sbra{(Z_1 - Z_{1,k}) \II_{E_1}} \bE_{\bQ_{\mu}} \sbra{Z_{2k}} \\
		&\quad + \bE_{\bQ_{\mu}} \sbra{(Z_{1,k} - \bE_{\bQ_{\mu}} \sbra{Z_1})(Z_{2k} - \bE_{\bQ_{\mu}} \sbra{Z_{2k}})\II_{E_1}} + O \left( e^{- c_3 k^{\beta}} \right)  \\
		&= \bE_{\bQ_{\mu}} \sbra{(Z_1 - Z_{1,k}) Z_{2k} \II_{E_0 \cap E_1}} - \bE_{\bQ_{\mu}} \sbra{(Z_1 - Z_{1,k}) \II_{E_0 \cap E_1}} \bE_{\bQ_{\mu}} \sbra{Z_{2k}} \\
		&\quad + \bE_{\bQ_{\mu}} \sbra{(Z_{1,k} - \bE_{\bQ_{\mu}} \sbra{Z_1})(Z_{2k} - \bE_{\bQ_{\mu}} \sbra{Z_{2k}})\II_{E_1}}  + O \left( e^{- c_3 k^{\beta}} \right) \\
		&= \bE_{\bQ_{\mu}} \sbra{(Z_{1,k} - \bE_{\bQ_{\mu}} \sbra{Z_1})(Z_{2k} - \bE_{\bQ_{\mu}} \sbra{Z_{2k}})\II_{E_1}} + O \left( e^{- c_3 k^{\beta}} \right).
	\end{align*}
	Note that by construction $Z_{1,k} \II_{E_1}$ is measurable with respect to $\cG \smallsetminus \cG(k)$ and $(\cX_{t})_{t \leq T_{k-1}}$, whereas $Z_{2k}$ is measurable with respect to $\cG(2k-1)$ and $(\cX_{t})_{T_{2k-1} \leq t}$ only. Thus by conditionning on $\{Y_{k-1},Y_{2k-1}\}$ Markov's property implies
	\begin{align*}
		&\bE_{\bQ_{\mu}} \sbra{(Z_{1,k} - \bE_{\bQ_{\mu}} \sbra{Z_1})(Z_{2k} - \bE_{\bQ_{\mu}} \sbra{Z_{2k}})\II_{E_1}} \\
		&\qquad = \bE_{\bQ_{\mu}} \sbra{\bE_{\bQ_{\mu}} \cond{(Z_{1,k} - \bE_{\bQ_{\mu}} \sbra{Z_1}) \II_{E_1}}{Y_{k-1}}  \bE_{\bQ_{\mu}} \cond{Z_{2k} - \bE_{\bQ_{\mu}} \sbra{Z_{2k}}}{Y_{2k-1}} }
	\end{align*}
	The right-hand side can be written as $\bE_{\bQ_{\mu}} \sbra{f(Y_{k-1}) g(Y_{2k-1})}$ for some functions $f, g$. From the proof of Proposition \ref{prop:variance_markov}, this expectation is exponentially small in $k$, provided $f$ and $g$ have finite moments. This is the direct consequence of \eqref{eq:moment_Zi} for $g$, while for $f$ we need moment bounds on $Z_{1k}$. These can be established with the same arguments as above.
\end{proof}

To relate the previous concentration with the weights \eqref{eq:def_weights_G}, we define similar weights in the quasi-tree. Let $\tau_l$ denote here the first time $t$ such that $d_{\LR}(\cX_0, \cX_t) = l$. For all long-range edge $e \in \cG$, write $\cG_e$ for the subquasi-tree at any endpoint of $e$ (they give the same quasi-tree). 
Given $R, L \geq 0$, $x \in \cV$ and a long-range edge $e$ at long-range distance $0$ from $x$ 
\begin{equation}\label{eq:def_weights_quasiT}
	\begin{split}
	w_{x,R,L}(e) &:= \bfP_{x} \sbra{\cX_{\tau_L} \in \cG_{e}, \tau_{L} < \TSR} \\
	w_{R,L}(e \ | \ x) &:= \bfP \cond{\cX_{\tau_L} \in \cG_{e}, \tau_{L} < \TSR}{\cX_{1/2} = x, \tau_{L} < \tau_{\eta(x)}}
	\end{split}
\end{equation}
Then if $e=(e_i)_{i=1}^{k}$ is a long-range non-backtracking path starting from $\BLR(x,0)$, set
\begin{equation*}
	w_{x,R,L}(e) := w_{x,R,L}(\xi_1) \prod_{i=2}^{k} w_{R,L}(\xi_i \ | \xi_{i-1}^{+})
\end{equation*}
where the product if taken equal to $1$ if empty. 

In the proof of Lemma \ref{lem:entropy_regen}, we introduced the measure 
	\begin{equation*}
		\bfQ_{u,g} := \bP \cond{\cdot}{\cX_{1/2} = \eta(O), \iota(O) = u, \tau_{O} = \infty, \cG_O = g}.
	\end{equation*}
where $u \in V$ and $g$ is a possible realization of the subquasi-tree $\cG_O$. We use a notation which may be reminiscent of \eqref{eq:bfQu_G} as these two measures are very similar, although note that here $u$ is not the type of the starting state of the chain but its long-range neighbour. There should be no risk of confusion as the measure of \eqref{eq:bfQu_G} will not be used until Section \ref{section:nice}. Using the same arguments as for the lemma, we can see that
\begin{equation}\label{eq:harmonic_measure_product}
	\bfP_{x} \sbra{\xi_1 = e_1, \ldots, \xi_k = e_k} = \bfP_{x} \sbra{\xi_1 = e_1} \prod_{i=2}^{k} \bfQ_{\iota(e_{i-1}^{+}), \cG_{e_{i-1}}} \sbra{\xi_{1} = e_i}.
\end{equation}
The following Lemma establishes thus a bound on individual weights. 

\begin{lemma}\label{lem:weight_approx}
	There exist constants $C, C_0,C_1,C_2 > 0$ such that for all $R, L > 0$ the following holds, conditional on the types of $O, \eta(O)$: 
	\begin{enumerate}[label=(\roman*)]
		\item with probability at least $1 - e^{-C (R \wedge L)}$, for all long-range edge $e$ such that $d_{\SR}(O,e^{-}) < R$,
			\begin{equation*}
			\abs{\log w_{O,R,L}(e) - \log \bfP_O \sbra{\xi_1 = e}} \leq C_0 e^{- C_1 L + C_2 R}, \label{eq:weight_approx_O} 
			\end{equation*}
		\item for all $x \in \cV$, with probability at least $ 1 - e^{-C (R \wedge L)}$ conditional on $\cG \smallsetminus \cG_{x}$, for all long-range edge $e$ of $\cG_x$ such that $d_{\SR}(x, e^{-}) < R$
			\begin{equation*}
				\abs{\log w_{R,L}(e \ | \ x) - \log \bfQ_{\iota(e^+), \cG_e} \sbra{\xi_1 = e}} \leq C_0 e^{- C_1 L + C_2 R}. \label{eq:weight_approx_conditional}
			\end{equation*}
	\end{enumerate}
\end{lemma}

\begin{proof} 
	We prove the first bound in detail. As usual we forget about the conditionning by the types of $O, \eta(O)$ which changes little in practice. Given $k \geq 1$, let $\cA_{k} = \cA_{k}(q_0, \alpha)$ be the set of quasi-trees for which every path of length $k$ starting from the root contains a proportion at least $\alpha$ of vertices $x$ such that $\qesc(x) \geq q_0$. Let
	\begin{equation*}
		E_0 := \{ \cG_{O} \in \cA_{R} \cap \cA_{L} \} \cap \{\forall x \in \BSR(O, R): \qesc(x) \geq c_0 e^{- c_1 R} \}.
	\end{equation*}
	We claim the event $E_{0}^{c}$ occurs with exponentially small probability in $R \wedge L$ for an appropriate choice of parameters. First there exist $q_0, \alpha > 0$ such that $\cG \notin \cA_{R} \cap \cA_{L}$ with exponentially small probability by Proposition \ref{prop:escape_dense}. Next, Proposition \ref{prop:escape_gumbel} implies there exists $c \in (0,1)$ such that for a fixed $x \in \BSR(O, R)$, $\qesc(x) < c \delta^{4 R}$ with probability at most $c^{2^{R}}$. However there are $O(\Delta^{R})$ vertices in this ball, hence the claim by union bound. 
	
	Notice then that to have one of the two events $e \in \xi$ or $\cX_{\tau_L} \in \cG_e$ realized exclusively, the chain needs to backtrack from level $L$ to level $0$. On $E_0$, this occurs with probability exponentially small in $L$ by Lemma \ref{lem:no_backtracking}. Thus on $E_0$
	\begin{equation*}
		\abs{w_{O,R,L}(e) - \bfP_O \sbra{\xi_1 = e}} \leq \bfP_O \sbra{\TSR \wedge \TNB < \infty} \leq e^{- C' (R \wedge L)}
	\end{equation*}
	for some constant $C' > 0$. On the other hand, on $E_0$
	\begin{equation*}
		w_{O,R,L}(e) \wedge \bfP_O \sbra{\xi_1 = e} \geq \delta^{R} q_0
	\end{equation*}
	as the right-hand side is a lower bound on the probability to go from $O$ to $e$ and then escape to infinity, which forces both $e \in \xi$ and $\cX_{\tau_{L}} \in \cG_e$. Using $\abs{\log x - \log y} \leq \abs{x-y} / (x \wedge y)$, we deduce that on $E_0$, 
	\begin{equation*}
		\abs{\log w_{O,R,L}(e) - \log \bfP_O \sbra{\xi_1 = e}} \leq C_0 e^{-C_1 L + C_2 R}
	\end{equation*}
	for some constants $C_0, C_1, C_2 > 0$. 

	The second bound is proved similarly conditional on $\cG \smallsetminus \cG_x$. The main difference is the additional conditionning by either $\tau_L < \tau_{\eta(x)}$ or $\tau_{\eta(x)} = \infty$. Since the differences implies backtracking, these can be controlled using the same arguments as above and Proposition \ref{prop:escape_gumbel}. 
\end{proof}

\begin{lemma}\label{lem:concentration_weights_QT}
	There exist constants $C, C_1, \ldots, C_3 > 0$ and $\beta \geq 1$ such that the following holds. For all $\e > 0$ there exist $C_0 = C_0(\e), C_h=C_h(\e)$ such that for all $s,t \geq 0$, for all $R,L > 0$, with probability at least $ 1 - \e - 2 (s+t) e^{-C (R \wedge L)}$ conditional on the types of $O$ and $\eta(O)$, 
	\begin{equation}\label{eq:clt_weights}
		\abs{- \log w_{\cX_s, R, L}(\xi(\cX_s \cdots \cX_{s+t})) - h t} \leq C_{h} \sqrt{t} + C_0 \, t \, e^{-C_1 L + C_2 R} + C_3 R L^{\beta + 1}. 
	\end{equation}
\end{lemma}

\begin{proof}
	Fix $R,L$ for the rest of the proof. For notational simplicity, we drop subscripts $R,L$ from the weights and omit writing the conditionning by $\iota(O), \iota(\eta(O))$ but the probabilistic statements below should be interpreted conditional on these. Recall $\cG^{(R)}$ is the quasi-tree truncated at the $R$-boundary of small-range components. Note that that weights are positive only for edges in $\cG^{(R)}$.
	
	Given $t \geq 0$, let 
	\begin{equation*}
		N_t := \max \{k \geq 0 \ | \ T_k \leq t \}.
	\end{equation*}
	We first argue there exists a constant $C_3 > 0$ such that for all $s,t \geq 0$ if $\cX_{0} \cdots \cX_{s+t}$ is included in $\cG^{(R)}$ then
	\begin{multline}\label{eq:weight_truncation}
		\abs{\log w_{\cX_s}(\xi(\cX_s \cdots \cX_{s+t})) - \log w (\xi_{L_{N_{s}+2}} \cdots \xi_{L_{N_{s+t}}} \ | \ \xi_{L_{N_s}+1})} \leq C_0 R L \left( T_{N_{s} +2} - T_{N_s} \right. \\ \left. + T_{N_{s+t} +1} - T_{N_{s+t}}\right).
	\end{multline}
	Note first that $\xi_{L_{N_{s}+1}} \cdots \xi_{L_{N_{s+t}}}$ is necessarily part of the path $\xi(\cX_s \cdots \cX_{s+t})$. Since weights are below one, we can easily lower bound  
	\begin{equation*}
		- \log w_{\cX_s}(\xi(\cX_s \cdots \cX_{s+t})) \geq  - \log w (\xi_{L_{N_{s}+2}} \cdots \xi_{L_{N_{s+t}}} \ | \ \xi_{L_{N_{s}+1}}).
	\end{equation*}
	On the other hand the path $\xi(\cX_s \cdots \cX_{s+t})$ contains at most $T_{N_{s} + 2} - s \leq T_{N_{s} +2} - T_{N_s}$ edges until it reaches $\xi_{L_{N_{s}+2}}$ and similarly it contains at most $T_{N_{s+t} +1} - T_{N_{s+t}}$ after it leaves $\xi_{L_{N_{s+t}}}$. Then we bound the weights of these edges. If $x \in \cV$ and $e$ is a long-range edge such that $d_{\SR}(x, e^{-}) < R$, $\delta^{R}$ lower bounds the probability the chain goes from $x$ to $e$ and leaves by this edge. As small-range components have size at least $2$ \ref{hyp:cc3}, $\delta^{RL}$ lower bounds the probability that the chain goes from $x$ to $e^{+}$ and reaches long-range distance $L$ before coming back to $e^{-}$. Consequently, for some constant $C_3 > 0$, we can always bound $- \log w_{x}(e) \leq C_3 R L$	for all $x,e$ in $\cG^{(R)}$, for any realization of the quasi-tree. We deduce that
	\begin{multline*}
		- \log w_{\cX_s}(\xi(\cX_s \cdots \cX_{s+t})) + \log w_{\cX_s} (\xi_{L_{N_{s+2}}} \cdots \xi_{L_{N_{s+t}}}) \leq C_3 R L \left( T_{N_{s} +2} - T_{N_s} \right. \\ \left. + T_{N_{s+t} +1} - T_{N_{s+t}}\right).
	\end{multline*}
	
	Next fix $\e \in (0,1)$ and let $s,t \geq 0$. Let $\cA$ be the set of quasi-trees for which every long-range edge $e$ such that $d_{\SR}(O, e^{-}) < R$ satisfies the bounds of Lemma \ref{lem:weight_approx} for some $C_0,C_1,C_2 > 0$. Consider the following events: 
	\begin{align*}
		&E_0 := \{ \cG \in \cA \} \cap \bigcap_{s' \leq s + t} \{ \cG_{\cX_{s'}} \in \cA \}, \\
		&E_1 := \{ \TSR > s + t \} \\
		&E_2 := \{ \abs{L_{N_{s+t}} - L_{N_s} - \mathscr{d} t } \leq C_{\LR} \sqrt{t} \}
	\end{align*}
	with $C_{\LR} = C_{\LR}(\e) > 0$ to determine, and write $E := E_0 \cap E_1 \cap E_2$. By Lemmas \ref{lem:weight_approx} and \ref{lem:escape_stationary}, $C_0,C_1,C_2$ can be chosen so that $\bP_O \sbra{E_0^c} \leq (s+t) e^{- C_4 (R \wedge L)}$ for some $C_4 > 0$. By Lemma \ref{lem:typical_paths_QT} there exists $C_5 > 0$ such that $\bP_O \sbra{E_1^{c}} \leq (s+t) e^{- C_5 R}$. Finally from the proof of Proposition \ref{prop:drift}, $C_{\LR}(\e)$ can be taken so that $\bP_O \sbra{E_2^c} \leq \e$. Hence $\bP_O \sbra{E^c} \leq \e + (s+t) e^{-C (R \wedge L)}$, for some constant $C > 0$.

	Then let $l_{+} := \lceil \mathscr{d} t + C_{\LR} \sqrt{t} \rceil$. On the event $E$, using \eqref{eq:harmonic_measure_product}
	\begin{multline*}
		\abs{\log w (\xi_{L_{N_{s}+2}} \cdots \xi_{L_{N_{s+t}}} \ | \ \xi_{L_{N_{s} +1}}) - \log \bfP \cond{\xi_{L_{N_{s}+2}} \cdots \xi_{L_{N_{s+t}}} \in \xi'}{\cX, \xi'_{L_{N_{s}+1}} = \xi_{L_{N_{s}+1}}}} \\ \leq l_{+} e^{-C_1 L + C_2 R}.
	\end{multline*}	
	Letting $h'$ be the constant of Lemma \ref{lem:entropy_regen}, we claim then that for $h = h' / \bE_{\bQ_{\mu} \sbra{T_1}}$ and some constant $C_h = C_h(a,\e)$
	\begin{equation*}
		\abs{- \log \bfP \cond{\xi_{L_{N_{s}+2}} \cdots \xi_{L_{N_{s+t}}} \in \xi'}{\cX, \xi'_{L_{N_{s}+1}} = \xi_{L_{N_{s}+1}}} - h t} \leq C_{h} \sqrt{t}
	\end{equation*}
	with probability at least $\e$. This can be done using the same arguments as in the proof of Proposition \ref{prop:drift}, combining the fluctations for $N_{s+t}$ \eqref{eq:fluctations_Ns} established in the said proof with Lemma \ref{lem:entropy_regen}.

	Finally, we are left with bounding the right hand side of \eqref{eq:weight_truncation}. Using again the fluctations for $N_{s+t}$ \eqref{eq:fluctations_Ns} with the stretched exponential tail of regeneration times, there exists $\beta \geq 1$ and $C_6 > 0$ such that
	\begin{equation*}
		\bP \cond{ T_{N_{s} +1} - T_{N_s} + T_{N_{s+t} +1} - T_{N_{s+t}} > L^{\beta}}{O, \eta(O)} \leq \e + \sqrt{s+t} e^{- C_6 L}.
	\end{equation*}
	All in all, up to changing the value of $C$ we have thus proved that with probability at least $1 - 2 \e - (s+t) e^{-C (R \wedge L)}$,
	\begin{equation*}
		\abs{\log w_{\cX_s}(\xi(\cX_s \cdots \cX_{s+t})) - ht } \leq C_h \sqrt{t} + C_0 \, l_{+} e^{-C_1 L + C_2 R} + C_3 R L^{\beta + 1}.
	\end{equation*}
	Since $l_+ = O_{\e}(t)$ this proves the result.
\end{proof}
\section{First steps towards nice paths}\label{section:nice_first}

We now come back to the finite setting. The goal of this section is to prove Lemmas \ref{lem:QTlike_uniform}, \ref{lem:typical_paths} and Proposition \ref{prop:concentration_G}, which will be used in the next section to justify the chain is likely to follow a nice trajectory. 

\subsection{The case of typical starting states}

We start with the case of typical starting states, for which there is little left to prove. 

\begin{proof}[Proof of (i) in Lemma \ref{lem:typical_paths}]
    This is essentially a matter of checking definitions. For fixed $x \in V$, the choice of $s,t = O(\log n)^{a}, a \geq 1$, Lemma \ref{lem:coupling_quasiT} ensures the finite chain $X$ and the chain $\cX$ on the quasi-tree can be coupled until time $s+t$ with annealed probability $1 - o(n^{-\e})$ for some $\e > 0$. Provided the coupling succeeds, the property of backtracking or hitting the boundary of a small-range ball coincides for both chains, while the regeneration times of $\cX$ obviously give regeneration edges for the finite chain. Thus from Lemma \ref{lem:coupling_quasiT} and Lemma \ref{lem:typical_paths_QT} we deduce
    \begin{equation}
        \bP_x \sbra{(X_s \cdots X_{s+t} \cdots \in \Gamma(R,L,M))} \leq (s+t) e^{-C(R \wedge L \wedge M^{\alpha})} + o(n^{- \e})
    \end{equation}
    Since $s,t = (\log n)^{a}$, for any $b > 0$ the choice of $R,L = O(\log \log n)$ with large enough implicit constants depending on $a,b$ and $M = (\log \log n)^{2/\alpha}$ makes the right hand side $o((\log n)^{-b})$. The first moment argument \eqref{eq:markov_quenched} implies (i) in Lemma \ref{lem:typical_paths}.
\end{proof}

We can repeat the same kind of arguments for point (i) in Proposition \ref{prop:concentration_G} however we obtain for the time being a slightly weaker statement that the one expected. Because of the dependance in $\e$ in the annealed probability of \eqref{eq:clt_drift} and \eqref{eq:clt_weights}, applying the first moment arguments will not establish the concentration property with high probability but only $1 - \e - o(1)$. To obtain a high probability result we need to apply a higher order Markov's inequality, which is also the key argument to obtain the uniform statements.

\subsection{Bootstrapping annealed bounds with parallel chains}

Arguments used so far above establish errors on annealed error in $o(1)$, nay in $\e + o(1)$. Following ideas from \cite{bordenave2018random,bordenave2019cutoff}, one strategy to improve these error bounds consists in using higher order moments in Markov's inequality, which leads to an argument of "parallelizing chains" on the same environment. Namely for all $\e > 0$ and $k \geq 1$, for any trajectorial event $A$
    \begin{align}\label{eq:higher_markov_typical}
        \forall x \in V: \bP \sbra{\bfP_{x} \sbra{A} > \e} &\leq \frac{\bP_x \sbra{A}^{k}}{\e^{k}} = \e^{-k} \, \bP_{x} \sbra{ \bigcap_{i=1}^{k} \{ X^{(i)} \in A \}}
    \end{align}
    and 
    \begin{align}
        \bP \sbra{\max_{x \in V} \bfP_{x} \sbra{A} > \e} &\leq \frac{1}{\e^{k}} \bE \sbra{\sum_{x \in V} \bfP_x \sbra{A}^{k} } \nonumber \\
        &\leq \frac{n}{\e^{k}} \max_{x \in V} \bP_{x} \sbra{ \bigcap_{i=1}^{k} \{ X^{(i)} \in A \} }, \label{eq:higher_markov}
    \end{align}
where $X^{(1)}, \ldots,  X^{(k)}$ are $k$ versions of the chain $X$ generated independently conditional on the same environment. These $k$ trajectories can be generated sequentially with the environment, sampling the environment only when exploring parts of the environment not already visited by the previous chains. The idea will then be to argue that the chains are unlikely to follow each other, so they end up exploring disjoint parts of the environments. If the event $A$ depends only on what happens after the trajectories separate from each other (hence the consideration of shifted trajectories), we can then expect the annealed probability to decouple and become $\bP \sbra{A}^{k}$. If the event $A$ has probability below $f(\e)$ for some function $f$, this can be sufficient to beat the $\e^{-k}$ or $\e^{-k} n$ factor.

As an example of application, we can start proving point (i) in Proposition \ref{prop:concentration_G}, which also illustrates some of the arguments used in the sequel.

\begin{proof}[Proof of (i) in Proposition \ref{prop:concentration_G}]
    Let $\e > 0$, $t = O(\log n)$ and $x \in V$. For $C_{\LR} = C_{\LR}(\e)$ and $C_h=C_h(\e) >0$ to determine, consider the set of paths
    \begin{equation}\label{eq:Gamma_ent}
        \Gamma_{\Ent} := \left\{ \frp \ \left| \ \begin{array}{c} \abs{\abs{\xi(\frp)} - \mathscr{d} \abs{\frp} } \leq C_{\LR}\sqrt{\abs{\frp}} \\
            \abs{- \log w_{\frp_0,R,L} (\xi(\frp)) - h \abs{\frp} } \leq C_h \sqrt{\abs{\frp}} \end{array} \right. \right\}
    \end{equation}
    which can be paths either in $G^{\ast}$ or $\cG$. The goal is thus to show $\bP_x \sbra{(X_0) \cdots X_{t} \notin \Gamma_{\Ent}} = o(1)$. Consider $k$ versions $X^{(1)}, \ldots, X^{(k)}$ of the chain $X$ generated independently conditional on the same environment, with $k \gg 1, k = O(\log \log n)$. Then using the same coupling as in Section \ref{subsec:coupling} these can be coupled with $k$ versions $\cX^{(1)}, \ldots, \cX^{(k)}$ of the chain $\cX$. For $j \in [k]$ let $\cG(j,t)$ denote the union of the first $j$ trajectories up to time $t$, along with their long-range neighbourhoods up to depth $L$. Compared to the case of a single chain, the number of visited vertices gets multiplied by $k$ at most, so the proof of Lemma \ref{lem:coupling_quasiT} shows the chains $X^{(1)}, \ldots, X^{(k)}$ can be coupled with high probability, along with their long-range neighbourhood up to length $L = O(\log \log n)$, up to time $t = O(\log n)$, to coincide with $\cG(k,s+t)$. Now the definitions of the long-range distance and of weights (\eqref{eq:def_weights_G} and \eqref{eq:def_weights_quasiT}), which are measurable with respect to this neighbourhood, were exactly made to coincide in the finite and quasi-tree settings. Thus if the coupling succeeds $(X_0 \cdots X_{t}) \in \Gamma_{\Ent}$ if and only in $\cX_{0} \cdots \cX_{t} \in \Gamma_{\Ent}$. Letting for $j \in [k]$:
    \begin{equation*}
        B_j := \bigcap_{i=1}^{j} \left\{ \cX^{(i)}_0 \cdots \cX^{(i)}_{t} \notin \Gamma_{\Ent} \right\}
    \end{equation*}
    \eqref{eq:higher_markov_typical} and the coupling imply
    \begin{equation*}
        \bP_{x} \sbra{X_0 \cdots X_{t} \notin \Gamma_{\Ent}} \leq \e^{-k} \, \bP \cond{B_k}{\iota(O) = x} + o(1),
    \end{equation*}
    so it suffices to show $\bP_O \cond{B_k}{\iota(O) = x} = o(\e^{k})$. We consider in fact a slightly modified event. Let $l := (\log \log n)^3$ and consider the set of quasi-trees $\cA_l = \cA_{l}(\alpha, q_0)$ for which every length $l$ path from the root contains a proportion at least $\alpha$ of vertices for which the probability to escape to infinity outside $\frp$ is lower bounded by $q_0$. Recalling Proposition \ref{prop:escape_dense} and Remark \ref{rk:escape_notonpaths}, Lemma \ref{lem:escape_stationary} implies the existence of constants $q_0, \alpha > 0$ and $C > 0$ such that all vertices in $\cG(j,t)$ have their subquasi-tree in $\cA_l$ with probability at least $1 - t e^{-C l}$. Since $t = O(\log n)$ and $l = (\log \log n)^3$ this probability is $1 - o(1)$. Thus we can suppose that the event $B_{j}$ also contains the event that all vertices of $\cG(j,t)$ have their subquasi-tree in $\cA_l$. 
	Secondly, consider the $l$-th regeneration time $T^{(i)}_l$ of the $i$-th chain. By the stretched exponential tails of regeneration times \ref{lem:regeneration_tails} and Chernoff's bound, there exist constants $\alpha \in (0,1]$ and $C > 0$ such that for all $i \leq k$
	\begin{equation*}
		\bP_O \cond{T^{(i)}_l > t_0}{\iota(O) = x} \leq e^{- t_0^{\alpha} + C l}
	\end{equation*}
	Thus considering $t_0 := (\log \log n)^\beta$, for $\beta > 3/\alpha$ large enough makes the right hand side $o(k^{-1})$. Consequently by union bound we obtain that $\bP_O \sbra{\exists i \leq k: T^{(i)}_l > t_0} = o(1)$, so we can also incorporate to the events $B_j$ the event that $T^{(i)}_l \leq t_0$ for $i \leq j$.
    
    Next we claim that $\bP_O \cond{B_1}{\iota(O) = x} \leq \e / 2 +  o(1)$ and $\bP_O \cond{B_{j}}{\iota(O) = x, B_{j-1}} \leq \e / 2 + o(1)$ for all $j \in [2, k]$. Noting that
    \begin{equation*}
        \bP_O \cond{B_k}{\iota(O) = x} = \bP_O \cond{B_1}{\iota(O) = x} \prod_{j=2}^{k} \bP_O \cond{B_j}{\iota(O) = x, B_{j-1}},
    \end{equation*}
    the claims imply that $\bP \cond{B_k}{\iota(O) = x} = (\e / 2 + o(1))^{k} = o(\e^{k})$ as desired.
   
    Let us prove the claims. The first claim is the consequence of Proposition \ref{prop:drift} and Lemma \ref{lem:concentration_weights_QT}, provided the constants $C_{\LR}, C_h$ are chosen adequatly. To prove the second claim, let $j \geq 2$ and suppose $B_{j-1}$ holds. We argue that $\cX^{(j)}$ has quenched probability at least $1 - o(1)$ to exit $\cG(j-1, t)$ for the last time by time $t_0$. Consider then the loop-erased chain $(\xi^{(j)}_m)_{m}$ obtained from $\cX^{(j)}$. Thanks to the first event we added to $B_{j-1}$, for any long-range path $\zeta$ of length $l$ contained in $\cG(j-1,t)$,
    \begin{equation*}
        \bfP_O \cond{\xi^{(j)}_l = \zeta_l}{B_{j-1}} \leq e^{- C l}
    \end{equation*}
    for some constant $C > 0$. Since $\cG(j-1,t)$ is made of the long-range neighbourhood of the first $j-1$ trajectories, it contains at most $k t \Delta^{R(L+1)}$ long-range paths of length $l$. Thus by union bound 
    \begin{equation*}
        \bfP_O \cond{\xi^{(j)}_l \in \cG(j-1,s+t)}{B_{j-1}} \leq k t \Delta^{R(L+1)} e^{- C l} = e^{-C (\log \log n)^3 + O(\log \log n)^{2}} = o(1).
    \end{equation*}
	On the other hand, the second event added to $B_j$ implies that $T^{(j)}_l \leq t_0$ so the chain must in particular have made the first $l$ steps of its loop-erased trace before time $t_0$. Combined with what precedes we deduce that conditional on $B_{j-1}$, the chain $\cX^{(j)}$ leaves the graph $\cG(j-1,t)$ before time $t_0$ with probability $1-o(1)$. 

    Let $L^{(j)}$ denote now this last exit time. Conditional on $L^{(j)}$ and $\cX_{L^{(j)}} = u$, the quasi-tree that contains the subsequent trajectory needs then to be generated according to the measure $\bQ_u$ \eqref{eq:Qu}. Using Remark \ref{rk:comparison_Qu} upper bounds under this law only differ by a constant factor from upper bounds with the usual law. We can thus apply Proposition \ref{prop:drift} and Lemma \ref{lem:concentration_weights_QT} conditional on $L_j$ to obtain that 
    \begin{equation*}
        \bP_O \cond{(X^{(j)}_{L_j} \cdots X^{(j)}_{L_j +t}) \notin \Gamma_{\Ent}}{B_{j-1}, \iota(O) = x} \leq \e/2 + o(1).
    \end{equation*}
    Now by what precedes $d_{\LR}(\cX^{(j)}_0, \cX^{(j)}_{L_j}) \leq l$ with probability $1-o(1)$ conditional on $B_{j-1}$, hence
    \begin{equation*}
        d_{\LR}(\cX^{(j)}_0, \cX^{(j)}_t) \in [\mathscr{d} t - C_{\LR} \sqrt{t},\mathscr{d} t - C_{\LR} \sqrt{t} + l ]
    \end{equation*}
    and 
    \begin{equation*}
        - \log w_{\cX_0,R,L}(\xi(\cX^{(j)}_0 \cdots \cX^{(j)}_t)) \in [ht - C_h \sqrt{t}, ht + C_h \sqrt{t} + l w_{\max}],
    \end{equation*}
    where $w_{\max}$ denotes an upper bound on all the weights contained in $\cG(j-1,t)$. For any realization of the quasi-tree we can bound $w_{\max} \leq O(RL)$ as in the proof of Lemma \ref{lem:concentration_weights_QT}. All in all, since $l w_{\max} = o(\sqrt{t})$, this ensures $(\cX^{(j)}_0, \cX^{(j)}_t) \in \Gamma_{\Ent}$ with probability $1 - \e /2 - o(1)$, proving the claim.
\end{proof}

\subsection{From $o(1)$ to $o(1/n)$ bounds}

\subsubsection{Quasi-tree with a cycle}

The rest of this section will now be devoted to the proof of the uniform, worst-case statements. The main ideas are the same as those used above. From \eqref{eq:higher_markov}, we only need to push them yet a little further to make error bounds on annealed probabilities $o(1/n)$. From the proof of (i) in Proposition \ref{prop:concentration_G} two arguments require a modification: the coupling with the a quasi-tree and the use of escape probabilities.

Let us focus on the first point. It seems quite obvious that the coupling with quasi-trees cannot be improved beyond to a $o(1/n)$ error as some vertices do not have a quasi-tree like neighbourhood. However from the proof of Lemma \ref{lem:coupling_quasiT} this error can be achieved allowing just one cycle. We are thus led to consider a modified quasi-tree model which contains at most one long-range cycle. 

Let $(u_i)_{i=0}^{l}, (v_i)_{i=0}^{l}, l \geq 0$ be two sequences of vertices in $V$ such that $u_{i+1} \in \BSR^{+}(v_i, \infty)$ for all $i \in [0,l-1]$ and $\BSR^{+}(v_l, R) \cap \BSR^{-}(u_0, R) \neq \emptyset$. For every $i$ and vertex $z \in \BSR^{+}(u_i,\infty) \smallsetminus \{u_i, v_i \}$, add an edge $(z, \eta(z))$ and grow a one-sided quasi-tree rooted at $\eta(z)$. We call the oriented graph $\cG = (\cV, \cE)$ obtained a quasi-tree with a cycle $((u_i,v_i))_{i =0}^{l}$. As a usual quasi-tree, it is given by maps $\iota, \eta$ such that $\iota$ identifies vertices of $\cV$ with vertices in $V$, while $\eta: \cV \rightarrow \cV$ is an involution. Here $\eta$ is obtained by the corresponding maps in the quasi-trees outside the cycle whereas for the cycle we set $\eta(u_i) := v_{i+1}$ for all $i = 0, \ldots l-1$. A root $O$ can be chosen, which does not to be on the cycle. The definition of the Markov chain $\cX$ \eqref{eq:free_transitions} extends directly to this setting.

The coupling of Section \ref{subsec:coupling} gives a natural way to couple $\cX$ on a random realization of $\cG$ with the Markov chain $X$: the rejection scheme is used until the first occurence of a cycle, after which cycles are ignored in the construction of $\cG$. The stochastic comparison in the proof of Lemma \ref{lem:coupling_quasiT} still holds, but using now that 
\begin{equation}\label{eq:binomial_2}
     \bP \sbra{Z \geq 2} \leq \frac{m^{4} \Delta^{2R}}{n^{2}}
\end{equation}
if $Z$ is a binomial $\mathrm{Bin}(m, m \Delta^{R} / n)$, we deduce that for $t = n^{o(1)}$, the chains $X$ and $\cX$ can be coupled so that the chains coincide and $B_{\LR}(X_s,L)$ and $B_{\LR}(\cX_s,L)$ are isomorphic for all $s \leq t$ with probability $1 - o(1/n)$. 

More generally, let $k \geq 1$ and $X^{(1)}, \ldots, X^{(k)}$, resp. $\cX^{(1)}, \ldots, \cX^{(k)}$, $k$ versions of the chain $X$, resp. $\cX$ generated independently conditional on the same environment. Letting $G^{\ast}(k,t), \cG(k,t)$ be the environments explored by these trajectories up to time $t$ along with their $L$-long-range neighbourhoods, the same arguments as above implies that 
\begin{equation}\label{eq:coupling_uniform}
    \bP \sbra{\begin{array}{c} \text{$G^{\ast}(k,t)$ and $\cG(k,t)$ are isomorphic} \\ \forall i \in [1,k], \forall s \leq t: X^{(i)}_s = \iota(\cX^{(i)}_s) \end{array}} = 1 - o((kt)^{4}/n^{7/4}) = o(\e^{k} / n)
\end{equation} 
if for instance $k = \lfloor \log n / 2 \log \e^{-1} \rfloor$.

\subsubsection{An essentially uniform bound on escape probabilities}

The second argument that needs to be adapted is the use of escape probabilities, which can be expected to of constant order only for typical states. Consider $\cG$ to be a random realization of a quasi-tree with a cycle as described above and $\cX$ the associated Markov chain. Conditional on the cycle, the quasi-trees that are added to it are generated according to the model studied in Sections \ref{section:QT1} - \ref{section:QT3}, therefore the asymptotic analysis of the chain $\cX$ directly extends to that case, conditional on the cycle. The chain ultimately leaves the cycle, after which it stays in a genuine quasi-tree. Given a vertex $x$ which is not on the cycle, let $\qesc(x)$ denote the probability to escape in the corresponding quasi-tree. If $x$ is on the cycle, there is no quasi-tree at $x$, so let $\qesc(x)$ denote the probablity of escaping through one of the vertices that are in the same component as $x$. Using that components of $V_1$ have size at least $3$ \ref{hyp:cc3}, these escape probabilities can be bounded exactly as in Section \ref{section:QT1}, up to a change of constants, conditional on the cycle. Thus from Proposition \ref{prop:escape_gumbel}, there exists $\beta \geq 1$ such that for all $u \in \cV$ not on the cycle,
    \begin{equation}\label{eq:escape_uniform}
        \bP \cond{\qesc(u) < \frac{1}{(\log n)^{\beta}}}{\cG \smallsetminus \cG_u} = O(1/n^2).
    \end{equation}
Combined with Lemma \ref{lem:escape_stationary}, we see that on a time scale $t = O((\log n)^{a}), a \geq 1$, the probability that the chain $\cX$ meets a vertex which does not satisfy the lower bound on escape probability above is still $o(1/n)$. This makes the chain sufficiently unlikely to meet such vertices that the chain $X$ cannot reach such quasi-tree-like neighbourhoods either, thanks to the the coupling described above. Informally, this allows us to proceed as if escape probabilities were everywhere lower bounded by $1/(\log n)^{\beta}$ and reason conditional on the environment, as we did in the reversible case.

\subsection{Nice paths from arbitrary states}

A first application of these uniform lower bounds is the following.

\begin{proof}[Proof of Lemma \ref{lem:QTlike_uniform}]
    We prove the first point. Let $L = C_L \log \log n$ for a fixed constant $C_L > 0$, $t \ll n^{1/16}$ and $x \in V$. By Remark \ref{rk:uniform_shifted} it suffices to consider the case of $s = C (\log n)^{a}$ for some $a > 0$ to determine. Using the coupling described above (with only one chain here, ie $k=1$) we can couple the trajectory of $X$ and its whole $L$ long-range neighbourhood with that of $\cX$. By \eqref{eq:coupling_uniform} the coupling succeeds up to time $s+t$ with probability $1- o(1/n)$, hence for $u \in V$
    \begin{multline*}
        \bP_u \sbra{\exists s' \in [s,s+t]: \BLR(X_{s'}, L) \text{ is not quasi-tree-like} } \\ = \bP_O \cond{ \text{$\exists s' \in [s,s+t]: \cX_{s'}$ is at distance less than $L$ from the cycle}}{\iota(O) = u} + o(1/n).
    \end{multline*}
    Now let $\beta \geq 1$ be as in \eqref{eq:escape_uniform} and consider the set
    \begin{equation*}
        \mathrm{Esc} := \{ x \in \cG \ | \ \qesc(x) \geq (\log n)^{- \beta} \}.
    \end{equation*}
    Write $\tau_{\mathrm{Esc}^{c}}$ for the hitting time of the complement set of $\mathrm{Esc}$. The number of vertices in the $L$ long-range neighbourhood of the trajectory being bounded by $O(s+t) \Delta^{R(L+1)} = O(n^{1/16})$, union bound and \eqref{eq:escape_uniform} imply
    \begin{equation*}
        \bP_O \cond{\tau_{\mathrm{Esc}^{c}} \leq s+t}{\iota(O) = u} = o(1/n).
    \end{equation*}
	We are being a bit sketchy here, but the argument can be formalized using the Ulam labelling.
    On the other hand, when the chain is at a vertex of $\mathrm{Esc}$ its long-range distance to the cycle has a quenched probability at least $(\log n)^{-\beta}$ to increase by $1$ indefinitely. Thus if we consider the sequence of times $(S_k)_{k \geq 1}$ at which this happens, for all $x \in \cV, s \geq 0$, 
    \begin{equation*}
    \bfP_{x} \sbra{\tau_{\mathrm{Esc}^{c}} > s+t, S_1 > s} \leq e^{-s / (\log n)^{\beta}}.
    \end{equation*} 
    Consequently by union bound and Strong Markov's property
    \begin{multline*}
        \bfP_{x} \sbra{\text{$\exists s' \in [s,s+t]: \cX_{s'}$ is at distance less than $L$ from the cycle}, \tau_{\mathrm{Esc}^{c}} > s+t} \\ \leq \bfP_{x} \sbra{S_L > s, \tau_{\mathrm{Esc}^c > s +t}} \leq L e^{- s / L (\log n)^{\beta}}.
    \end{multline*}
    Since $s = (\log n)^{a}$, choosing $a > \beta + 1$ makes these two bounds $o(1/n)$. The first moment argument \eqref{eq:higher_markov} concludes the proof of the first point in the general case. 
	
	Finally point (ii) can be proved with similar arguments.
\end{proof}

It remains to prove (ii) in Lemma \ref{lem:typical_paths} and Proposition \ref{prop:concentration_G}. The argument is similar and uses the same ideas as above.

\begin{proof}[Proof of (ii) in Lemma \ref{lem:typical_paths} and Proposition \ref{prop:concentration_G}]
    Consider the concentration of the drift and entropy. Let $x \in V$, $\e > 0$ and $k := \lfloor \log n / 2 \log \e^{-1} \rfloor$. As for point (i) we combine the higher Markov inequalities \eqref{eq:higher_markov} and the coupling with $k$ chains on a quasi-tree with a cycle. As a cycle is allowed the error due to the coupling is $o(1/n)$ by \eqref{eq:coupling_uniform} so we need only focus on the infinite setting: letting
    \begin{equation*}
        B_j := \bigcap_{i=1}^{j} \{ \cX^{(i)}_{s} \cdots \cX^{(i)}_{s+t} \notin \Gamma_{\Ent} \}.
    \end{equation*}
    for $j \in [k]$, with $\Gamma_{\Ent}$ as in \eqref{eq:Gamma_ent} it suffices to prove $\bP_O \cond{B_k}{\iota(O) = x} = o(\e^{k} / n)$. Then we claim that $\bP_O \cond{B_1}{\iota(O) = x} \leq \e^3 / 2 +  o(1)$ and $\bP_O \cond{B_{j}}{\iota(O) = u, B_{j-1}} \leq \e^3 / 2 + o(1)$ for all $j \in [2, k]$. With our choice of $k$ this will imply $\bP_O \cond{B_k}{\iota(O) = x} = (\e^3/2 + o(1))^k = o(\e^{k} / n)$ as needed.

    As we mentionned already, results proved for genuine quasi-trees extend to the case where an additional cycle is present. In particular, the first regeneration time and level, defined to occur outside the cycle, have stretched exponential or exponential tail respectively, conditional on the cycle, while the remaining regenerations occurs on genuine quasi-trees. This is sufficient to establish that Proposition \ref{prop:drift} and Lemma \ref{lem:concentration_weights_QT} still hold in this case, implying that $\bP \cond{B_1}{\iota(O)=u} \leq \e^{3}/2$ for a good choice of $C_h, C_{\LR}(\e)$. 

    For the second claim, we argue as for (i) by showing the chains are unlikely follow each other. Let $\beta \geq 1$ be as in \eqref{eq:escape_uniform}, consider the set
    \begin{equation*}
        \mathrm{Esc} := \{ x \in \cG \ | \ \qesc(x) \geq (\log n)^{- \beta} \}.
    \end{equation*}
    and let $\tau_{\mathrm{Esc}^{c}}$ be the first exit time by any of the chains. As we argued above in the proof of Lemma \ref{lem:QTlike_uniform}, $\bP_O \cond{\tau_{\mathrm{Esc}^{c}} \leq s + t}{\iota(O) =u} = o(1/n)$. On the other hand, while on this set the chains have a quenched probability at least $(\log n)^{\beta}$ to have their long-range distance increased by $1$ permanently, hence for any $j \in [k]$, letting $(S^{(j)}_k)_k$ be the sequence of times at which this occurs for the $j$-th chain,
    \begin{equation}\label{eq:distance_uniform}
        \bfP_O \sbra{S^{(j)}_l > s, \tau_{\mathrm{Esc}^{c}} > s+t} \leq l e^{- C_1 s / l (\log n)^{\beta}}
    \end{equation}
    for some constant $C_1 > 0$. Taking $a > 1 + 2 \beta$ in $s = (\log n)^{s}$ and $l = (\log n)^{\beta} (\log \log n)^3$ the right hand side is $o(1)$. 
    
    Consider now the loop-erased chain $(\xi^{(j)}_m)_{m}$ obtained from $\cX^{(j)}$. If all vertices visited by the chain are in $\mathrm{Esc}$, and using that components of $V_1$ have size at least $3$, the probability to follow a given path is exponentially small in the length. Since $\cG(j-1,s+t)$ contains at most one long-range cycle this subgraph contains at most $2 k (s+t) \Delta^{R(L+1)}$ long-range paths, so by union bound 
    \begin{equation*}
        \bfP_O \sbra{\xi^{(j)}_l \in \cG(j-1,s+t), \tau_{\mathrm{Esc}^{c}} > s+t} \leq 2 k (s+t) \Delta^{R(L+1)} e^{-C_2 l / (\log n)^{\beta}}.
    \end{equation*}
    for some constant $C_2 > 0$. For $l = (\log n)^{\beta} (\log \log n)^3$ this is $o(1)$. 

    Combining the two previous observations, we deduce that conditional on $B_{j-1}$, the last exit of $\cG(j-1,s+t)$ by the trajectory $(\cX^{(j)}_{t'})_{t' \leq s+t}$ occurs before $s$, leaving at least $t$ subsequent steps. Letting $L^{(j)}$ denote this last exit time, we can reason conditional on $L_{j}$ and use Proposition \ref{prop:drift} and Lemma \ref{lem:concentration_weights_QT} about shifted trajectories to deduce the claim. 

    Finally, the proof of Lemma \ref{lem:typical_paths} uses the same arguments, considering the set of paths $\Gamma(R,L,M)$ instead of $\Gamma_{\Ent}$. 
\end{proof}

\section[Nice paths]{Approximation by nice paths: proof of Proposition \ref{prop:nice_approx}}\label{section:nice}

We move to the second part of the proof of Theorem \ref{thm:cutoff}. Most of this section is identical to the reversible case so we will skip some of the proofs that can be found in \cite{cutoff_reversible}. 

\subsection{Forward neighbourhood}

Nice paths between $x$ and $y$ will have their first steps and last steps contained respectively in some quasi-tree-like neighbourhoods of $x$ and $y$. We define here the forward neighbourhood of $x$.

Let $x \in V$, $l \geq 1$ an integer and $w_{\min} \geq 0$. The forward graph $K(x,l, w_{\min})$ is designed essentially as a "spanning quasi-tree" of the ball $\BLR(x,l)$, obtained by exploring this ball algorithmically but giving priority to paths with large weights and truncating whenever cycles are encountered. This process will thus build iteratively a sequence $(K_m)_{m=0}^{\tau}$ of subsets of $\BLR(x,l)$, until it stops at a random time $\tau$ to yield $K(x,l, w_{\min}) := K_{\tau}$. Unless the procedure is initiated at a vertex $x$ whose ball $\BLR(x,L)$ is not quasi-tree-like, $K_m$ remains at all time quasi-tree-like. In this case for every long-range edge $e \in E_m$ there exists a unique long-range path $\xi(e)$ from $x$ to $e$ contained in $K_m$. Define its cumulative weight as $\hat{w}(e) := w_x(\xi(e))$. Because weights require the knowledge of $(L-1)$-neighbourhoods, the exploration queue will consist in subsets $E_m$ of long-range edges for which the whole long-range $(L-1)$-neighbourhood is contained in $K_m$, so that cumulative weights can be computed from $K_m$ only. Finally, a constraint of minimal weights is added during the procedure, in order to keep the number of vertices explored as $o(n)$.

\paragraph{Exploration of the forward neighbourhood}

The procedure is initiated with $K_0 := \BLR(x,L)$. If $K_0$ contains a long-range cycle, $E_0 := \emptyset$ and the procedure stops. Otherwise let $E_0$ be the set of long-range edges at distance $0$ from $x$. Then for all $m \geq 0$ the $(m+1)$-th step goes as follows: 
\begin{enumerate}
	\item Among all long-range edges $e$ in $E_m$ at long-range depth at most $l$ from $x$ in $K_m$ and such that $\hat{w}(e) \geq w_{\min}$, pick the edge $e_{m+1}$ with maximal cumulative weight, using an arbitrary ordering of the vertices to break ties. If there is no such edge, the procedure stops.
	\item Explore the depth-$L$ neighbourhood of $e_{m+1}$: for each descendant $z \in \partial K_m$ at long-range distance $L-1$ from $e_{m+1}$, reveal $\eta(z)$. This exploration phase stops if a revealed edge violates the quasi-tree structure: this occurs if for some $z$ the small-range ball $\BSR(\eta(z), R)$ has a non-empty intersection with $K_m$ or one of the previously revealed balls.
	\item If the previous exploration phase stopped because of intersecting small-range balls, then $E_{m+1} := E_m \smallsetminus \{ e_{m+1} \}$ and $K_{m+1} := K_m$. If it stopped because of an intersection with $K_m$, let $Z_{m}$ be this intersection. Then $E_{m+1}$ is obtained by deleting from $E_m$ every long-range edge which has either a descendant or an ancestor in $Z_m$, as well as the edge $e_{m+1}$, and set $K_{m+1} := K_m$. Finally, if the exploration ended without a violation of the quasi-tree structure, add the subsequent long-range edges of $e_{m+1}$ to $E_{m+1}$, whereas the newly revealed vertices are added to $K_{m+1}$.
\end{enumerate}

When the procedure ends, the set $E_{\tau}$ consists by construction of edges at long-range distance $l - L$ from $x$, which contain no long-range cycles in their $(L-1)$-neighbourhood and whose weights are measurable with respect to $K(x,l,w_{\min})$. The two following lemmas were proved for the reversible case.

\begin{lemma}[{\cite[Lemma 8.1]{cutoff_reversible}}]
	Let $\kappa_m$ denote the number of long-range edges revealed during the first $m$ steps $m$ and $\kappa(x,l, w_{\min}) := \kappa_{\tau}$ the total number of long-range edges revealed during the construction of $K(x,l, w_{\min})$. Suppose $\BLR(x,L)$ is quasi-tree-like so that $\tau \geq 1$. There exists a constant $C > 0$ such that for all $m \in [1,\tau]$
	\begin{equation}\label{eq:bound_kappa_m}
		\hat{w}(e_m) \leq C \frac{l \Delta^{RL} }{\kappa_m}.
	\end{equation}
	In particular
	\begin{equation}\label{eq:bound_kappa}
		\kappa(x,l, w_{\min}) \leq C \frac{l \Delta^{RL} }{w_{\min}}.
	\end{equation}
\end{lemma}

\begin{lemma}[{\cite[Lemma 8.2]{cutoff_reversible}}]\label{lem:forward_martingale}
	If $m > \tau$, let $\mathrm{cycle}(e_{m+1})$ be the event that the exploration of the $(L+1)$ neighbourhood of the long-range edge $e_{m+1} \in E_{m}$ considered at the $(m+1)$-th step revealed a cycle. Let $\e \in (0,1)$ and consider the following process. If $\BLR(x,L)$ is not quasi-tree-like, let $W_m := 1$ for all $m \geq 0$, otherwise set $W_0 := 0$ and for all $m \geq 0$ define
	\begin{equation*}
		W_{m+1} := W_{m} + \left( \hat{w}(e_{m+1}) \wedge \e / 2 \right) \II_{m < \tau} \II_{\mathrm{cycle}(e_{m+1})}.
	\end{equation*}
	This is the total cumulative weight of edges that violated the quasi-tree structure at step $m+1$ and that are below $\e / 2$. Suppose $l = O(\log n)$, $R,L = O(\log \log n)$ and that $w_{\min} \geq e^{ (\log \log n)^{3}} / n$. Then for all $s=s(n)$, with high probability, for all $x \in V$,
	\begin{equation}
		W_{\tau} \leq W_{s} + \e.
	\end{equation}
\end{lemma}

\subsection{Nice paths: definition}

We now define nice paths in order to prove Proposition \ref{prop:nice_approx}.
For the rest of this section set
\begin{equation*}
	R := C_R \log \log n, \qquad C_L := C_L \log \log n, \qquad M := (\log \log n)^{\kappa}.
\end{equation*}
where $C_R, C_L > 0$ and $\kappa \geq 1$ are constants chosen large enough so that the conclusions of Lemma \ref{lem:typical_paths} holds. 
Fix $\epsilon \in (0,1)$ for the rest of this section. In the sequel, we consider several constants $C_0, \ldots, C_5$, defined in terms of the constants $C_{\LR}(\e), C_h(\e)$ of Proposition \ref{prop:concentration_G}, which can in particular depend on $\e$. They are indexed in the reverse order in which they are fixed, so $C_0$ is chosen after $C_1$, which is chosen after $C_2$, etc. Let 
\begin{equation}\label{eq:nice_parameters}
	\begin{gathered}
	s := \left\lfloor \frac{\log n}{10 \log \Delta} \right\rfloor \wedge \left\lfloor \frac{\log n}{10 h} \right\rfloor, \qquad t := \left\lfloor \frac{\log n}{h} + C_0 \sqrt{\log n} \right\rfloor, \\
	l_1 := \mathscr{d} (t-s) - C_4 \sqrt{t}, \qquad w_{\min} := e^{- h (t-s) - C_h \sqrt{t}} \qquad w_{\max} := e^{-h t + C_1 \sqrt{t}}.
	\end{gathered}
\end{equation}

Given $x,y \in V, r \in \bN / 2$, $L \leq r \leq M$ and $l_3 \in [\mathscr{d} s - C_5 \sqrt{s}, \mathscr{d} s + C_5 \sqrt{s}]$, consider the following three-stage exploration of the environment: 
\begin{enumerate}
	\item Explore $K=K(x_0,l_1, w_{\min})$ as explained in the previous section. Let $E := E_{\tau}$ be the set of long-range edge remaining in the exploration queue at the end of the procedure, and consider the set $E'$ of boundary vertices at long-range distance $l_1$ from $x_0$, whose image under $\eta$ is yet to determine. 
 	\item The backward neighbourhood is $B = B(r,r+s,l_3) := B_{\scP}^{-}(y,r+s) \cap \BLR^{-}(y,l_3)$. 
  	\item Finally, reveal everything else.
\end{enumerate}

It will be crucial in the sequel to control the numbers $N_1, N_2$ of long-range edges revealed during the two first stages. By definition $N_1 := \kappa(x,l_1, w_{\min})$. Observe that for any $\e' < \frac{h}{10 \log \Delta} \wedge \frac{1}{10}$, $h (t-s) < (1-\e') \log n + O (\sqrt{\log n})$. Thus by \eqref{eq:bound_kappa}, 
\begin{equation*}
	N_1 = O( l_1 \Delta^{RL} e^{h (t-s) + C_h \sqrt{t}}) = O(\log n \ \Delta^{C_R C_L (\log \log n)^2} n^{1 - \e'} e^{C \sqrt{\log n}})
\end{equation*}
for some constant $C$. The cardinality of $B$ is bounded by that of $B_{\scP}^{-}(y,r+s)$, which is $O(\Delta^{r+s}) = O(n^{2/10})$ by choice of $r \leq M$ and $s$. 
All in all, for any $\e' < \frac{h}{10 \log \Delta} \wedge \frac{1}{10}$,
\begin{equation}\label{eq:stages}
	\begin{split}
		&N_1 = O(n^{1-\e'}), \\
		&N_2 = O(n^{1/10}).
	\end{split}
\end{equation}

Let $\cF_{r,l}$ be the $\sigma$-algebra generated by the long-range edges revealed during the two first stages. Since $K$ is quasi-tree-like, a non-backtracking long-range path $\xi$ from $x$ to $E'$ entirely contained in $K$ must cross a unique edge of $E$. Let $\xi_E$ denote this edge and define $w_E(\xi) := w_{x,R,L}(\xi_1 \cdots \xi_E)$. This is essentially the weight $w_x(\xi)$, but where last steps of the path were truncated to keep a weight that is $\cF_{r,l}$-measurable. 

In $B$, let $F'$ be the set of boundary vertices which are at long-range distance exactly $l$ from $y$, for which the shortest long-range path to $y$ is unique and has a tree-like neighbourhood in $B_{\scP}(y,r+s)$, that is $B \cap \BLR(z,L)$ contains no long-range cycle for all vertex $z$ on this path. Consider now the set $F$ of long-range edges in $B$ that are at long-range distance $L-1$ from $F'$. If $\xi$ is a non-backtracking long-range path from $F'$ to $y$, let $\xi_F$ denote the unique edge of $F$ crossed by $\xi$ and $\xi_{F+1}$ the subsequent edge. Set $w_F(\xi) := w_{\xi_{F}^{+},R,L}(\xi_{F+1} \cdots \xi_{\abs{\xi} - L+1})$, where $\xi_{F}^{+}$ is the endpoint of $\xi_F$ closest from $y$. Here we truncate the path at both ends: the first steps to have a $\cF_{r,l}$-measurable weight but also the last steps, as the trajectories considered afterwards end at $y$. 

\begin{definition}\label{def:nice_paths}
	Given $r \in \bN/2$, $r \leq M$ and $l \in [\mathscr{d} s - C_5 \sqrt{\log n}, \mathscr{d} s + C_5 \sqrt{\log n}]$ let $\frN^{t}_{r,l}(x,y)$ be the set of length $t$ paths $\frp$ between $x$ and $y$ such that
	\begin{enumerate}[label = (\roman*)]
		\item $\frp \in \Gamma(R,L,M)$
		\item $\frp$ can be decomposed as the concatenation $\frp = \frp_1 \frp_2 \frp_3$ of three paths such that: $\frp_1$ is a path from $x$ to $E'$ entirely contained in $K(x, l_1)$, whose endpoint is the only vertex of $E'$ it contains (which implies that it starts and ends with a long-range step),
		\item $\frp_2$ is a path between $E'$ and $F'$ which starts with a small-range step but ends with a long-range step such that
		\begin{equation*}
			C_3 \sqrt{\log n} \leq \abs{\frp_2} \leq C_2 \sqrt{\log n}.
		\end{equation*}
		for some $C_2, C_3$. 
		In addition, $\frp_2$ contains a regeneration edge outside $K(x,l_1,w_{\min})$ before step $ML$, and the endpoint of $\frp_2$ is the only vertex of $B(r,r+s,l)$ it contains. 		
		\item $\frp_3$ is a path of length $r+s$ from $F'$ to $y$ entirely contained in $B(r,r+s,l)$, which starts and ends with a small-range step and does not contains any regeneration edge in its first $r$ steps,
		\item $w_E(\xi(\frp_1)) w_F(\xi(\frp_3)) \leq w_{\max}$.
	\end{enumerate}
	In the sequel we consider 
	\begin{equation*}
		\scP^{t}_{r,l}(x,y) := \sum_{\frp \in \frN^{t}_{r,l}(x,y)} \scP(\frp).
	\end{equation*}
	The complete set of nice paths is obtained by taking the union over parameters $r,l$, namely 
	\begin{equation*}
		\frN^{t}(x,y) := \bigcup_{\substack{r \leq M \\ r \in \bN + 1/2}} \bigcup_{l = \mathscr{d} s - C_5 \sqrt{s}}^{\mathscr{d} s + C_5 \sqrt{s}} \frN_{r,l}(x,y).
	\end{equation*}
	and the total probability of nice paths	by $\scP^{t}_{\frN}(x,y) := \sum_{\frp \in \frN^{t}(x,y)} \scP(\frp)$.
\end{definition} 

Conditions (i) and (ii) allow to relate the probability of following a nice path with the weight constraint (v): thanks to the tree structure of $K(x, l_1)$, each vertex in $E'$ has a unique ancester edge in $E$. Given $e \in E$, let $E'(e)$ denote the set of vertices in $E'$ with ancester $e$ and recall $\xi(e)$ is the unique non-backtracking long-range path from $x$ to $e$. Similarly for $f \in F$, let $F'(f)$ be the set of vertices of $F'$ from which the unique non-backtracking long-range path to $y$ goes through $f$ and $\xi(f)$ the unique non-backtracking long-range path from $f^{+}$ to $y$ (which thus does not include $f$). Then for a fixed total long-range length $l$:
\begin{align*}
	\sum_{\substack{\frp: \xi(\frp_1)_E = e \\ \xi(\frp_3)_F = f \\ \abs{\xi(\frp)} = l}} \scP(\frp) &\leq \sum_{\substack{\xi: \xi_E = e \\ \xi_F = f, \\ \abs{\xi} = l}} \bfP_{x} \sbra{\xi(X_0 \cdots X_{\tau_{l}})_{\leq l - L+1} = \xi, (X_0 \cdots X_{\tau_{l}}) \in \Gamma(R,L,M)} \\
	&\leq\sum_{\substack{\xi: \xi_E = e \\ \xi_F = f, \\ \abs{\xi} = l}} w_{x,R,L}(\xi). 
\end{align*}
by Lemma \ref{lem:probability_weights}. For each $\xi$ in the sum, 
\begin{align*}
	w_{x,R,L}(\xi) &= w_{x,R,L}(\xi_1 \cdots \xi_E) w_{\xi_{E}^{+},R,L}(\xi_{E+1} \cdots \xi_{F}) w_{\xi_{F}^{+},R,L} (\xi_{F+1} \cdots \xi_{l-L+1}) \\
	&= w_{E}(\xi_1 \cdots \xi_E) w_{\xi_{E}^{+},R,L}(\xi_{E+1} \cdots \xi_{F}) w_{F}(\xi_{F+1} \cdots \xi_{l-L+1})
\end{align*}
Observe now that for fixed $e$ and $f$ the first and third factors are fixed as well and determined by $\xi(\frp_1), \xi(\frp_3)$, so the sum is only over the steps $\xi_{E +1} \cdots \xi_{F-1}$. Since weights sum up to $1$, we can also sum over $l$ and bound
\begin{equation}\label{eq:nice_paths_bound}
	\sum_{\substack{\frp: \xi(\frp_1) = \xi_1 \\ \xi(\frp_3) = \xi_3}} \scP(\frp) \leq w_{\max}
\end{equation}
thanks to the weight constraint (v). 
In words: the total probability of nice paths with prescribed long-range edges in $E$ and $F$ is upper bounded by $w_{\max}$.

\subsection{Nice paths are typical}\label{subsec:nice_typical}

We prove here points (i)-(ii) of Proposition \ref{prop:nice_approx}. The argument is the same for both points and relies on Lemmas \ref{lem:coupling_quasiT}, \ref{lem:QTlike_uniform}, \ref{lem:typical_paths} and Proposition \ref{prop:concentration_G}. Let $t = \log n / h + C_0 (\e) \sqrt{\log n}$. To prove (i) we can consider a fixed starting state $x \in V$ and then need to ensure that the $t$ steps of a trajectory are nice with quenched probability at least $1-\e$, while for point (ii) we need to prove the last $t$ steps of a length $s' + t$ trajectory are nice, for $s' = (\log n)^{a}, a \geq 1$, this time independently of the starting state. The proof thus only differs by which statement of the aforementionned Lemmas and Proposition we use. Let us for instance focus on point (ii). Consider $a \geq 1$ such that the statement of Lemmas \ref{lem:QTlike_uniform}, \ref{lem:typical_paths} and Proposition \ref{prop:concentration_G} hold. There are several properties to check. Since there are finitely many of them, once a property is shown to hold with probability $1-o_{\bP}(1)$ or $1 -\e$ we can automatically assume it is satisfied when checking the remaining properties. Write $t' := s' + t$ to simplify notation. 

We already proved in Lemma \ref{lem:typical_paths} that a failure of the requirement $(X_{s'} \cdots X_{t'}) \in \Gamma(R,L,M)$ occurs with probability $o_{\bP}(1)$. For fixed $r, l$ the very definition of the backward neighbourhood $B(r,r+s,l)$ implies that it necessarily contains the last $r-s$ steps of the trajectory provided they have the prescribed long-range length $l$. Summing over $r$ and $l$, the constraint that $\frp_2$ only has its last endpoint in the backward neighbourhood while $\frp_3$ is contained in it amounts to conditionning by the last regeneration time occurring before $s$. In particular this requires the existence of a regeneration time in the interval $[t'-s-M,t'-s]$ but this is exactly ensured by the fact that $(X_{s'} \cdots X_{t'}) \in \Gamma(R,L,M)$. Assume now this property hold and let $T_{F'}$ denote the last regeneration before time $t'-s$. The remaining obstructions to following a nice paths are: 
\begin{enumerate} 
	\item the first steps of $(X_{s'} \cdots X_{t'})$ are not contained in $K=K(X_{s'},l, w_{\min})$, which occurs if $\BLR(X_{s'}, L)$ is not quasi-tree-like. This occurs with probability $o_{\bP}(1)$ by Lemma \ref{lem:QTlike_uniform}.
	\item from $X_{s'}$, the chain leaves $K$ before it reaches long-range distance $l_1$, which occurs if: 
	\begin{itemize}
		\item the loop-erased trace exits $K$ through the $L$ long-range neighbourhood of an edge $e$ which satisfied $\hat{w}(e) < w_{\min}$: since cumulative weights along a path are non-increasing this implies $- \log w(\xi(X_{s'} \cdots X_{t'-s})) > - \log w_{\min} = h (t-s) + C_h \sqrt{t}$ which occurs with probability less than $\e$ by Proposition \ref{prop:concentration_G}.
  		\item the chain crosses an edge which violated the quasi-tree structure. Recall the process considered in Lemma \ref{lem:forward_martingale}. As mentionned above we can suppose $\BLR(X_{s'},L)$ is quasi-tree-like. Then by Proposition \ref{prop:concentration_G}, the total probability of paths with long-range length $L$ and weights above $\e / 2$ is $o_{\bP}(1)$. Thus up to this $o_{\bP}(1)$ error the quantity $W_{\kappa}$ considered in the lemma exactly counts the probability to exit $K$ at an edge which violated the tree structure. Thanks to the choice of $w_{\min}$ \eqref{eq:nice_parameters} the lemma establishes that with high probability, $W_{\kappa} \leq W_{\lfloor L/2 \rfloor} + \e$ for any value of $X_{s'}$. However $W_{\lfloor L/2 \rfloor} = 0$ as $\BLR(X_{s'},L)$ is quasi-tree-like.
	\end{itemize}
	\item the path from $X_{t'-s-T_{F'}}$ to $y$ is not unique or its long-range length is not in the interval $[\mathscr{d} s - C_5 \sqrt{\log n}, \mathscr{d} s + C_5 \sqrt{\log n}]$. If this path is not unique, this implies in particular that the ball $B(X_{t'-s-M}, t'-s-M)$ is not quasi-tree-like. Since $s + M = O(s)$ Lemma \ref{lem:QTlike_uniform} shows this occurs with probability $o_{\bP}(1)$. For the long-range distance requirement, Proposition \ref{prop:concentration_G} shows that $\abs{\abs{\xi(X_{t'-s} \ldots X_{t'})} - \mathscr{d} s} \leq C_{\LR} \sqrt{s}$ with probability at least $1- \e - o(1)$. Then note that the long-range distance traveled in the intervals $[T_{F'},t']$ and $[t'-s,t']$ differ by at most $M$, hence $\abs{\abs{\xi(X_{t'-s-T_{F'}} \cdots X_{t'})} - \mathscr{d} s} \leq C_{\LR} \sqrt{s} + M \leq C_5 \sqrt{s}$ for a large enough constant $C_5$, using that $M=o(s)$.
	\item No regeneration occured outside outside $K$ during the first $ML$ steps of the intermediate path $\frp_2$ or the latter does not have length $O(\sqrt{t})$. Since the trajectory is in $\Gamma(R,L,M)$ there must be regeneration edge in less than $M$ steps after having first reached $E'$. By the non-backtracking property, this regeneration edge is either outside $K$ or at long-range distance less than $L$ from $E'$, after which the chain is forced to exit $K$. Thus the $L$-th consecutive regeneration edge must be outside $K$. Using again that regeneration edges are spaced by at most $M$ steps, this makes in total $ML$ steps at most until a regeneration occurs outside $K(x,l_1)$. 
	
	For the length requirement observe that the long-range length is sub-additive. Since by definition a nice trajectory decomposes as the concatenation $(X_{s'} \cdots X_{t'}) = \frp_1 \frp_2 \frp_3$ the sub-additivity implies 
	\begin{equation*}
		\abs{\xi(X_{s'} \cdots X_{s'+t})} \leq \abs{\xi(\frp_1)} + \abs{\xi(\frp_2)} + \abs{\xi(\frp_3)}.
	\end{equation*}
	The path $\xi(\frp_1)$ has length $l_1$, whereas $\xi(\frp_3)$ has variable length but from the bounds on $l$ in Definition \ref{def:nice_paths} and \eqref{eq:nice_parameters} we infer that their combined length is
	\begin{equation*}
		\abs{\xi(\frp_1)} + \abs{\xi(\frp_3)} \leq \mathscr{d} t - (C_4 - C_5) \sqrt{t},
	\end{equation*}
	while the intermediate path obviously has length $\abs{\xi(\frp_2)} \leq \abs{\frp_2}$. Choose $C_4 \geq C_5 + 2 C_{\LR}$. Hence if $\abs{\frp_2} < C_3 \sqrt{t}$ with $C_3 := C_{\LR}$, then $\abs{\xi(X_{s'} \cdots X_{s'+t})} < \mathscr{d} t - C_{\LR} \sqrt{t}$, which occurs with probability at most $o_{\bP}(1)$ by Proposition \ref{prop:concentration_G}. 
	To prove the upper bound on $\abs{\frp_2}$, observe $\abs{\frp_1}+s'$ coincides with the first hitting time $\tau_{l_1}$ of long-range distance $l_1$ after step $s'$. From Proposition \ref{prop:concentration_G} we can deduce the existence of $C_2> 0$ such that $\tau_{l_1} - s' \geq t -s - C_2 \sqrt{t} $ with probability at least $1 -\e$. Since $\abs{\frp_3} \geq s$, we deduce that
	\begin{equation*}
		\abs{\frp_2} = t - \abs{\frp_1} - \abs{\frp_3} \leq C_2 \sqrt{t}.
	\end{equation*}
	\item $w_{E}(\xi(X_{s'} \cdots X_{\tau_{l_1}})) w_{F} (\xi(X_{t'-s-T_{F'}} \cdots X_{t'})) > w_{\max}$: let $\xi := \xi(X_{s'} \cdots X_t)$. When deriving \eqref{eq:nice_paths_bound} we used that
	\begin{equation*}
		w_{x,R,L}(\xi) = w_{E}(\xi_1 \cdots \xi_{E}) w_{\xi_{E}^{+},R,L}(\xi_{E+1} \cdots \xi_{F}) w_{F}(\xi_{F+1} \cdots \xi_l)
	\end{equation*}
	wher $\xi_l$ is the last edge of $\xi$. Now by the non-backtracking property of nice paths $\xi(X_{s'} \cdots X_{\tau_{l_1}})$ and $\xi(X_{t'-s - T_{F'}} \cdots X_{t'})$ contain $\xi_1 \cdots \xi_{E}$ and $\xi_{F+1} \cdots \xi_{l}$ respectively so the goal is to show that $w_{E}(\xi_1 \cdots \xi_{E}) w_{F}(\xi_{F+1} \cdots \xi_l) > w_{\max}$ with probability at most $\e$. As we argued for the previous point, Proposition \ref{prop:concentration_G} implies $\tau_{l_1} - s' \geq t_1 := t-s-C_2 \sqrt{t}$ with probability at least $1 - \e$. Since cumulative weights are non-increasing along a path, we deduce that 
	\begin{equation*}
		w_{E}(\xi_1 \cdots \xi_E) \leq w_{x,R,L}(\xi(X_{s'} \cdots X_{t_1})) \leq e^{- t_1 h + C_h \sqrt{t_1}}
	\end{equation*}
	with probability at least $1 - \e$.
	Similarly we know that $T_{F'} \leq t'-s$, hence
	\begin{equation*}
		w_{E}(\xi_1 \cdots \xi_F) \geq w(\xi(X_{s'} \cdots X_{t'-s})) \geq e^{-(t-s) h - C_h \sqrt{t-s}}
	\end{equation*}
	with probability at least $1 - \e - o(1)$. From these two bounds, we deduce
	\begin{equation*}
		w_{\xi_{E}^{+},R,L}(\xi_{E+1} \cdots \xi_{F}) \geq e^{- C_{h} \sqrt{t-s} - C_{h} \sqrt{t_1} - C h \sqrt{t}} \geq e^{-C' \sqrt{t}}
	\end{equation*}
	for some $C' = C'(\e) > 0$. Thus if $w_{E}(\xi_1 \cdots \xi_{E}) w_{F}(\xi_{F+1} \cdots \xi_l) > w_{\max}$ we obtain that
	\begin{equation*}
		w_{x,R,L}(\xi) \geq w_{\max} e^{-C' \sqrt{t}}
	\end{equation*}
	which has probability at most $\e$ if the constant $C_1$ in $w_{\max}$ \eqref{eq:nice_parameters} is taken sufficiently large.
\end{enumerate}

\subsection{Concentration of nice paths}

The proof of the concentration of $\scP_{\frN}^{t}$ is identical to the reversible case so we will only sketch the arguments. Recall $\cF_{r,l}$ is the $\sigma$-algebra generated by the long-range edges revealed during the two first stages. Concentration is established for $r,l$ and $x,y$ fixed then extended by union bound. It is based on the following concentration inequality for functions on the symmetric group, proved specifically for this problem. It is inspired from a concentration inequality by Chatterjee (\cite[Prop. 1.1]{chatterjee2007stein}, see also \cite{chatterjee2007concentration}) that was used to prove the cutoff of non-backtracking chains in \cite{bordenave2019cutoff}. It corresponds to the degree $d=1$ case, but provides better constants. 

\begin{proposition}[{\cite[Corollary 1.1]{cutoff_reversible}}]\label{prop:concentration_symmetric}
	Let $\phi$ be a polynomial in the indeterminates $(X_{ij})_{i,j \in [n]}$, with non-negative coefficients, of degree at most $1$ in each entry and total degree $d$. We write $\partial_{ij}$ for the partial derivative in the $(ij)$-entry. Let $M(\phi)$ be the maximal coefficient of $\phi$ and given a matrix $S=(S_{ij})_{i,j \in [n]}$, let $N(\phi, S)$ be the number of monomials in $\phi$ which are non-zero when evaluated at $S$. Finally consider $A_{\phi} := M(\phi) \max_{S} N(\phi,S)$ where the maximum is taken over all permutation matrices and $A_{\nabla \phi} := \max_{i,j \in [n]} A_{\partial_{ij} \phi}$. Then if $S$ is a uniform permutation matrix, for all $t \geq 0$,
	\begin{equation}
		\bP \sbra{\abs{\phi(S) - \bE \sbra{\phi(S)}} \geq t} \leq 2 \exp \left( \frac{- t^2}{2 \alpha_{\phi} (\frac{4}{3} \bE \sbra{\phi} + \frac{2^{d+1}(d-1)}{3n} A_{\phi}  + t)} \right).
	\end{equation}
	with
	\begin{equation*}
		\alpha_{\phi} := 6 \, d \, 2^{d} A_{\nabla \phi} \left( \log \left( \frac{4 A_{\phi} \, n}{A_{\nabla \phi}} \right)^{+} + \frac{(2/n)(2-e^{-2/n})}{1-e^{-2/n}}\right)
	\end{equation*}
\end{proposition}

Let us clarify a bit this result which is arguably a bit heavy in notation. First note any function on permutations can be represented as a polynomial in permutations matrices of degree at most $n-1$, although not uniquely. Furthermore, since permutation matrices have $0-1$ entries there is no loss of generality in supposing the degree is at most in each entry. The prohibitive term of the inequality are obviously the exponential terms in the degree, limiting the applications to small degrees that can nonetheless depend on $n$, and the term $A_{\nabla \phi}$. As the notation indicates, this term is essentially a bound on the norm of the gradient of the function. In the case $d = 1$, the random variable $\phi(S)$ can be written as $\tr(B^{\top}S)$ for some matrix $B$, in which case $B$ is exactly the gradient of $\phi$ and we have $A_{\nabla \phi} = \max_{i,j} B_{ij}$. All in all, Proposition \ref{prop:concentration_symmetric} provides a concentration inequality for smooth low-degree functions on the symmetric group. 

In our case, this result is applied conditional on $\cF_{r,l}$, so all the randomness comes from the intermediate path $\frp_2$, whose length will thus give the degree of the function. Then derivating the function at some entry $(i,j)$ comes down to focusing on the probability of nice paths which follow a prescribed long-range edge. The term $M(\partial_{ij} \phi)$ will thus correspond to the maximal probability of following one given nice path, which is $w_{\max}$ by \eqref{eq:nice_paths_bound}, while the boundedness of degrees allows to upper bound the number of nice paths which can use the long-rang edge $(i,j)$. We thus obtain the following bounds, which will be sufficient to apply the proposition.

\begin{lemma}[{\cite[Lemma 8.3]{cutoff_reversible}}]\label{lem:conditions_concentration}
	With the notations of Proposition \ref{prop:concentration_symmetric}, for all $r \leq M, l \geq 0$, conditional on $\cF_{r,l}$ $\scP_{\frN_{r,l}}^{t}(x,y)$ can be realized as a function $\phi$ on $S_{n'}$ of degree at most
	\begin{equation}\label{eq:degree_d}
		d := C_2 \sqrt{\log n'}.
	\end{equation}
	which satisfies
	\begin{equation*}
		\begin{gathered}
		\alpha_{\phi} = O \left( d^{3} (4 \Delta)^{d} w_{\max} \log n \right),\quad A_{\phi} = O(d^{2} \Delta^{d} n^{2/10} w_{\max}) \\
		\text{and} \quad A_{\nabla \phi} = O \left( d^{2}  (2 \Delta)^{d} w_{\max} \right).
		\end{gathered}
	\end{equation*}
\end{lemma}

The next step is to compute the expectation of nice paths. The following Lemma will be sufficient to deduce Proposition \ref{prop:nice_approx}. Recall the measure $\bfQ^{(L)}_u$ was defined in \eqref{eq:bfQu_G}. To ease notation we will drop the exponent $(L)$ in the sequel. Recall also that $T_1^{(G,l)}, T_1^{(\cG,l)}$ denote regeneration times with horizon $L$ in $G^{\ast}$ and $\cG$ respectively, $\mu$ is the invariant measure of the regeneration chain in the quasi-tree and $\nu$ was considered in Proposition \ref{prop:mixing_annealed_QT}. Below we write $\bfQ_{\nu + c} = \sum_{u} (\nu(u) + c(u)) \bfQ_{u}$.

\begin{lemma}\label{lem:mixing_annealed_G}
	There exists a measure $c$ on $V$ such that $\sum_{v \in V} c(v) = o_{\bP}(1)$ and
	\begin{equation*}
		\bE \cond{ \scP_{\frN_{r,l}}^{t}(x,y) }{\cF_{r,l}} =  \frac{(1-o_{\bP}(1))}{\bE_{\bQ_{\mu}} \sbra{T_{1}^{(\cG,\infty)}}} \bfQ_{\nu + c} \sbra{X_{r+s} = y, \abs{\xi(X_0 \cdots X_{r+s})} = l, r < T_1^{(G,L)} \leq M}.
	\end{equation*}
\end{lemma}

\begin{proof}[Proof of Lemma \ref{lem:mixing_annealed_G}]
	The proof is almost identical to that of Lemma 8.4 in \cite{cutoff_reversible} for the reversible case. It is essentially an application of Proposition \ref{prop:mixing_annealed_QT} combined with the coupling on the quasi-tree. 
	
	Let $\tau_{E'}$ be the hitting time of $E'$ and $U_{t'}(K)$ the time spent in $K$ before the time $t' \geq 0$. By definition, a nice path requires that $\tau_{E'} \leq t-s$ and the first part $\frp_1$ of the path is the trajectory until $\tau_{E'}$. The second part of a nice path is the trajectory until hitting $F'$ and is assured to spend less than $m := ML$ steps in $K$. Therefore by strong Markov's property, one can bound
	\begin{align}
		&\sum_{\frp \in \frN_{r,l}} \scP(\frp) \leq \sum_{\substack{t_1 + t_2 =  t -(r+s) \\ C_3 \sqrt{\log n} \leq t_2 \leq C_2 \sqrt{\log n}}} \sum_{u \in E'} \sum_{v \in F'} \left( \bfP_{x} \sbra{\tau_{E'} = t_1, X_{t_1} = u} \right. \label{eq:lem_upperbd}\\
		& \times \bfP \cond{\exists k \geq 0: T_{k}^{(G,L)} = t_2, X_{t_2} = v, U_{t_2}(K) \leq m}{X_{1/2} = u} \nonumber \\
		&\left. \bfP \cond{X_{r + s} = y, r < T_1 \leq M}{X_{1/2} = v, \tau_{\eta(v) > \tau_L}} \right) .\nonumber
	\end{align}
	In each term of this sum, the first and third factor are $\cF_{r,l}$-measurable, so only the second factor gets averaged when taking conditional expectation. We claim this expectation satisfies 
	\begin{equation*}
		\sum_{v \in F'} \abs{ \bP \cond{\exists k \geq 0: T^{(G,L)}_{k} = t_2, X_{t_2} = v, U_{t_2}(K) \leq m}{X_{1/2} = u, \cF_{r,l}} - \frac{\nu(v)}{\bE_{\bQ_{\mu}} \sbra{T_{1}^{(\cG,\infty)} }}} + o_{\bP}(1).
	\end{equation*}
	Note that $\nu / \bE_{\bQ_{\mu}} \sbra{T_{1}^{(\cG, \infty)}}$ is independent of $u$ and $t_1$, while the first factor in the sum considered above can be summed to at most $1$. We can also recognize the measure $\bfQ_v$ \eqref{eq:bfQu_G} in the third factor. Therefore provided the claim holds one obtains that for some $c = (c_v)_{v}$ satisfying $\sum_{v \in V} c(v)= o_{\bP}(1)$, 
	\begin{align*}
		&\bE \cond{\sum_{\frp \in \frN_{r,l}} \scP(\frp)}{\cF_{r,l}} \leq \sum_{v \in F'} \left( \frac{\nu(v)}{\bE_{\bQ_{\mu}} \sbra{T_{1}^{(\cG,\infty)}}} + c(v) \right) \bfQ_{v} \sbra{X_{r + s} = y, r < T_1 \leq M} \\
		&\qquad \leq \sum_{v \in V} \frac{\nu(v) + c(v)}{\bE_{\bQ_{\mu}} \sbra{T_{1}^{(\cG,\infty)}}} \bfQ_{v} \sbra{X_{r + s} = y, d_{\LR}(v,y) = l, r < T_1 \leq M},
	\end{align*}
	where in the second we implicitely changed the definition of $c$, using that $\bE_{\bQ_{\mu}} \sbra{T_{1}^{(\cG, \infty)}} = O(1)$. This proves an upper bound. To prove a matching lower bound, note that the inequality \eqref{eq:lem_upperbd} is not sharp if the trajectory is not in $\Gamma(R,L,M)$ or there is no regeneration time in the interval $[\tau_{l_1}, M L ]$. This was shown to occur with probability $o_{\bP}(1)$ in Section \ref{subsec:nice_typical}, so the upper bound is also a lower bound up to a $o_{\bP}(1)$ error.
	
	Let us prove the claim. Let $u \in E', v \in F'$. The idea is to couple the chains $X$ and $\cX$ on a quasi-tree to relate their regeneration times, in order to use Proposition \ref{prop:mixing_annealed_QT}. However the conditionning by $\cF_{r,l}$ already revealed some-long range edges, which requires in turn a similar conditionning in the quasi-tree. We argue that the only conditionning required is by $\BLR(O,L)$, the ball of radius $L$ in a quasi-tree. Let $u_0$ be the ancestor of $E$ at long-range distance $L$ from $u$. By definition of $K(x,l_1)$, the $L$ long-range neighbourhood $\BLR(u_0,L)$ contains no long-range cycle and thus is a possible realization of $\BLR(O, L)$ around the root of a quasi-tree. Consider a quasi-tree $\cG$ which has this ball as the neighbourhood of its root and is completed with the standard procedure (so without taking consideration of the long-range edges revealed in $K \cup B$). Using the coupling of Section \ref{subsec:coupling}, the chain $X$ started at $u$ can thus be coupled with the chain $\cX$ on $\cG$, started at the vertex of $\BLR(O,L)$ that identifies with $u$. This coupling fails after $X$ enters $B$ or if it re-enters $K$ by another path that the one it used. Let $\tcoup$ denote this new coupling time. 
	
	 From \eqref{eq:stages}, the probability of sampling an element of either $K$ or $B$ under the uniform measure on $V$ is $O(1/n^{\e'})$ for some $\e' > 0$. Consequently, using the same comparison with a binomial we used in the proof of Lemma \ref{lem:coupling_quasiT}, this implies that the chain $X$ re-enters $K$ or reaches $B$ in $O(\sqrt{\log n})$ steps is $O(\log n / n^{2 \e'}) = o(1)$. Note that considering regeneration times requires the knowledge of the $L$ steps ahead but this is $O(\log \log n)$. Consequently it remains true that for $t_2 + L = O(\sqrt{\log n})$,
	 \begin{multline*}
		\sum_{v \in F'} \left| \bP \cond{\exists k \geq 0: T_{k}^{(G, L)} = t_2, X_{t_2} = v, U_{t_2}(K) \leq m}{X_{1/2}= u, \cF_{r,l}} \right. \\ \left. - \, \bP \cond{\exists k \geq 0: T_{k}^{(\cG, L)} = t_2, \iota(\cX_{t_2}) = v, U_{t_2}(K) \leq m}{\cX_{1/2} = x, \BLR(O,L)} \right| = o(1)
	\end{multline*}
	where $x \in \BLR(O,L)$ is a vertex at long-range distance $L$ with type $u$.
	Obviously, $\{T_{k}^{(G, \infty)}, k \geq 1 \} \subset \{T_{k}^{(G, L)}, k \geq 1 \}$. This inclusion may be strict however, if the chain backtracks over a long-range distance $L$.  Lemma \ref{lem:typical_paths_QT} shows this occurs before time $O(\sqrt{\log n})$ with probability $o(1)$. Consequently, with high probability $T_{k}^{(\cG,L)} = T_{k}^{(\cG,\infty)}$ for all regeneration times that occur before $t_2$, so we can exchange these random times in the equation above. 
	Now since $t_2 = \Theta(\sqrt{\log n})$ and $m^{2} = O((\log \log n)^{2 \kappa + 2}) = o(t_2)$, Proposition \ref{prop:mixing_annealed_QT} proves that
	\begin{multline*}
		\sum_{v \in V} \left| \bP \cond{\exists k \geq 0: T_{k}^{(\cG, \infty)} = t_2, \iota(\cX_{t_2}) = v, U_{t_2}(K) \leq m}{\cX_{1/2} = x, \BLR(O,L)} - \frac{\nu(v)}{\bE_{\bQ_{\mu}} \sbra{T_{1}^{(\cG,\infty)}}} \right| \\
		\leq \bP \cond{U_{t_2}(K) > m}{\cX_{1/2} = x, \BLR(O,L)} / \bE_{\bQ_{\mu}} \sbra{T_{1}^{(\cG,\infty)}} + o(1).
	\end{multline*}
	By Lemma \ref{lem:regeneration_tails}, the first term is stretched exponential in $m$ and thus $o(1)$. Using triangle inequality to combine the two previous bounds yields the claim. 
\end{proof}

Now comes the last proof of this paper. 

\begin{proof}[Proof of Proposition \ref{prop:nice_approx}]
	Points (i)-(ii) were established proved in Section \ref{subsec:nice_typical} so we now prove (iii) - (iv).

	For (iii), suppose first $x, y \in V$ and $r,l$ are fixed. Let $\phi := \scP_{\frN_{r,l}}^{t}(x,y)$ as in Lemma \ref{lem:conditions_concentration} and $z:= \frac{\e}{2} \bE \sbra{\phi} + \frac{\e}{2 C_6 M n \sqrt{\log n}}$ for some $C_6 > 0$. 
	Note that $w_{\max} \leq e^{-C \sqrt{\log n}} / n$ for some constant $C$ which tends to $+ \infty$ as the constant $C_0$ in the definition of $t$ grows, while other factors of $\alpha_{\phi}, A_{\nabla \phi}$ are of all of order at most $e^{C' \sqrt{\log n}}$. Thus for any choice of $C = C(\e) > 0$ Lemma \ref{lem:conditions_concentration} shows that $\alpha_{\phi}$, $A_{\nabla \phi}$ can both be bounded by $e^{-C \sqrt{\log n}} / n$, provided the constant $C_0$ is sufficiently large, while $d^{2} (d-1) A_{\phi}/ n = o(n^{-3/5})$. In particular we can choose $C$ so that $A_{\nabla \phi} \leq z$. Since $\bE \sbra{\phi} \leq 2 z / \e$ and $z \geq \e / (2 C_6 M n \sqrt{\log n})$, applying Proposition \ref{prop:concentration_symmetric} yields 
	\begin{equation}\label{eq:concentration_nice}
		\bP \sbra{\abs{\phi - \bE \sbra{\phi}} \geq z} \leq 2 \exp \left( \frac{- C' \e^{2}}{\alpha_{\phi} M n \sqrt{\log n} } \right)
	\end{equation}
	for some $C' = C'(\e) > 0$. Up to increasing again the value of $C$, we can ensure that $\alpha_{\phi} M n \sqrt{\log n} \leq (\log n)^{-2}$, hence
	\begin{equation*}
		\bP \sbra{\abs{\phi - \bE \sbra{\phi}} > \frac{\e}{2} \bE \sbra{\phi} + \frac{\e}{2 C_6 M n \sqrt{\log n}} } \leq 2 \exp \left(- C' \e^{2} (\log n)^{2} \right).
	\end{equation*}
	This is sufficent to take a union bound over $x,y \in V$ and $r \leq M, l \in [\mathscr{d} s - C_5 \sqrt{\log n}, \mathscr{d} s + C_5 \sqrt{\log n}]$. Thus summing over $r,l$ we obtain that with high probability, for all $x,y \in V$,
	\begin{equation*}
		\abs{\scP_{\frN}^{t}(x,y) - \sum_{r,l} \bE \cond{\scP_{\frN_{r,l}}^{t}(x,y)}{\cF_{r,l}}} \leq \frac{\e}{2} \sum_{r,l} \bE \cond{\scP_{\frN_{r,l}}^{t}(x,y)}{\cF_{r,l}} + \frac{\e}{n}.
	\end{equation*}
	Lemma \ref{lem:mixing_annealed_G} gives an estimate of the conditional expectation which shows:
	\begin{align}
		\scP_{\frN}^{t}(x,y) &\leq \frac{1 + \e / 2}{\bE_{\bQ_{\mu}} \sbra{T_{1}^{(\cG,\infty)}}} \sum_{r = 0}^{M} \bfQ_{\nu + c} \sbra{X_{r+s} = y, r < T_1^{(G,L)} \leq M} + \frac{\e}{n} \nonumber \\
		&= (1 +\e /2) A \hat{\pi}(y) + \frac{(1 +\e /2 )A}{\bfE_{\bfQ_{\nu}} \sbra{T_{1}^{(G,L)}}} \sum_{r = 0}^{M} \bfQ_{c} \sbra{X_{r+s} = y, r < T_1^{(G,L)} \leq M} +\frac{\e}{n}. \label{eq:upper_bound_nice}
	\end{align}
	with $A:= \frac{\bfE_{\bfQ_{\nu}} \sbra{T_{1}^{(G, L)}}}{\bE_{\bQ_{\mu}} \sbra{T_{1}^{(\cG, \infty)}}}$. Summing over $y \in V$ and $r \in [0, M]$ in the second term yields $(1 + \e /2) A \sum_{v \in V} c(v) \leq 2 A \e$ with high probability. On the other hand, from Lemma \ref{lem:mixing_annealed_G} we also have the lower bound
	\begin{equation*}
		\scP_{\frN}^{t}(x,y) \geq (1 - \e / 2)(1-o_{\bP}(1)) A \hat{\pi}(y) - \e / n,
	\end{equation*}
	which by summing over $y \in V$ shows that $A \leq 1 + 10 \e$ with high probability. Plugging this in the upper bound above yields the second statement of the proposition. 	

	Let us move to the proof of (iv). Define
	\begin{equation*}
		\hat{\pi}_{1}(v) := \frac{1}{\bfE_{\bfQ_{\nu}} \sbra{T_{1}^{(G, L)}}} \sum_{r=0}^{M} \sum_{u \in S} \nu(u) \bfQ_{u} \sbra{X_{r+s} = v, r < T_1 \leq M, \TSR \wedge \TNB > M+s}
	\end{equation*}
	for all $v \in V$ and decompose $\hat{\pi} =: \hat{\pi}_1 + \hat{\pi}_2$.

	We start bounding $\hat{\pi}_2(V)$. Summing over $y$ in \eqref{eq:upper_bound_nice}, the left hand side becomes the total probability of following a nice path, which by the first part of the proposition is at least $1 - \e$ with high probability for a typical $x \in V$. Consequently $A$ can be lower bounded by a constant with high probability, which implies that for some constant $c > 0$, $\bfE_{\bfQ_{\nu}} \sbra{T_{1}^{(G, L)}} \geq c$ with high probability. Thus by summing over $v \in V$ we can bound
	\begin{equation*}
		\hat{\pi}_2(V) \leq c^{-1} M \bfQ_{\nu} \sbra{\TSR \wedge \TNB \leq M+s} + o_{\bP}(1).
	\end{equation*}
	Then notice that the quenched probability $\bfP \cond{\tau_{\eta(u)} > \tau_L}{X_{1/2} = u}$ can always be lower bounded by $q_n := \delta^{C L}$ for some constant $C > 0$, hence 
	\begin{align*}
		\hat{\pi}_2(V) &\leq c^{-1} M \sum_{u \in V} \nu(u) \frac{\bfP \cond{\TSR \wedge \TNB \leq M+s, \tau_{\eta(u)} > \tau_L}{X_{1/2} = u}}{\bfP \sbra{\tau_{\eta(u)} > \tau_L, X_{1/2} = u}} + o_{\bP}(1) \\
		&\leq c^{-1} q_n^{-1} M \sum_{u \in V} \nu(u) \bfP \cond{\TSR \wedge \TNB \leq M+s}{X_{1/2} = u} + o_{\bP}(1).
	\end{align*}
	As $L = O(\log \log n)$ one has $q_n^{-1} = O ((\log n)^{-b})$ for some $b > 0$. Next by Markov's property for all $u \in V$
	\begin{align*}
		\bfP \cond{\TSR \wedge \TNB \leq M+s}{X_{1/2} = u} &= \bfP \cond{\TSR \wedge \TNB \leq M+s}{X_0 = u, X_{1/2} =  u} \\
		&= \frac{\bfP_u \sbra{\TSR \wedge \TNB \leq M+s, X_{1/2} = u}}{p(u,\eta(u))} \\
		&\leq c' \bfP_u \sbra{\TSR \wedge \TNB \leq M+s}
	\end{align*}
	for some constant $c' > 0$ as the entries of $p$ are bounded (Assumption \ref{hyp:bdd_p}). For fixed $u \in V$, Lemma \ref{lem:typical_paths} shows $\bfP_u \sbra{\TSR \wedge \TNB \leq M+s} = o_{\bP}((\log n)^{-b}) = o_{\bP}(q_n^{-1})$ as $M+s = O(\log n)$. Letting $U \in V$ be a random variable independent from the environment, this can be rephrased as the fact that $\bfP_U \sbra{\TSR \wedge \TNB \leq M+s} = o_{\bP}(q_n^{-1})$ conditional on $U$, but then this statement must also hold unconditionally. As the measure $\nu$ is deterministic we can take $U$ of law $\nu$, after which by taking the conditional expectation with respect to the environment we get that
	\begin{equation*}
		\bE \cond{\bfP_U \sbra{\TSR \wedge \TNB \leq M+s}}{\eta} = \sum_{u \in V} \nu(u) \bfP_u \sbra{\TSR \wedge \TNB \leq M+s} =o_{\bP}(q_{n}^{-1}).
	\end{equation*} 
	All in all, this proves $\hat{\pi}_2(V) = o_{\bP}(1)$.
	
	To bound the $\ell^2$ norm of $\hat{\pi}_1$, we can use what precedes to first bound
	\begin{align*}
		\hat{\pi}_1 &\leq c^{-1} M \max_{r \leq M} \bfQ_{\nu} \sbra{X_{r+s} = \cdot, \TSR \wedge \TNB > s} \\
		&\leq c^{-1} q_n^{-1} M \max_{r \leq M} \sum_{u \in V} \nu(u) \bfP \cond{X_{r+s} = \cdot, \TSR \wedge \TNB > M+s}{X_{1/2} = u}.
	\end{align*}
	Consequently
	\begin{equation*}
		\bE \sbra{\sum_{v \in V} \tilde{\pi}_{1}(v)^{2}} \leq c^{-2} q_n^{-2} M^2 \, \max_{r \leq M} \bP \sbra{ \begin{array}{c} X_{r+s} = X'_{r+s}, \TSR \wedge \TNB > s, \\ \tau_{\eta(X_{1/2})} > \tau_L, \tau_{\eta(X'_{1/2})} > \tau_{L} \end{array}}
	\end{equation*}
	where $X, X'$ are two independent chains conditional on the same environment, both started at time $1/2$ under the law $\nu$. Here the event $\TSR \wedge \TNB > s$ applies to both chains. Fix $r \leq M$.	Observe that if all long-range edges crossed by $X'$ are at distance more than $R$ from the trajectory followed by the first chain $X$, then $X'$ can reach $X_{r+s}$ only by deviating by a small-range distance $R$ or by crossing the long-range edge $(X'_{1/2}, \eta(X'_{1/2}))$ if $\eta(X'_{1/2}) \in \BSR^{-}(X_{1/2} \ldots X_{r+s},R)$. The event $\{ \tau_{\eta(X'_{1/2})} > \tau_{L} \}$ proscribes this last possibility to occur before the chain has crossed a long-range distance $L$, so the chain must backtrack. Thus either of these two possibilities contradicts $\TSR \wedge \TNB > s$. Consequently the only possibility that the two chains meet is that the second one crosses a long-range edge at small-range distance less than $R$ from the first chain. From \eqref{eq:nu_uniform} $d_{\SR}(X_{1/2}, X'_{1/2}) \leq R$ with probability $O(\Delta^{R} / n ) = o(1)$ under $\nu$. Then assuming this does not hold, under the annealed law the environment can be generated sequentially along the trajectories of the chains as explained in Section \ref{subsec:coupling}. The number of long-range edges met by $X'$ up to time $M+s$ which fall in a $R$ small-range neighbourhood of $X$ is thus stochastically dominated by a binomial $\mathrm{Bin}(M+s, \Delta^{R} (M+s) / n)$. As $q_n = O((\log)^{b}), R = O(\log \log n)$, $s = O(\log n), M = o(\log n)$ we deduce that 
	\begin{equation*}
		\bE \sbra{\sum_{v \in V} \hat{\pi}_{2}(v)} \leq O \left( \frac{C (\log n)^{b'}}{n} \right)
	\end{equation*}
	for some $b' > 0$, which gives the result by the first moment argument \eqref{eq:markov_quenched}.
\end{proof}

\bibliographystyle{plainurl}
\bibliography{biblio.bib}

\end{document}